\documentclass [a4paper, 12pt]{article}
\usepackage[english]{babel}
\usepackage{color}
\usepackage{subcaption}
\usepackage{babel}
\usepackage{fontenc}
\usepackage{graphicx}
\usepackage{epstopdf}
\usepackage{amsmath,amsthm,amssymb}
\usepackage[colorlinks=true,citecolor=blue, linkcolor=blue, urlcolor=blue]{hyperref}
\usepackage{float}
\usepackage{lineno}
\newtheorem{theorem}{Theorem}

\newtheorem{lemma}[theorem]{Lemma}

\usepackage[percent]{overpic}
\usepackage{multicol}
\usepackage{pgfplots}
\usepackage{soul}
\pgfplotsset{compat=1.12}
\newcommand{\sn}{\mbox{\tiny SN}}
\newcommand{\tc}{\mbox{\tiny TC}}
\usepackage[mathscr]{euscript}
\DeclareSymbolFont{rsfs}{U}{rsfs}{m}{n}
\DeclareSymbolFontAlphabet{\mathscrsfs}{rsfs}
\usepackage[title]{appendix}
\newtheorem{prop}{Proposition}
\usepackage{xcolor}
\definecolor{mg}{RGB}{255,0,255}
\theoremstyle{remark}
\newtheorem{remark}{Remark}
\usepackage{color,soul}
\usepackage[normalem]{ulem}

\begin{document}

\title{Analytical detection of stationary and dynamic patterns in a  prey-predator model with reproductive Allee effect in prey growth} \vspace{-1em}
\author{Subrata Dey, S Ghorai,  Malay Banerjee\thanks{Corresponding author} }
\maketitle
\vspace{-2em}
{\center Department of Mathematics and Statistics, Indian Institute of Technology Kanpur, Kanpur 208016, India \\}

{\center
E-mails:  subratad@iitk.ac.in, sghorai@iitk.ac.in, malayb@iitk.ac.in\\[2mm]  }

\begin{abstract}
Allee effect in population dynamics has a major impact  in suppressing the paradox of enrichment through global bifurcation,  and it can generate highly complex dynamics. The influence of  the reproductive Allee effect, incorporated in the prey's growth rate  of a prey-predator model with Beddington-DeAngelis functional response, is investigated here. Preliminary local and global bifurcations  are identified 
of the temporal model.  Existence and non-existence  of heterogeneous  steady-state solutions of the spatio-temporal system are established for suitable ranges of parameter values. The spatio-temporal model satisfies Turing instability conditions, but numerical investigation reveals that the heterogeneous patterns corresponding to unstable  Turing eigen modes acts as a transitory pattern. Inclusion of the  reproductive Allee effect in the prey population has a destabilising effect on the coexistence equilibrium. For a range of parameter values, various branches of stationary solutions including mode-dependent Turing  solutions and localized pattern solutions  are identified using numerical bifurcation technique. The model is also capable to produce some complex dynamic patterns such as travelling wave, moving pulse solution, and spatio-temporal chaos for certain range of parameters and diffusivity along with appropriate choice of  initial conditions Judicious choices of parametrization for the Beddington-DeAngelis functional response help us to infer about the resulting patterns for similar prey-predator models with Holling type-II functional response and ratio-dependent functional response.
    
\end{abstract}

\section{Introduction}\label{sec1}

Understanding self-organized spatio-temporal patterns and mechanisms of  spatial dispersal in interacting populations have always been  important areas of mathematical biology and theoretical ecology. Since the pioneering work of Lotka \cite{LotkaUNDAMPEDOD}  and Volterra \cite{Volterra}, a wide variety of mathematical models  have been developed to investigate various complex ecological phenomena, e.g., Allee effect \cite{stephens1999allee}, group defense \cite{venturino2013spatiotemporal}, parental care \cite{wang2006prey}, hunting cooperation \cite{alves2017hunting}, intra-guild predation \cite{kang2013dynamics},  etc. 
In addition, persistence, stability of   steady-state, local and global bifurcation, etc., of these models play significant roles in dynamical system theory. The study of such phenomena under spatial distribution makes it more interesting and realistic in spatial ecology. The spatio-temporal distribution of species has been extensively studied  to explain pattern formation observed in fish skin \cite{fish},  prey-predator interactions \cite{prey}, terrestrial vegetation \cite{vegetation}, mussel bed \cite{cangelosi2015nonlinear}, etc.

The classical explanation of spatial pattern formation comes from Turing's pioneering work in  reaction-diffusion (RD) system for chemical morphogenesis \cite{Turing}. The framework of the diffusion-driven instability in Turing's work needs two basic criteria: an activator-inhibitor interaction structure of at least two interacting species, and an order of magnitude difference in their dispersal coefficient. Segel and Jackson \cite{Segal} first used Turing's model framework to explain the spatial pattern formation for ecological systems. Turing instability produces various stationary patterns such as spots, stripes, and a mixture of both \cite{cross1993pattern,turing1}. For one-dimensional habitats, Turing patterns lead to stationary spatially periodic patterns \cite{hillen1996turing,dey2022analytical}. As parameters are varied towards certain thresholds, the Turing pattern solution becomes a localized stationary  pattern solution  due to another instability, the so-called  Belyakov-Devaney (BD) transition \cite{belyakov1997abundance,champneys1998homoclinic}.  The bounds and the region of existence of the localized pattern solution, and the homoclinic snaking mechanism \cite{woods1999heteroclinic,burke2007snakes} of the patterns   can be described by semi-strong asymptotic analysis \cite{al2022spikes}. These spatial instabilities depend on the reaction kinetics of the system.  For example,  Rosenzweig-MacArthur(RM) prey-predator system, one of the most popular models in ecology, does not follow Turing's activator-inhibitor structure and it does not produce any stationary patterns \cite{petrovskii1999minimal,banerjee2017spatio}.  Some  additional nonlinear effects such as predator-dependent functional response, generalist type predator, nonlinear death rate of the predator, Allee effect, etc., \cite{turing1,mimura1978diffusive,dey2022analytical,wang2013dynamics} are required to produce a stationary Turing pattern.  

The Allee effect, introduced by renowned ecologist W. C. Allee \cite{Alleebook},  has important consequences on population dynamics and their persistence. It refers to the  correlation between the population density and per capita growth rate of the population at low densities \cite{stephens1999allee}. It plays an important role in ecological conservation and wildlife management since it is closely related to population extinction.  Generally speaking, there are two kinds of  Allee effects \cite{courchamp2008allee}: component Allee effect and demographic Allee effect. Component Allee  effect focuses on components of individual fitness,  whereas demographic Allee effect is at the level of the overall mean individual fitness of the whole population. Allee effect can be induced by a variety of factors that include  difficulties in  social interaction,  cooperative breeding, reproductive facilitation, anti-predator behavior, mate finding, and environmental conditioning, among others \cite{courchamp2008allee,boukal2002single,gascoigne2004allee}. The demographic Allee effect is a consequence of one or more  components Allee effect. The reproductive Allee effect is due to many different mechanisms.  Due to the difficulty in finding a suitable mate during their receptive period, individuals in small or sparse populations face  the mate-finding Allee effect and have less reproductive success.  Mate-finding Allee effect is observed in many species such as  Glanville fritillary butterfly \cite{kuussaari1998allee},  sheep ticks \cite{rohlf1969effect},  polar bears \cite{molnar2008modelling}, queen conch \cite{stoner2012negative},  and many others [see \cite{courchamp2008allee,kramer2009evidence} for further examples].  Other mechanisms of the reproductive Allee effect include broadcast spawning in corals, pollen limitation in field gentian, reproductive facilitation in whiptail lizard, and sperm  limitation in blue crab, among others \cite{courchamp2008allee,kramer2009evidence}. This component or demographic Allee effect can be strong or weak.  If there is a threshold population density, called the Allee threshold,  below which the per capita growth rate of the population is negative, then it is called  the strong Allee effect.  For the weak Allee effect, there is no such Allee threshold. If the population density  rises above this threshold, it will persist, whereas  it will become extinct if it falls below \cite{berec2007multiple,courchamp2008allee}.

Let $N(T)$ be the  density of a single species population (prey) at time $T$. The single species logistic growth is described by the equation $\displaystyle{ \frac{dN}{dT}=rN\left(1-\frac{N}{K}\right)}$; $r$ and $K$ respectively denote the intrinsic growth rate and carrying capacity of the population. This description has a limitation in that the per capita growth rate is maximum at low population density. Single species growth with the Allee effect assumes that the per capita growth rate is an increasing function of the population at low population density. There exist several parametrizations of the growth of a single species population subject to the Allee effect. In terms of mathematical formulations, those can be broadly classified into two types: additive Allee effect and  multiplicative Allee effect. In the case of additive Allee effect, the logistic growth equation is modified as \cite{stephens1999allee,aguirre2009three}
$$\frac{dN}{dT}\,=\,rN\left(1-\frac{N}{K}-\frac{s}{N+\phi}\right),$$
where $s$ and $\phi$ indicate the severity of the additive Allee effect. On the other hand, some well-known mathematical formulations for the multiplicative Allee effect include the parametrization $\displaystyle{rN\left(1-\frac{N}{K}\right)\left(\frac{N}{K}-\frac{\theta}{K}\right)}$ \cite{lewis1993allee}, $\displaystyle{rN\left(1-\frac{N}{K}\right)\frac{N}{N+\theta}}$ \cite{dennis1989allee}, and $\displaystyle{rN\left(1-\frac{N}{K}\right)\left(1-\frac{N+C}{N+\theta}\right)}$  \cite{boukal2002single,courchamp2008allee}. In these parametrizations, $\theta$ is the strength of the Allee effect, and  $C$ characterizes the shape of the per capita growth rate curve.  The most widely used form of the multiple Allee effect is 
\begin{equation}
\frac{dN}{dT}=rN\left(1-\frac{N}{K}\right)\left(\frac{N}{K}-\frac{\theta}{K}\right),
\label{sAllee}
\end{equation}
where $0<\theta<K$ corresponds to strong Allee effect and $-K<\theta<0$ stands for weak Allee effect. This formulation has a drawback as the intra-specific cooperation and competition appear jointly in the formulation in a multiplicative way.
To overcome this situation,  Petrovskii et al. \cite{petrovskii2008consequences} proposed a different parametrization to model the single species population growth subject to Allee effect as follows:
\begin{equation}
\frac{dN}{dT}=N\big(B(N)-c-I(N)\big),
\label{Allee1}
\end{equation}
where $B(N)$ and $I(N)$ represent  population multiplication due to reproduction and  an additional density-dependent mortality rate due to  intra-species competition respectively and $c$ denotes the intrinsic mortality rate. Taking the parametrization of  sexual reproduction of the population and density-dependent mortality rate in the form $B(N)=abN$ and $I(N)=aN^2$ (following the approach in \cite{jankovic2014time, mukherjee2021bifurcation}), we can write the single species growth equation as
\begin{equation}
\frac{dN}{dT}=aN^2(b-N)-cN,
\label{Allee2}
\end{equation}  
where the positive constants $a$ and $b$ can be interpreted as intrinsic growth rate and threshold for positive growth rate respectively. This description can be considered as the growth equation of a two-sex population model \cite{terry2015predator}. To illustrate this claim, let $N_m$ and $N_f$ denote the densities of male and female populations of a species ($N_m+N_f=N$), and let $R$ be the available resources. Then the reproduction rate is proportional to $N_mN_fR,$ and the available resources $R$ can be expressed as 
$R=K-S(N_m+N_f)$,
where $K$ is the carrying capacity and $S$ denotes the  per capita resource consumption rate. Assuming an equal sex ratio, i.e., $N_m=N_f=N/2$,  we find the growth rate due to reproduction to be proportional to $N^2(K-SN)/4$.  Therefore, the growth equation can be written as
\begin{equation}\label{Bz2}
    \frac{dN}{dT}= R_1 N^2(1-N/K_1)-c N,
\end{equation}
where $R_1$ and $K_1$ are positive parameters. We recover equation  \eqref{Allee2} from \eqref{Bz2} by taking  $R_1/K_1=a$ and $K_1=b$. 

The per capita growth rate, $G(N)=aN(b-N)-c$, in  \eqref{Allee2} is negative  below the threshold  $\displaystyle{N_1=\frac{b}{2}-\sqrt{\frac{b^2}{4}-\frac{c}{a}}}$ and above the threshold  $\displaystyle{N_2=\frac{b}{2}+\sqrt{\frac{b^2}{4}-\frac{c}{a}}}$, when $ab^2>4c$. Thus, equation \eqref{Allee2} can be compared with the single species population growth with strong Allee effect \eqref{sAllee} with the revised parametrization 
$r=aN_2^2$, $\theta=N_1$, and $K=N_2$. In summary, we can consider equation \eqref{Allee2} as the description of sex-structured single species population growth which implicitly includes the mate-finding Allee effect with the Allee threshold $N_1$. In this work, we consider equation \eqref{Allee2} to describe the growth rate of prey in the absence of specialist predator.

Predator-independent functional responses are widely used in prey-predator systems, which  is not always realistic in ecology \cite{skalski2001functional}. If predator density $P$ increases, then the per capita growth rate of predators may decrease. This phenomenon is called  predator interference \cite{pvribylova2015predator} that can be incorporated into system dynamics through a predator-dependent functional response.  A well-known model incorporating predator interference comes from  the RM system by replacing Holling type-II functional response with Beddington–DeAngelis functional response 
\begin{equation}\label{Beddington}
F(N,P)=\frac{m N}{p+N+q P},
\end{equation}
where parameters $m, p$ and $q$ denote maximum predation rate, self-saturation constant and predator mutual interference respectively  \cite{beddington1975mutual,alonso2002mutual}. Inclusion of mutual interference in predators induces Turing's activator-inhibitor structure in the system dynamics \cite{alonso2002mutual}. Depending on the parameter values  $p$ and $q$, the Beddington–DeAngelis functional response can be converted into a ratio-dependent functional response or predator-independent  Holling type-II  functional response. 
Here, we incorporate prey-predator interaction through the Beddington–DeAngelis functional response and compare the system dynamics with the ratio-dependent functional response and  Holling type-II  functional response.

Homogeneous and non-homogeneous stationary solutions of  prey and predator distributions correspond to one of the possible feeding strategies of predators within a permanent habitat and hunting in a nearby area. To avoid predators, some prey species migrate to nearby places \cite{skov2013migration,fryxell2012individual}.  To be successful in capturing  prey, predator species need to follow the prey species with some strategy which leads to the formation of various complex dynamic patterns that include travelling wave, spatio-temporal chaos, and moving pulse solution \cite{banerjee2016prey}. Predator invasion  into the prey habitat leads to the formation of a travelling wave that connects a predator-free steady state  to a coexisting steady state \cite{dunbar1983travelling,banerjee2016prey}.  The existence and non-existence of travelling wave solutions and invasion speed are also interesting topics in dynamical system theory \cite{huang2016geometric,ai2017traveling}. The impact of spatio-temporal chaos on population dynamics is an active area of research in ecology. Appearance of spatio-temporal chaos is widespread and it has also been observed in systems that do not produce  stationary pattern \cite{petrovskii1999minimal}. The stationary and moving pulse solutions are studied in a  prey-predator model following  the prey growth rate equation \eqref{Allee2}  and  Holling type-II functional response \cite{banerjee2020prey}.

Here, we  consider a prey-predator model with prey's growth rate as given in \eqref{Allee2}  and Beddington-DeAngelis functional  response \eqref{Beddington} in the prey-predator interaction. First, we perform a bifurcation analysis of the temporal model and illustrate some representative dynamics through bifurcation diagrams. The global existence of the solutions of the diffusive system  and their global asymptotic behavior are explored in various scenarios. Using energy estimates, we obtain a priori bounds of global steady-state solutions and identify the range of diffusion parameter for the non-existence of spatial patterns. The temporal model exhibits bistable dynamics but the corresponding spatio-temporal model possesses many non-constant spatio-temporal patterns that include localized patterns, various Turing mode patterns, and biological invasion, among others.  We use stability analysis and Leray–Schauder degree theory to show the existence of such non-constant steady states. A variety of stable and unstable non-constant stationary solutions are obtained through numerical continuation. 
Using numerical continuation, we illustrate  the bifurcation scenario of the spatial patterns that include Turing and localized patterns. We also investigate the parametric region for the existence and non-existence of the travelling wave solution and its profile inside the wedge-shaped region of invasion. 

The contents of this paper are as follows. Section \ref{temporalmodel} contains the description of the temporal model along with its equilibria and their bifurcations.  In section \ref{spatialmodel}, we extend it to spatio-temporal model  with no-flux boundary condition and investigate the existence and non-existence of constant and non-constant solutions with their prior bounds.   Using extensive numerical continuation, we present various stationary mode-specific Turing pattern solutions and localized pattern solutions together with their bifurcations in section  \ref{SDNC}. In section \ref{dynamic}, we discuss some complex dynamic
patterns that include travelling wave, moving pulse solution, and spatio-temporal chaos.  Finally, conclusions and discussions are presented in section \ref{discussion}.

\section{Temporal Model}\label{temporalmodel}
Let $N(T)$ and $P(T)$ respectively be the prey and predator densities at time $T$. 
Suppose that the growth rate of the prey population  is subjected to  reproductive Allee effect \eqref{Allee2} and Beddington–DeAngelis functional  response \eqref{Beddington} is chosen to represent the prey-predator interaction. Then, the governing prey-predator system, subject to non-negative initial conditions, is 
\begin{subequations}
\begin{alignat}{4}
 \frac{dN}{dT}&=aN^2(b-N)-cN-\frac{mNP}{p+N+qP},\\
 \frac{dP}{dT}&=\frac{emNP}{p+N+qP}-dP,
\end{alignat}
\label{2}
\end{subequations}
where $d$ and $e$ respectively denote the per capita natural death rate of the predator population and the conversation coefficient. All other parameters have been described before.  Introducing dimensionless variables $u=\frac{N}{b}$, $ v=\frac{m P}{bd}$ and $ t=d T$,  we obtain  dimensionless version of (\ref{2}):
\begin{subequations}
\begin{alignat}{4}
 \frac{du}{dt}&=\big(\sigma u^2(1-u) -\eta u\big)-\frac{uv}{\alpha+u+\beta v} \equiv F_1(u,v), \\
 \frac{dv}{dt}&=\frac{\gamma uv}{\alpha+u+\beta v}-v\equiv F_2(u,v),
\end{alignat}
\label{ode}
\end{subequations}
where $ \alpha=\frac{p}{b},\; \beta=\frac{qd}{m},\; \gamma=\frac{em}{d},$ $\sigma=\frac{ab}{d}$ and  $\eta=\frac{c}{d}$ are dimensionless parameters.

\subsection{Existence and stability of equilibria}
The system (\ref{ode}) has trivial equilibrium point $E_0(0,0)$ irrespective of parameter values. It has none, one or two axial equilibria depending on parameter values. We denote the axial equilibria (whenever they exist) by $E_j(u_j,0)$, $j=1,2$. The $u$ components of the axial equilibria  are roots of the equation
$$
P(u)\equiv\sigma u(1-u) -\eta=0.
$$
For $\sigma>4\eta$, the system (\ref{ode}) has two axial equilibria  $E_{1}(u_1,0)$ and $E_{2}(u_2,0)$,  where 
\begin{equation}\label{axialexpre}
    u_1=\frac{\sigma+ \sqrt{\sigma^2-4\sigma \eta}}{2\sigma} \text{ and } u_2=\frac{\sigma- \sqrt{\sigma^2-4\sigma \eta}}{2\sigma}.
\end{equation}
On the other hand, it has only one axial equilibrium point $E_1(1/2,0)$ for $\sigma=4\eta$  and  no axial equilibrium point for $\sigma<4\eta$. 

A interior equilibrium  $E_*(u_*,v_*)$ is a point of intersection of the nontrivial prey and predator nullclines $f_1(u,v)=0$ and $f_2(u,v)=0$, where 
\begin{equation}
    f_1(u,v)=\sigma u(1-u) -\eta -\frac{v}{\alpha+u+\beta v} \text{ and }
    f_2(u,v)=\frac{\gamma u}{\alpha+u+\beta v}-1.
    \label{predator nullcline}
\end{equation}
The  component $u_*$ of the interior equilibrium  $E_*$ is a root of the cubic equation 
$$
Q(u)\equiv \sigma \gamma \beta{u}^{3}-\sigma \gamma \beta {u}^{2}+ \left( 
\beta \eta \gamma+\gamma-1 \right) u-\alpha=0.
$$  
From (\ref{predator nullcline}), we get
$$
v_*={\frac {\gamma u_*-\alpha-u_*}{\beta}}.
$$
We must have $u_*>\frac{\alpha}{\gamma-1}\geq 0$ for the feasibility of $E_*$.
The system has only the trivial equilibrium point for $\sigma<4\eta$ since $f_1(u,v)<0$ in the positive quadrant of $u$-$v$ plane in this case. For $\sigma>4\eta$,  the coexisting equilibria $E_*$ is feasible if $u_*$ satisfies   $$\text{max}\left\{\frac{\alpha}{\gamma-1},u_2\right\}<u_*<u_1.$$
The Jacobian matrix of the system (\ref{ode}) evaluated at an equilibrium point $E(u,v)$ is
\begin{equation}
J(E)=\begin{bmatrix}
  \sigma u \left( 2-3u \right) -\eta-{\frac {v} {\alpha+u+\beta v}}+{\frac {uv}{ \left(  \alpha+u+\beta v \right) ^{2}}} &  -{\frac {u \left( \alpha+u \right) }{
 \left(  \alpha+u+\beta v \right) ^{2}}} \vspace{0.06in} \\
 {\frac {
\gamma v \left( \alpha+\beta v \right) }{ \left(  \alpha+u+\beta v
 \right) ^{2}}} &  {\frac {\gamma u} {\alpha+u+\beta v}}-{\frac {\gamma uv\beta}{ \left(  \alpha+u+\beta v
 \right) ^{2}}}-1
\end{bmatrix}. 
\label{jaco}
\end{equation}
Clearly, trivial equilibrium point $E_0(0,0)$ is asymptotically stable since both the eigenvalues of $J(E_0)$ are negative. The eigenvalue of  Jacobian matrix evaluated at an axial equilibrium point $E_j(u_j,0)$ ($j=1$ or $2$) are $\sigma u_j(1-2  {u_j})$ and $-1+{\frac {\gamma u_j}{\alpha+u_j}}.$ 
Hence, $E_j$ is asymptotically stable 
if $\frac{1}{2}<u_j< \frac{\alpha}{\gamma-1}$, unstable if $ \frac{\alpha}{\gamma-1}<u_j<\frac{1}{2}$, and a saddle point if $u_j>\text{max}\{\frac{\alpha}{\gamma-1},\frac{1}{2}\}$ or $u_j<\text{min}\{\frac{\alpha}{\gamma-1},\frac{1}{2}\}.$

 The Jacobian matrix evaluated at a coexisting equilibrium point $E_*(u_*,v_*)$  is given by 
 \begin{equation}
J(E_*)= \begin{bmatrix}
-u_* \left( 2 u_*-1 \right) \sigma+{\frac {v_*}{ \left( \beta v_*+\alpha+u_*
 \right) ^{2}}}
 & -{\frac {u_* \left( \alpha+u_* \right) }{
 \left( \beta v_*+\alpha+u_* \right) ^{2}}}
 \vspace{0.06in} \\{\frac {
\gamma v_* \left( \beta v_*+\alpha \right) }{ \left( \beta v_*+\alpha+u_*
 \right) ^{2}}}
 &  -{\frac {\beta v_*}{\beta v_*+\alpha+u_*}}
\end{bmatrix} \equiv \begin{bmatrix}
a_{10}  & a_{01}  \vspace{0.06in} \\
b_{10} &  b_{01}
\end{bmatrix}.
\label{jacoh}
\end{equation}
Using Routh-Hurwitz stability criteria \cite{Perko}, $E_*$ is asymptotically stable if $$\mathrm{T}:=\text{trace}(J(E_*))<0 \text{ and } \mathrm{D}:=\text{det}((J(E_*))>0.$$
Due to unavailability of explicit expression of a interior equilibrium point $E_*$, it is difficult to determine the analytic condition  for the stability of  coexisting equilibrium point. However,  we discuss the existence and stability of of $E_*$ through various bifurcations that occur due to the variation of control parameters.
\subsection{Local bifurcation analysis }
Here we discuss some preliminary  bifurcation results that are required to study of spatio-temporal pattern formation. In this subsection, we follow the same notations as given in \cite{Perko}.
\subsubsection{Saddle-node bifurcation}
\begin{theorem}
For $\gamma>2\alpha+1$, the temporal system (\ref{ode}) undergoes a saddle-node bifurcation   when the quadratic polynomial $P(u)$   has a double root.
 \end{theorem}
 \begin{proof}
Suppose that $P(u)$  has a double root $u=u_s$ at $\sigma=\sigma_{\sn},$ i.e., $P(u_s)=P'(u_s)=0$. Now, $P'(u_s)=0$ gives $u_s=\frac{1}{2}$ and  $P(\frac{1}{2})=0$ implies  $\sigma_{\sn}=4 \eta. $ From \eqref{axialexpre}, we observe that $u_1>\frac{1}{2}$ and $u_2<\frac{1}{2}$ for $\sigma>\sigma_{\sn}$. Since  $\frac{\alpha}{\gamma-1}<\frac{1}{2}$, therefore $E_1$ is always a saddle point and $E_2$ is an unstable node or a saddle point depending on the parameter $\sigma$.
 These two axial equilibria merge into a single axial equilibrium point $E_{\sn}(\frac{1}{2},0)$ at $\sigma=\sigma_{\sn}$ and disappear for $\sigma<\sigma_{\sn} $. The Jacobian matrix evaluated  at $E_{\sn}$ for $\sigma=\sigma_{\sn}$ is  
\begin{equation*}
   J(E_{\sn})= \begin{bmatrix}
0 & -\frac{1}{(2\alpha+1)}  \vspace{0.06in} \\
0 &  \frac{2 \gamma}{2 \alpha+1}-1
\end{bmatrix},
\label{jacoE}
\end{equation*}
which has a zero eigenvalue. Let $\phi=[0,\;1]^T$ and $\psi=[ 2(\gamma-\alpha)-1,\;1]^T$ respectively be two eigenvectors of $J(E_{\sn})$ and $ J(E_{\sn})^T$ corresponding to zero eigenvalue. 
Now, we verify the following transversality conditions:
  \begin{eqnarray*}
\psi^T{F}_\sigma({E_{\sn}};\sigma=\sigma_{\sn})&=&\frac{2(\gamma-\alpha)-1}{8}>0,\\
\psi^TD^2{F}({E_{\sn}};\sigma=\sigma_{\sn})(\phi,\phi)&=& - {\frac {4\beta  \left( -\gamma+2 \alpha+1 \right) }{ \left( 2 
\alpha+1 \right) ^{2}}}\neq 0,
\end{eqnarray*}
where ${F}=\left[F_1(u,v),F_2(u,v)\right]^T$.
Hence, the system (\ref{ode}) undergoes a  saddle-node bifurcation  at $\sigma=4\eta.$ The pictorial  representation of the saddle-node bifurcation is shown in $\sigma$-$u$ plane in Fig. \ref{Hopf1}(a).
\end{proof}

\subsubsection{Transcritical bifurcation}
\begin{theorem}
For $\gamma>2\alpha+1$, the system (\ref{ode}) undergoes a transcritical bifurcation around the axial equilibrium point $E_{\tc}(u_{\tc},0)$ at $\sigma\equiv \sigma_{\tc}={\frac {\eta   \left( \gamma- 1 \right) ^{2}}{\alpha  
 \left(\gamma- \alpha- 1 \right) }}
$, where $u_{_{\tc}}=({\sigma_{\tc}-\sqrt{\sigma_{\tc}^2-4\eta\sigma_{\tc}}})/{2\sigma_{\tc}}$.
 \end{theorem}
 \begin{proof}
 For $\sigma<\sigma_{\tc},$ the axial equilibrium point $E_2$ is unstable which becomes saddle for $\sigma>\sigma_{\tc}$ (see Fig. \ref{Hopf1}(a)). The unstable coexisting equilibrium point $E_*$ becomes feasible for $\sigma>\sigma_{\tc}.$
 
 The Jacobian matrix evaluated  at $E_{\tc}(u_{{\tc}},0)$ for $\sigma=\sigma_{\tc}$ is 
\begin{equation*}
   J(E_{\tc})= \begin{bmatrix}
\sigma (2u_{{\tc}}-3 u_{{\tc}}^2)-\eta & -{u_{{\tc}}}/{(\alpha+u_{{\tc}})}  \vspace{0.06in} \\
0 &  0
\end{bmatrix},
\label{jacoTC}
\end{equation*}
which has a zero eigenvalue.  Let $\phi=[p, \; 1]^T$ and $\psi=[0, \; 1]^T$ be two  eigenvectors corresponding to the zero eigenvalue of $J(E_{\tc})$ and $J(E_{\tc})^T$  respectively, where 
$$p=\frac{u_{{\tc}}}{(\alpha+u_{{\tc}})\big(\sigma (2u_{{\tc}}-3 u_{{\tc}}^2)-\eta\big)}.
$$ 
Now, we verify the following transversality conditions:
 \begin{eqnarray*}
 \psi^T{F}_\sigma({E_{\tc}};\sigma=\sigma_{{\tc}})&=&0,\\
 \psi^TD{F}_\sigma({E_{\tc}};\sigma=\sigma_{{\tc}})\phi&=& 0,\\
 \psi^TD^2{F}({E_{\tc}};\sigma=\sigma_{{\tc}})(\phi,\phi)&=& {\frac {2\gamma u_{{\tc}} \beta}{ \left( \alpha+u_{{\tc}} \right) ^{2}}}\left(\frac{\alpha}{(\alpha+u_{{\tc}})\sigma u_{{\tc}}(1-2u_{{\tc}})}+\beta\right)\neq 0. 
 \end{eqnarray*}
 However, the second transversality condition should be nonzero for non-degenerate transcritical bifurcation \cite{Perko}. Hence, the system (\ref{ode}) undergoes a  degenerate  transcritical bifurcation \cite{degeneratetranscritical}  around $E_{\tc}$ at $\sigma=\sigma_{\tc}.$ 

 \end{proof}

\begin{figure}[ht!]
 \begin{subfigure}[b]{.45\textwidth}
 \centering
\includegraphics[scale=0.5]{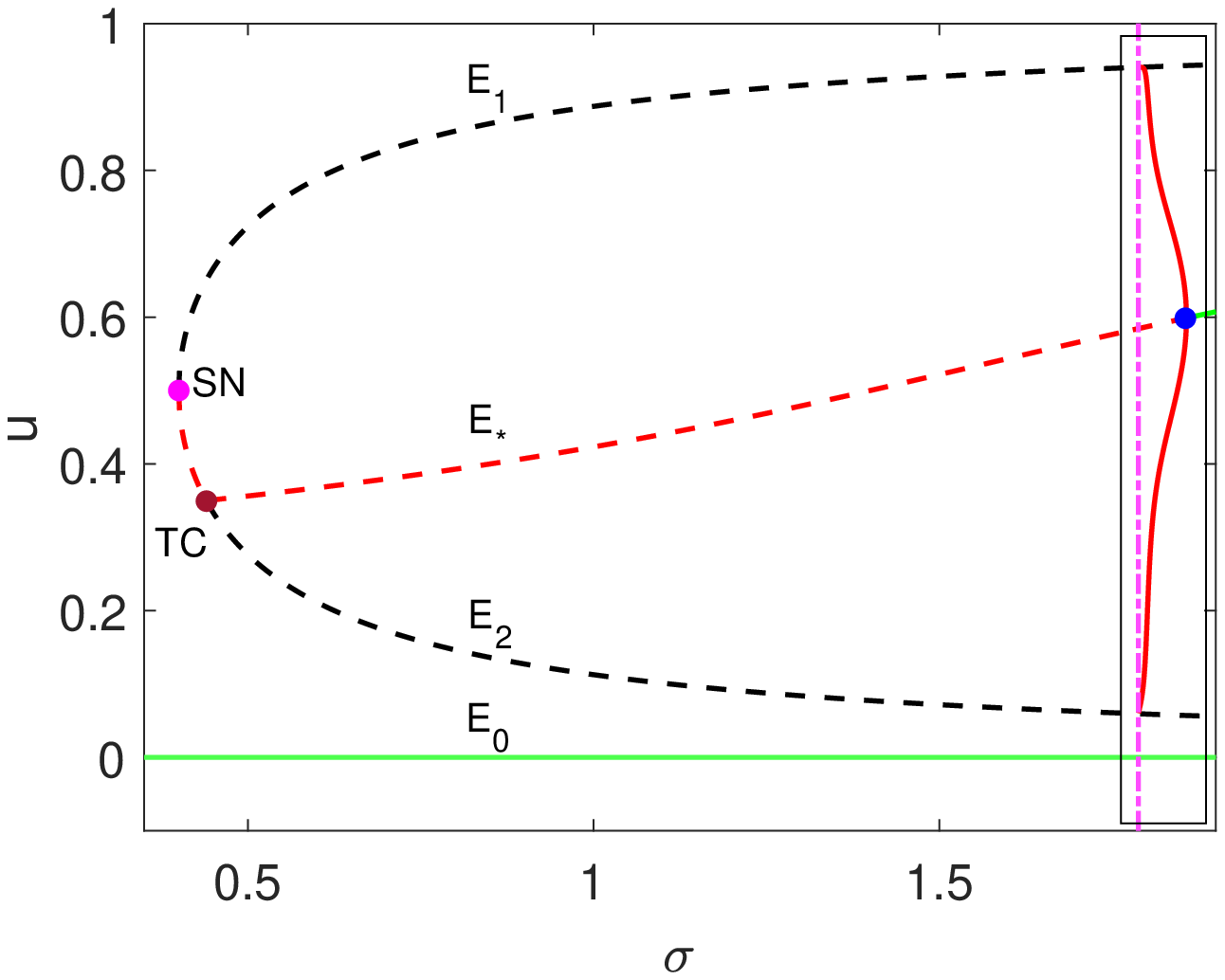}\\ 
 \caption{}
 \end{subfigure}
 \begin{subfigure}[b]{.45\textwidth}
   \centering
\includegraphics[scale=0.5]{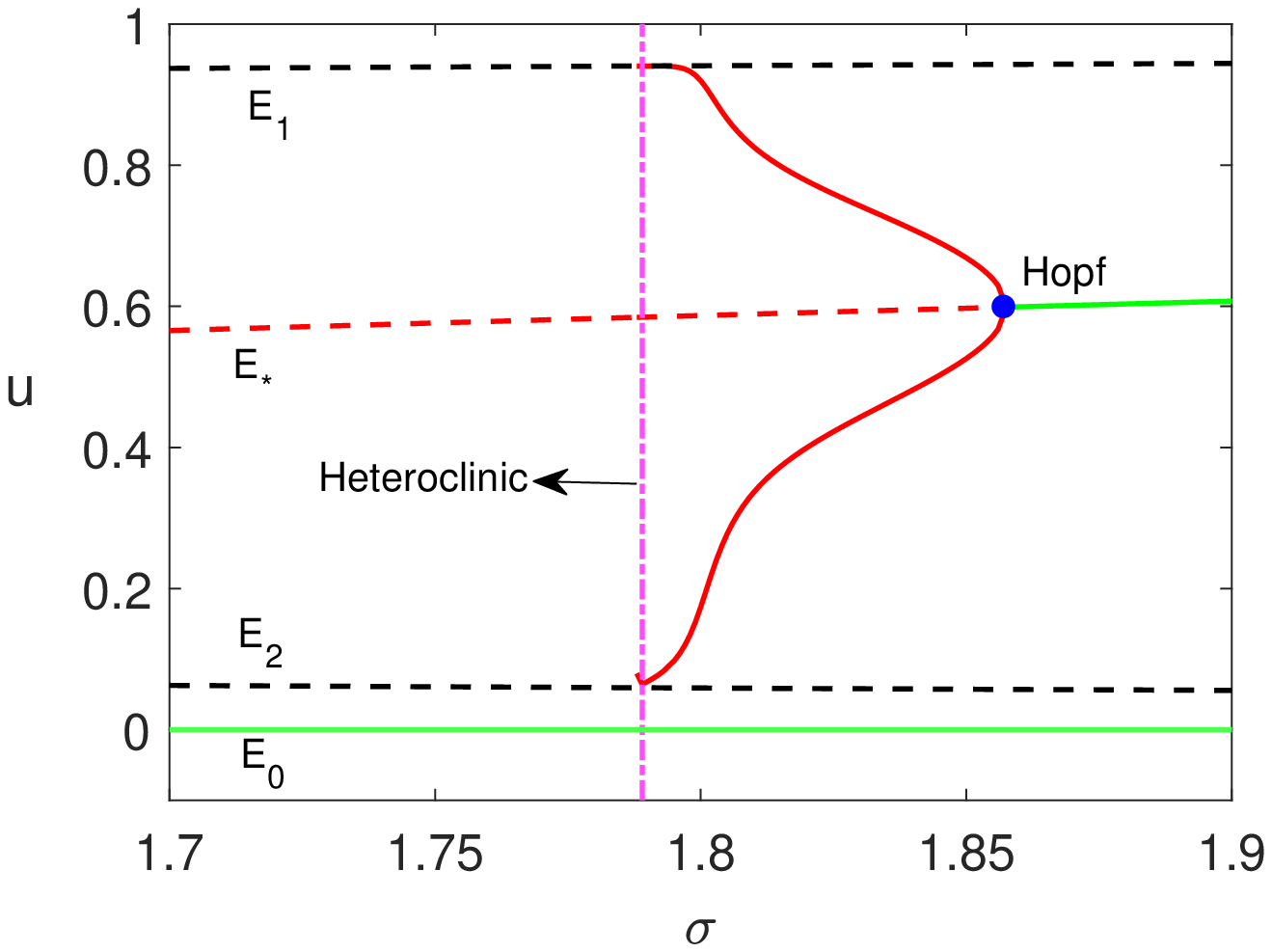}\\ 
 \caption{}
 \end{subfigure} 
\caption{(a) Bifurcation diagram in the $\sigma$-$u$ plane. (b) A zoom version of Fig. (a) for $\sigma\in [1.7,1.9]$. Points marked TC and SN denote the transcritical and saddle-node  bifurcation thresholds respectively. Here, solid green, red dashed and black dashed curves represent stable, unstable  and saddle branches of equilibria respectively. Further, solid red color curve represents the maximum and minimum of $u$ of the stable limit cycle.  Parameter values are $\alpha=0.07,\beta=0.2,  \gamma=1.2$ and $ \eta=0.1$.}
\label{Hopf1}
\end{figure}

\subsubsection{Hopf bifurcation}
 An unstable coexisting equilibrium point $E_*$ can become stable via a Hopf Bifurcation when the trace of the Jacobian matrix evaluated at $E_*$ changes from positive to negative value due to variation of a model parameter.  Here, we choose $\sigma$ as a Hopf bifurcation parameter.
 
 Solving
 $$
 T_{\sigma_{_H}}\equiv tr[J(E _*;\sigma=\sigma_{_H})]=0,
 $$
 we find
$$\sigma_{_H}={\frac { \left( 1-v_*{\beta}^{2}- \left( \alpha+u_* \right) \beta
 \right) v_*}{ \gamma^{2}u_*^3 \left( 2 u_*-1
 \right) }}.
 $$ 
 This is an implicit expression for $\sigma$ since $E_*(u_*,v_*)$ depends on $\sigma$.  The system (\ref{ode}) undergoes a Hopf bifurcation  at   $\sigma=\sigma_{_{H}}$ if the following non-hyperbolicity and transversality conditions are  satisfied: 
 \begin{eqnarray*}
&C1:&  \; \;  D_{\sigma_{_H}}\equiv \mathrm{det}[J(E _*;\sigma=\sigma_{_H})]>0,\\   
&C2:& \; \; \frac{d}{d \sigma} (T_{\sigma}) \vert_{\sigma=\sigma_{_H}} \neq 0.
\end{eqnarray*}
For $\sigma>\sigma_{_H}$, $E_*$ is always asymptotically stable. A limit cycle surrounding  $E_*$ is generated through a Hopf bifurcation at $\sigma=\sigma_{_H}.$  The stability  of the limit cycle is determined by the first Lyapunov coefficient $l_1$ \cite{Perko}.  The Hopf bifurcation is supercritical when  $l_1<0$ and results in a stable limit cycle. In contrary, the subcritical Hopf bifurcation  results in an unstable limit cycle for $l_1>0$. Due to unavailability  of explicit expression of the interior equilibrium $E_*$, it is impossible to obtain the sign of $l_1$ analytically. However, we obtain Hopf bifurcation and its stability numerically. We find that Hopf bifurcation occurs around  ${E_*}=(0.5986,0.2486)$ at $\sigma_{_H}=1.8566$ for parameter values $\alpha=0.07,\beta=0.2, \gamma=1.2$ and $ \eta=0.1$. The corresponding first Lyapunov coefficient $l_1=-22.7488\pi<0$, indicating that the  Hopf bifurcation  is supercritical.  The corresponding temporal bifurcation diagram  is plotted in the Fig. \ref{Hopf1}(b). The Hopf generating stable limit cycle vanishes through a heteroclinic bifurcation at $\sigma_{Het}=1.789<\sigma_H.$ For $\sigma<\sigma_{Het}$, trivial equilibrium point $E_0$ is  globally asymptotically stable for the temporal system (\ref{ode}). A decrease in prey growth drives the system from stable coexistence to oscillatory coexistence and further decrease results in system collapse through a global bifurcation.
 
\section{Spatio-temporal  Model}\label{spatialmodel}
Random movement of the population species is taken into account by incorporating diffusion term in the temporal model. For simplicity, we consider the spatio-temporal system in  one-dimensional spatial domain $\Omega:=(0,L)$. The dimensionless spatio-temporal system is 
\begin{subequations}
\begin{equation}
\frac{\partial u}{\partial t}=  \frac{\partial^2 u}{\partial x^2}+F_1(u,v),
\label{second:a}
\end{equation}
\begin{equation}
\frac{\partial v}{\partial t}=  d\frac{\partial^2 u}{\partial x^2}+F_2(u,v),
\label{second:b}
\end{equation}
\label{pde}
\end{subequations}
where $d$ is is ratio of the diffusion coefficients of predator and prey populations. The above system is  subjected to initial conditions $u_0(x)\equiv u(x,0)\geq 0$, $v_0(x)\equiv v(x,0)\geq 0$ for $x\in \Omega$ and no-flux boundary conditions $\frac{\partial u}{\partial n}=\frac{\partial v}{\partial n}=0$ for $x\in \partial\Omega:=\{0,L\}$.

\subsection{Existence and boundedness of global solution}

In order to ensure well-posedness of the model (\ref{pde}), we establish the global existence of the solutions of the system (\ref{pde})  and their priori bounds under specified initial and boundary condition.  Let $\mathscrsfs{D}=\Omega\times (0,\infty)$, $\bar{\mathscrsfs{D}}={\bar \Omega}\times (0,\infty)$ and $\mathscrsfs{B}=\partial\Omega\times(0,\infty)$. 
\begin{theorem} \label{prop3}
For the system (\ref{pde}), the following hold whenever $\sigma>4\eta$:
\begin{itemize}
    \item[(a)]  If $u_0(x)\geq0 $ and $v_0(x)\geq 0$, then $u(x,t)>0$  and $v(x,t)>0$ in $\bar{\mathscrsfs{D}}$ unless $\big(u(x,t),v(x,t)\big)\equiv (0,0)$ in $\bar{\mathscrsfs{D}}$.
    \item[(b)] If $\sup _{\bar \Omega} u_0(x) < u_2$, then $\big(u(x,t),v(x,t)\big)\rightarrow(0,0)$ uniformly as $t\rightarrow \infty.$
    
    \item[(c)]  If $\gamma<1$, then $\big(u(x,t),v(x,t)\big)$ $\rightarrow$ $(u_A(x),0),$ where $u_A(x)$ is the solution of 
     \begin{equation*}
     \begin{cases}
\nabla^2 u +\big(\sigma u^2(1-u) -\eta u\big)=0, \text{ for } x\in \Omega,\\
\frac{\partial u}{\partial n}=0, \text{ for } x\in \partial\Omega . 
 \end{cases}
 \end{equation*}
     \item[(d)] For any solution  $\big(u(x,t),v(x,t)\big)$  of \eqref{pde},
     $$\limsup_{t\rightarrow \infty}\int _\Omega u(x,t)\,dx\leq u_1 \vert \Omega \vert \mbox{ and }\limsup_{t\rightarrow \infty} \int _\Omega v(x,t)\,dx \leq \gamma\left(1+\frac{\sigma}{4}-\eta\right)u_1\vert \Omega \vert.$$ 
\end{itemize}
\end{theorem}
\begin{proof}
(a) Since  
$$
\frac{\partial F_1(u,v)}{\partial v}\leq0\;\,\mbox{ and }\;\, \frac{\partial F_2(u,v)}{\partial u}\geq0\quad\mbox{for}\;\, (u,v)\in \mathbb{R}^2_{\geq0}=\{u\geq0,v\geq0\},
$$
the system \eqref{pde} is a mixed quasi-monotone system \cite{pao2012nonlinear,travebook}. Let $\big(\hat u(t), \hat v(t)\big)$ be the solution of the differential equations 
\begin{equation}
     \begin{cases}
 \frac{du}{dt}=\sigma u^2(1-u) -\eta u\equiv \sigma u(u-u_1)(u_2-u),\\[5pt]
 \frac{dv}{dt}=\displaystyle{\frac{\gamma uv}{\alpha+u+\beta v}-v},\\[5pt]
  u(0)=\sup _{\bar\Omega} u_0(x) \text{ and } v(0)=\sup _{\bar\Omega} v_0(x),
 \end{cases}
 \label{existode}
 \end{equation}
where $u_1$ and $u_2$  are given in \eqref{axialexpre}. Clearly $\hat u(t)<\rho_1$ and $\hat v(t)<\rho_2$ \cite{dey2022analytical},  where $$\rho_1=\text{max}\left\{\hat u(0), u_1\right\}\; \text{and}\;\rho_2=\text{max}\left\{\hat v(0),   u_1 \left(1+\frac{\gamma}{\beta}+\frac{\sigma}{4} -\eta \right)\right\}.$$ 
Using the definition 8.1.2 in \cite{pao2012nonlinear}, we find that $\big(\underline u(x,t),\underline v(x,t)\big)=(0,0)$ and  $\big(\overline u(x,t),\overline v(x,t)\big)=\big(\hat u(t),\hat v(t)\big)$  are lower and upper solutions of system \eqref{pde} since  $$\frac{\partial\overline u(x,t)}{\partial t}-\nabla^2 \overline{u}(x,t)- F_1(\overline{u}(x,t),\underline{v}(x,t))=0 \geq 0= \frac{\partial\underline u(x,t)}{\partial t}  -\nabla^2 \underline{u}(x,t)- F_1(\underline{u}(x,t),\overline{v}(x,t)),
$$
$$\frac{\partial\overline v(x,t)}{\partial t}-\nabla^2 \overline{v}(x,t)- F_2(\overline{u}(x,t),\overline{v}(x,t))=0 \geq 0=  \frac{\partial\underline v(x,t)}{\partial t}    -\nabla^2 \underline{v}(x,t)- F_2(\underline{u}(x,t),\underline{v}(x,t)),
$$
for $(x,t)\in {\mathscrsfs{D}}$ with boundary conditions 
$$
\frac{\partial \overline{u}(x,t)}{\partial n}\geq 0 \geq \frac{\partial \underline{u}(x,t)}{\partial n}, \quad      \frac{\partial \overline{v}(x,t)}{\partial n}\geq 0 \geq \frac{\partial \underline{v}(x,t)}{\partial n}\; \mbox{ for } (x,t) \in \mathscrsfs{B}, 
$$ and initial conditions
$$
\underline u(x,0)\leq u_
0(x)\leq\overline{u}(x,0)\text{ and }\underline v(x,0)\leq v_0(x)\leq\overline{v}(x,0)\; \text{ for } x \in \Omega.$$
Now, using theorem 8.3.3 of \cite{pao2012nonlinear},  we conclude that the system \eqref{pde} admits unique globally defined solution $\big(u(x,t),v(x,t)\big)$ which satisfies 
$$0\leq u(x,t)\leq\hat u(t)\text{ and }0\leq v(x,t)\leq\hat v(t).$$
From the expressions of $\rho_1$ and $\rho_2$ given above, we observe that
$$
F_1(\rho_1,0)\leq 0\leq F_1(0,\rho_2)\text{ and }F_2(\rho_1,\rho_2)\leq 0\leq F_2(0,0).
$$ 
Hence, if we define 
$$
S=\left\{(u,v)\in C^2(\bar{\mathscrsfs{D}}):\; 0\leq u\leq\rho_1, 0\leq v\leq \rho_2\right\},
$$
and $\big(u(x,0), v(x,0)\big)\in S$, then  $u(x,t)>0, v(x,t)>0$ or $\big(u(x,t),v(x,t)\big)\equiv (0,0)$ for $(x,t)\in \bar{\mathscrsfs{D}}$ (see theorem 8.3.3 of \cite{pao2012nonlinear}).

(b) Since $\hat u(0)=\sup _{\bar \Omega} u_0(x)$, we have $\hat u(0)< u_2$. Now, from the first equation of \eqref{existode} we find that $\hat u(t)\rightarrow 0$ as $t\rightarrow \infty$ and consequently  $\hat v(t)\rightarrow 0$ as $t\rightarrow \infty$. Thus, $\big(u(x,t),v(x,t)\big)\rightarrow(0,0)$ uniformly as $t\rightarrow \infty.$

(c) If $\gamma<1,$ then 
$$
\frac{d \hat v}{dt}\leq (\gamma-1)\hat v<0,
$$
which gives $\hat{v}(t)\rightarrow 0$ as $t\rightarrow\infty$. Since, $0\leq v(x,t)\leq \hat{v}(t)$, we conclude that ${v}(x,t)\rightarrow 0$ uniformly as $t\rightarrow\infty$  for $x\in  \bar\Omega.$ Therefore the limiting behavior of $u(x,t)$ is determined by the semi-flow generated by the parabolic 
equation 
\begin{equation}
     \begin{cases}
u_t=\nabla^2 u +\big(\sigma u^2(1-u) -\eta u\big), \text{ for } x\in \Omega,\,t>0,\\
\frac{\partial u}{\partial n}=0, \text{ for } x\in \partial\Omega.
 \end{cases}
 \label{grad}
 \end{equation}
 Now, every orbit of the gradient system \eqref{grad} converges to a steady state
 $u_A(x)$ \cite{hale2010asymptotic}. Then, using the theory of asymptotically autonomous dynamical system \cite{hale2010asymptotic}, 
 we conclude that the solution $(u(x,t),v(x,t))\rightarrow(u_A(x),0)$ as $t\rightarrow\infty$. 
 
 (d) If $\hat u(0)> u_2,$ then $\hat u(t)\rightarrow u_1$ as $t\rightarrow \infty$. Thus for any $\epsilon>0,$ there exist $T_0>0$ such that $u(x,t)\leq u_1+\epsilon$ for $(x,t)\in \bar\Omega\times (T_0,\infty)$. Thus $$\limsup_{t\rightarrow \infty}\int _\Omega u(x,t)\,dx\leq u_1 \vert \Omega \vert.$$
 To estimate the bound of $v(x,t)$, we introduce $P(t)=\int _\Omega u(x,t)\,dx$ and $Q(t)=\int _\Omega v(x,t)\,dx$. Then 
\begin{eqnarray*}
\frac{dP}{dt}&=& \int_\Omega \nabla^2 u \;dx+\int _\Omega\left[\sigma u^2(1-u) -\eta u-\frac{uv}{\alpha+u+\beta v}\right]dx ,
\label{9}\\[3pt]
\frac{dQ}{dt}&=& d\int_\Omega \nabla^2 v \;dx+\int _\Omega\left[\frac{\gamma uv}{\alpha+u+\beta v}-v\right]dx.
\label{10}
\end{eqnarray*}
 Using Neumann boundary condition, we obtain  \begin{equation*} \label{eq1}
\begin{split}
  \frac{d}{dt}(\gamma P+Q)&= \int _\Omega\gamma \big(\sigma u^2(1-u) -\eta u\big)\,dx-Q\,\leq\,-(\gamma P+Q)+\gamma(1+\frac{\sigma}{4}-\eta)P\\
       &\leq -(\gamma P+Q)+\gamma(1+\frac{\sigma}{4}-\eta)(u_1+\epsilon)\vert \Omega \vert\; \text{ for } t>T_0.
\end{split}
\end{equation*}
  Integrating the above gives
  $$\gamma P(t)+Q(t)\leq \gamma(1+\frac{\sigma}{4}-\eta)(u_1+\epsilon)\vert \Omega \vert+\epsilon\;\, \text{ for } t>T_1,$$
  where $T_1>T_0$. Since $P(t)\geq0$, we find $$\limsup_{t\rightarrow\infty}Q(t)\leq\gamma(1+\frac{\sigma}{4}-\eta)u_1\vert \Omega \vert.$$
\end{proof}

\subsection{ Homogeneous steady state analysis}
The equilibrium points of the temporal model \eqref{ode} correspond to  homogeneous steady-states of spatio-temporal model. Therefore, the system \eqref{pde} has homogeneous steady-states corresponding to the trivial steady state $E_0$, semi-trivial steady states $E_1$ and $E_2$,  and the coexisting steady state $E_*.$ Here we discuss the stability of these homogeneous steady state solutions. 
\begin{theorem}
Suppose $\sigma> \sigma_{TC}$ and $\gamma>1$, then the following conditions hold:
\begin{itemize}
    \item[(a)]  Trivial steady state $E_0(0,0)$ is always asymptotically  stable.
    \item[(b)] Semi-trivial steady state $E_1(u_1,0)$ and $E_2(u_2,0)$ are always unstable.
    \item[(c)]  Coexisting steady state $E_*(u_*,v_*)$ is  asymptotically  stable if $$\sigma>\frac{v_*}{\gamma^2u_*^3(2u_*-1)}\equiv\sigma_S,$$
    and unstable if $\sigma< \sigma_{_H}.$
\end{itemize}
\end{theorem}
\begin{proof}
The linearized system of (\ref{pde}) about a homogeneous steady state $E(u_e,v_e)$ can be expressed as 
\begin{equation*}
    \frac{\partial W}{\partial t}=\mathcal{L}(W):=\mathcal{D} \nabla^2 W+J(E) W,
\end{equation*}
where $J(E)$ is defined in \eqref{jaco}, $\mathcal{D}=$diag$(1,d)$ and $W\in C^2(\bar{\mathscrsfs{D}})\times  C^2(\bar{\mathscrsfs{D}}).$

 Let $0=k_0<k_1<\cdots<k_j<\cdots$ be the eigenvalues and $E(k_j)$ be the eigenfunction space corresponding to $k_j$  for the eigenvalue problem
 \begin{eqnarray*}
 &&-\nabla^2 w= k w \quad  \mbox{in} \;\Omega,\\
 &&\frac{\partial w}{\partial n}=0\qquad\quad\;\mbox{on}\;  \partial \Omega.
 \end{eqnarray*}
  Further, suppose that $\{\psi_{i,j}: i=1,\cdots,\text{dim}\big(E(k_j)\big)\}$ be an orthogonal basis set of $E(k_j)$ and $\mathcal{W}_{ij}=\{c \psi_{i,j}: c=(c_1,c_2)^T \}$. Let  $\mathcal{W}_j=\bigoplus_{i=1}^{\mbox{dim}(E(k_j))} \mathcal{W}_{ij}$  be the direct sum of $\mathcal{W}_{ij}$. 
  It can be shown that  \begin{equation}
    \mathcal{W}\equiv\left\{\left(\phi,\psi\right)^T \in C^2(\bar \Omega)\times  C^2(\bar \Omega): \frac{\partial \phi}{\partial n}=\frac{\partial \psi}{\partial n}=0\; \text{ for } x\in \partial\Omega\right\}=\bigoplus_{j=1}^\infty \mathcal{W}_j,\label{defw}
  \end{equation}
  and $\mathcal{W}_j$ is invariant under the operator $\mathcal{L}$. Now, $\lambda$ is an eigenvalue of $\mathcal{L}$ if and only if $\lambda$ is an eigenvalue of the matrix $\mathcal{L}_j=-k_j\mathcal{D}+J(E)$ for some $j\geq0.$  The characteristic equation of $\mathcal{L}_j$ is given by 
  \begin{equation*}
      \text{det} (\lambda I -\mathcal{L}_j)=\lambda^2-\text{trace} (\mathcal{L}_j)\lambda+\text{det} (\mathcal{L}_j),
  \end{equation*}
where $\text{trace} (\mathcal{L}_j)=a_{10}+b_{01}-(1+d)k_j,\;\text{det} (\mathcal{L}_j)=dk_j^2-(da_{10}+b_{01})k_j+(a_{10}b_{01}-a_{01}b_{10}).$

(a) At $E_0$, we find $J(E_0)=\begin{pmatrix}
   -\eta & 0\\
   0 &-1
    \end{pmatrix},$ and $\text{trace} (\mathcal{L}_j)=-1-\eta-(1+d)k_j<0,\;\text{det} (\mathcal{L}_j)=dk_j^2+(d\eta+1)k_j+\eta>0.$ Hence, $E_0$ is always asymptotically  stable.
    
(b) For $E(u,0)$, we find $J(E)=\begin{pmatrix}
   \sigma u(1-2  {u}) & \frac {u}{\alpha+u}\\
   0 &-1+{\frac {\gamma u}{\alpha+u}}
    \end{pmatrix}.$  
   For $j=0$, $\mathcal{L}_j$  has one positive eigenvalue $-1+\frac {\gamma u}{\alpha+u}>0$ since $\gamma>1$ and $\sigma> \sigma_{TC}$. Thus, both $E_1$ and $E_2$ are  unstable. 
   
   (c) From the Jacobian matrix $J(E_*)$ given in \eqref{jacoh}, we find that  $a_{01}<0$, $b_{10}>0$ and $b_{01}<0.$  Now, $E_*$ is asymptotically stable if and only if   \begin{equation}
       \text{trace} (\mathcal{L}_j)<0 \text{ and }\text{det} (\mathcal{L}_j)>0 \text{ for all } j.
       \label{cond}
   \end{equation}
    
   If $\sigma>\sigma_S$, then $a_{10}<0$ and both the conditions in (\ref{cond}) are satisfied. 
   Thus, the steady state solution $E_*$ is stable when $\sigma>\sigma_S$.
   On the other hand,  $a_{10}+b_{01}>0$ for $\sigma<\sigma_{H}$ which leads to $\text{trace} (\mathcal{L}_j)>0$ for $j=0$.  Therefore, one of the eigenvalues of $\mathcal{L}_0$ must have positive real part and the steady state $E_*$ becomes unstable.
\end{proof}
\begin{remark}From Theorem \ref{prop6}, we observe that $E_*$ is asymptotically stable for $\sigma>\sigma_{S}$ and unstable for $\sigma<\sigma_{H}$ independent of the diffusion parameter value $d$. Hence, the system \eqref{pde} shows bistability between the trivial steady state $E_0$ and the coexisting steady state $E_*$ for $\sigma>\sigma_{S}$. The system may show various non-constant stationary solutions for $\sigma<\sigma_{S}$,  which we discuss in section \ref{SDNC}.
\end{remark}
 \begin{remark}
Since the semi-trivial steady states are always saddle points, there is a possibility of appearance of travelling wave solution which we discuss in section \ref{TRWS}.
\end{remark}

\subsection{Steady state solution}
The spatio-temporal system \eqref{pde} admits time independent solutions $(u(x),v(x))$ which satisfy 
\begin{equation}
    \frac{d^2u}{dx^2}\,+\,F_1(u,v)\;=\;0,\;\;d\frac{d^2v}{dx^2}\,+\,F_2(u,v)\;=\;0,\quad\text{ for }x\in \Omega,
 \label{PSS}
\end{equation}
along with the boundary conditions $\displaystyle{ \frac{\partial u}{\partial n}=\frac{\partial v}{\partial n}=0 \text{ on } \partial  \Omega.}$ Here we determine the conditions which establish the existence and non-existence of such solutions.

\subsubsection{Nonexistence of non-constant steady state solution}
Here we show the non-existence of non-constant solution. Before proving the main theorem,  we need the following proposition.
\begin{prop}\label{prop5}
Suppose that  $(u(x),v(x))$ is a non-negative nontrivial solution of (\ref{PSS}). Then $(u(x),v(x))$ is of the form of semi-trivial solution   $(u(x),0)$ which satisfy \eqref{grad} or coexisting  solution $(u(x),v(x))$ which satisfy  $$0<u(x)<u_1 \text{ and }  0<v(x)<M^*\equiv\gamma  u_1 \left(\frac{\sigma}{4} -\eta \right) \text{ for }x\in \Omega.$$
\end{prop}
\begin{proof}
If $v(x_0)=0$ for some $x\in \Omega$, then  strong  maximum principle  implies $v(x)\equiv0$  and $u(x)$ satisfies  \eqref{grad}. On the other hand, if there exists $x_0\in \Omega$ such that $u(x_0)=0$, then strong  maximum principle gives $u(x)\equiv0$, which also implies that $v(x)\equiv0$.  Hence, the other possibility is $u(x)>0$ and $v(x)>0$ for $x\in\Omega$.  

 From theorem \ref{prop3}(d), we have $u(x)\leq u_1$ for $x\in \Omega$ as  $t\rightarrow \infty$,  and thus the strong  maximum principle implies $u(x)< u_1 \text{ for }x\in \Omega.$ Now from \eqref{PSS}, we find  \begin{equation*} 
\begin{split}
  -(\gamma\nabla^2 u+ d  \nabla^2v)=\gamma\big(\sigma u^2(1-u) -\eta u\big)- v   &= -\frac{1}{d}(\gamma u+d v)+\gamma u\left(\sigma u(1-u) -\eta +\frac{1}{d}\right)\\&
  \leq -\frac{1}{d}(\gamma u+d v) +\gamma u_1 \left(\frac{\sigma}{4} -\eta +\frac{1}{d}\right).
\end{split}
\end{equation*}
Then using maximum principle theorem \cite{lou1996diffusion}, we find 
$$(\gamma u+d v)\leq\gamma d u_1 \left(\frac{\sigma}{4} -\eta +\frac{1}{d}\right)\;\text{ for }x\in \Omega,$$
which  leads to the desired estimate $$0<v(x)<M^*\;\text{ for }x\in \Omega.$$
\end{proof}

\begin{theorem} \label{prop6}
If $\sigma>4\eta$ and $ \gamma>1$, then there exists  $d^*=d^*(\alpha,\beta,\gamma,\sigma,\eta,\vert \Omega \vert)$ such that 
the system (\ref{pde}) does not admit a non-constant steady state solution for $d^*<d<1$.
\end{theorem}
\begin{proof}
Let  $\big(u(x),v(x)\big)$ be a non-constant solution of (\ref{PSS}). Then
$$
\int _\Omega (u-\bar u) dx=\int _\Omega (v-\bar v) dx=0,\,\,\textrm{where} \,\,\bar u=\frac{1}{\Omega}\int _\Omega u(x) dx, \text{ and }  \quad\bar v=\frac{1}{\Omega}\int _\Omega v(x) dx.
$$ 
Multiplying the first equation of (\ref{PSS}) by $u-\bar u$ and  integrating over $\Omega$ using the no-flux boundary
condition, we find 
\begin{equation*} 
  \int _\Omega \vert\nabla(u-\bar u)\vert^2 dx=\int _\Omega (u-\bar u) F_1(u,v)dx = \int _\Omega (u-\bar u) \big(F_1(u,v)-F_1(\bar u,\bar v)\big)dx\equiv I_1+I_2,
\end{equation*}
where.
\begin{alignat*}{4}
I_1&=\int _\Omega (u-\bar u)\big(u(u-u_1)(u_2-u)-\bar u(\bar u-u_1)(u_2-\bar u)\big) dx, \;\\I_2&=\int _\Omega (\bar u- u)\left(\frac{uv}{\alpha+u+\beta v}-\frac{\bar u \bar v}{\alpha+\bar u+\beta\bar  v}\right) dx.
\end{alignat*}

Similarly,  multiplying the second equation of (\ref{PSS}) by $v-\bar v$ and integrating over $\Omega$ using no-flux boundary
condition, we find 
\begin{equation*} 
\begin{split}
  d\int _\Omega \vert\nabla(v-\bar v)\vert^2 dx=&\int _\Omega (v-\bar v) F_2(u,v)dx\\
  =& \int _\Omega (v-\bar v)^2 \left(\frac{\gamma u}{\alpha+u+\beta v}-1\right)dx+I_3,\\
\end{split}
\end{equation*}
where \begin{equation*} 
\begin{split}
  I_3=& \int _\Omega (v-\bar v)\bar v \left(\frac{\gamma u}{\alpha+u+\beta v}-1\right)dx.
\end{split}
\end{equation*}
Using upper bounds of $I_1$, $I_2$ and $I_3$ (see Appendix \ref{Ap1}), we get
\begin{equation}
     \int _\Omega \vert\nabla(u-\bar u)\vert^2 dx\leq \left(\frac{u_1}{2\alpha}+u_1(2-u_2)\right)\int _\Omega (u-\bar u)^2dx +\frac{u_1}{2\alpha}\int _\Omega (v-\bar v)^2dx.
     \label{none1}
\end{equation}

\begin{equation}
     d\int _\Omega \vert\nabla(v-\bar v)\vert^2 dx\leq \frac{\gamma}{2\beta}\int _\Omega (u-\bar u)^2dx +\left(\frac{\gamma}{2\beta}+(\gamma-1)\right)\int _\Omega (v-\bar v)^2dx.
     \label{none2}
\end{equation}
Adding \eqref{none1}, \eqref{none2} and using Poincare inequality, we find 
\begin{equation}
      \int _\Omega \vert\nabla(u-\bar u)\vert^2 dx+ d\int _\Omega \vert\nabla(v-\bar v)\vert^2 dx\leq  \frac{1}{k_1}\left(A\int _\Omega \vert\nabla(u-\bar u)\vert^2 dx+ B\int _\Omega \vert\nabla(v-\bar v)\vert^2 dx\right),
     \label{none3}
\end{equation}
where $$A=\frac{\gamma}{2\beta}+\frac{u_1}{2\alpha}+u_1(2-u_2)\; \text{ and }B=\frac{\gamma}{2\beta}+(\gamma-1)+\frac{u_1}{2\alpha}.$$
If $$\text{min}\{1,d\}\geq  \frac{1}{k_1}\text{max}\{A,B\}\equiv d^*,$$ then $$\nabla(u-\bar u)=\nabla(v-\bar v)=0,$$ which leads to constant solution.
\end{proof}
\begin{remark}
Theorem \ref{prop6} is not applicable for any parameter set when $d^*>1$. For $d^*<1,$ the first wavenumber $k_1$ must be large which leads to a small domain size. To verify theorem \ref{prop6} with numerical simulation, consider a one-dimensional domain of length $\vert \Omega \vert=1$ with parameters values $\eta = 0.1,$ $\alpha = 0.2,$  $\beta = 2.4,$ $\gamma = 1.3,$ and $\sigma = 1.5.$
Then we obtain $d^*=0.4439$ with these parameter values and the system always shows a homogeneous steady-state  corresponding to either the trivial steady state $E_0$ or the coexistence steady state $E_*$ for $d^*<d<1.$ 
\end{remark}

\subsection{ Existence of non-constant steady state solution}
Here we establish the existence of positive non-constant steady solution of \eqref{pde} using Leray–Schauder degree theory.  We define 
$\displaystyle{ W=(u,v)^T,\; W_*=(u_*,v_*)^T,\;
\mathcal{W}^+=\{W\in \mathcal{W}\;:u,v>0\, \text{ on } \bar \Omega\},}$ and $\displaystyle{ \mathcal{F}(W)=(F_1,d^{-1}F_2)^T,}$
where $\mathcal{W}$ has been defined in (\ref{defw}). We write the system (\ref{PSS}) as 
\begin{equation}
    \begin{cases}
-\nabla^2 W=\mathcal{F}(W), &x\in \Omega, \\
 \frac{\partial W}{\partial n}=0, &x\in \partial  \Omega.
 \end{cases}
 \label{PSSS}
\end{equation}
Now the system (\ref{PSSS}) has a positive solution $W$ if and only if 
$$\mathcal{T}(W):=W-(\mathcal{I}-\nabla^2)^{-1}\{\mathcal{F}(W)+W\}=0\; \text{ for } W \in \mathcal{W}^+,$$
where $\mathcal{I}$ is the identity operator, $ (\mathcal{I}-\nabla^2)^{-1}$  is the inverse of $(\mathcal{I}-\nabla^2)$  under Neumann boundary condition.  We observe that  
$$D_W\mathcal{F}(W_*)=\mathcal{A},\; \text{and } D_W\mathcal{T}(W_*)=\mathcal{I}-(\mathcal{I}-\nabla^2)^{-1}(\mathcal{I}+\mathcal{A}),$$
where $$\mathcal{A}=\begin{pmatrix} a_{10} & a_{01}\\d^{-1}b_{10}& d^{-1}b_{01}\end{pmatrix}.$$
Now, $\lambda$ is an eigenvalue of $D_W\mathcal{T}(W_*)$ if and only if $\lambda(1+k_j)$ is an eigenvalue of the matrix $(k_j \mathcal{I}-\mathcal{A}).$ Hence, $D_W\mathcal{T}(W_*)$ is invertible if and only if $$\mathcal{H}(k_j,d):=\text{det} (k_j \mathcal{I}-\mathcal{A})=d^{-1}\big(dk_j^2-k_j(da_{10}+b_{01})+\mathrm{D}
\big)\neq0 \text{ for all }j\geq0.$$  
By the Leray–Schauder theorem, we know that if $D_W\mathcal{T}(W_*)$ is invertible then $$\text{index} (\mathcal{T(.),W_*})=(-1)^\xi,$$
where $\xi$ is the sum of the algebraic multiplicities of all the negative eigenvalues of $D_W\mathcal{T}(W_*).$ This leads to the following lemma \cite{index}:
\begin{lemma}
If $\mathcal{H}(k_j,d)\neq0$ for all $j\geq 1$, then $$\mathrm{index}(\mathcal{T(.),W_*})=(-1)^\xi,$$ where \[\xi:=\sum_{j\geq1,\; \mathcal{H}(k_j,d)<0}\mathrm{dim }\,E(k_j).\]
\end{lemma}
We now describe a theorem which guarantees the existence of at least one positive non-constant solution. For this we assume that
\begin{equation}
a_{10}>0\;\mathrm{ and }\; a_{10}d + b_{01}  > 2 \sqrt{ d\mathrm{D} }> 0.\label{hypth}
\end{equation}
Then $\mathcal{H}(k,d)=0$ has two real roots $k_{\pm}$, where 
 $$
k_{\pm}= \frac{(da_{10}+b_{01})\pm\sqrt{(da_{10}+b_{01})^2-4d\mathrm{D}}}{2d}.
$$ 
\begin{theorem}
Suppose that (\ref{hypth}) holds.  If there exist some integers $0 \leq n_ 1 < n_ 2$ such that \begin{equation} \label{20}
    k_{-}\in( k_{n _1} , k_{n_ 1 +1}),\;k_{+}\in( k_{n _2} , k_{n_ 2 +1}),
\end{equation}
and $\sum_{j=n_1+1}^{n_2}\text{dim}E(k_j)$ is odd, then model (\ref{pde}) has at least one positive non-constant solution.
\end{theorem}
\begin{proof}
We prove this theorem using topological degree with homotopy invariance. Theorem \ref{prop6} implies the existence of $\bar d<d$ such that
\begin{enumerate}
    \item[(i)]  system  (\ref{PSS}) with diffusion coefficient  $\bar d$ has no non-constant solutions,
    \item[(ii)] $\mathcal{H}(k_j,\bar d)>0$ for all $j\geq 0.$
\end{enumerate}
For $ \tau\in [0, 1]$, we define $$\overline{\mathcal{F}}(\tau,W)=\left(F_1,
\big(\tau d+(1-\tau)\bar d)^{-1}\big)F_2\right)^T$$
 and consider the problem 
\begin{equation}
    \begin{cases}
-\nabla^2 W=\overline{ \mathcal{F}}(\tau,W) &x\in \Omega, \\
 \frac{\partial W}{\partial n}=0 &x\in \partial  \Omega.
 \end{cases}
 \label{PSSS1}
\end{equation}
Then the system (\ref{PSSS1}) has positive solution $W$ if and only if 
$$\overline{ \mathcal{T}}(\tau,W)=W-(\mathcal{I}-\nabla^2)^{-1}\big(\overline{\mathcal{F}}(\tau,W)+W\big)=0 \text{ for } W \in \mathcal{W}^+.$$
A straightforward calculation yields $$D_W\overline{ \mathcal{T}}(\tau,W_*)= \mathcal{I}-(\mathcal{I}-\nabla)^{-1}(\mathcal{I}+\overline{ \mathcal{A}}),$$ where $$\overline{ \mathcal{A}}=\begin{pmatrix} a_{10} & a_{01}\\(\tau d+(1-\tau)\bar d)^{-1})b_{10}& (\tau d+(1-\tau)\bar d)^{-1})b_{01}\end{pmatrix}.$$
Thus, $W$ is a positive non-constant solution of (\ref{PSS}) if and only if it is a solution of (\ref{PSSS1}) when $\tau=1.$

Observe that $\overline{\mathcal{F}}(1,W)=\mathcal{F}(W), \text{ and } D_W\overline{ \mathcal{T}}(1,W_*)=D_W{ \mathcal{T}}(W_*).$ Now from (\ref{20}), we find 
\begin{equation}
    \begin{cases}
\mathcal{H}(k_j,d)<0 & \text{ for }n_1+1\leq j\leq n_2, \\
\mathcal{H}(k_j,d)>0 & \text{ otherwise}.
 \end{cases}
 \label{PSSS2}
\end{equation}
Further, using Lemma 1, we get $$\xi=\sum_{j\geq0,\;{H}(k_j,d)<0}\text{dim}E(k_j)=\sum_{j=n_1+1}^{n_2}\text{dim}E(k_j),$$
which is an odd number.  Therefore
$$\text{index}(\overline{\mathcal{T}}(1,W_*)=(-1)^\xi=-1.$$

Using (ii),  we obtain 
$$\text{index}(\overline{\mathcal{T}}(0,W_*)=(-1)^0=1.$$

Using proposition \ref{prop5}, 
we observe the  existence $\tilde M>0$ such that  $\tilde M ^{-1}<u,v<\tilde M$ for all $\tau\in [0,1]$ and thus $\overline{\mathcal{T}}(\tau,W)\neq0.$
Consider the set $\mathbb{B}(\tilde M)$ defined by $$\mathbb{B}(\tilde M)=\left\{W\in \mathcal{W}\;:\tilde M ^{-1}<u,v<\tilde M\right\}.$$
If the system has unique solution $W_*$ in $\mathbb{B}(\tilde M),$ then application of the homotopy invariance of the topological degree gives
\begin{equation}\label{deg}
    \text{deg}\left(\overline{\mathcal{T}}(0,\cdot),0,\mathbb{B}(\tilde M)\right)=\text{deg}\left(\overline{\mathcal{T}}(1,\cdot),0,\mathbb{B}(\tilde M)\right).
\end{equation}
Also, both the equations $\overline{\mathcal{T}}(0,W)=0$ and $\overline{\mathcal{T}}(1,W)=0$ have unique positive solution $W_*$ in $\mathbb{B}(\tilde M)$. Now,
$$\text{deg}\left(\overline{\mathcal{T}}(0,\cdot),0,\mathbb{B}(\tilde M)\right)=\text{index}(\overline{\mathcal{T}}(0,W_*)=(-1)^0=1$$ and 
 $$\text{deg}\left(\overline{\mathcal{T}}(1,\cdot),0,\mathbb{B}(\tilde M)\right)=\text{index}(\overline{\mathcal{T}}(1,W_*)=(-1)^\xi=-1.$$
 This contradicts (\ref{deg}) 
 and therefore, $\overline{\mathcal{T}}(1,W)$ has at least one positive solution other than $W_*$. Since the system (\ref{pde}) has unique homogeneous steady state $W_*$, the system  (\ref{pde}) has at least one positive non-constant steady-sate solution in  $\mathbb{B}(\tilde M)$.  
\end{proof}
\section{Continuation of stationary solution}\label{SDNC}


In the previous section, we have discussed the existence and the non-existence of non-constant solutions. Now, we are interested to find  the origin of various non-constant solutions that include mode-dependent Turing and localized solutions and their continuation. The stationary solution of the system \eqref{pde} can be analyzed by converting the system (\ref{PSS}) into a system of four first-order ordinary differential equations:
$$ 
\frac{du}{dx}=w,\;\frac{dv}{dx}=z,\; \frac{dw}{dx}=-F_1(u,v),\;\frac{dz}{dx}=-\frac{1}{d}F_2(u,v).
$$
The corresponding linearized system at $(u_*,v_*,0,0)$ can be written as
\begin{equation*}
    \frac{dZ}{dx}=BZ, \label{lins}
\end{equation*}
where $$Z=\begin{pmatrix}u\\v\\w\\z \end{pmatrix}\text{ and }B=\begin{bmatrix}
0& 0 &1& 0\\0& 0& 0& 1\\-a_{10}&-a_{01}&0&0\\-b_{10}/d&-b_{01}/d&0&0
\end{bmatrix}.$$
The characteristic equation of the matrix $B$ is given by 
\begin{equation}\label{nonconsev}
    H(\lambda^2):=d\lambda^4+(da_{10}+b_{01})\lambda^2+\mathrm{D}=0.
\end{equation}
$H(\lambda^2)$ is a quadratic polynomial in $\lambda^2.$  The zeros of \eqref{nonconsev} are always symmetric about real axis and imaginary axis in the complex plane. At the Turing bifurcation and BD transition thresholds, (\ref{nonconsev}) has a pair of eigenvalues with multiplicity two. These occur when
\begin{equation}\label{tureqn}
\mathrm{D}=\frac{(da_{10}+b_{01})^2}{4d}
\end{equation}
and $H(\lambda^2)$ reduces to
$$H(\lambda^2)=(\lambda^2-K)^2, \text{ where } K=\frac{da_{10}+b_{01}}{2d}.$$
Note that a Turing bifurcation corresponds to $K<0$, and a BD transition corresponds to $K>0$ \cite{champneys1998homoclinic,belyakov1997abundance,al2021unified}. The threshold values of Turing bifurcation and BD transition satisfy \eqref{tureqn}. Now, we verify these results with numerical  simulation and numerical  continuation.

We keep the parameter values $\eta=0.1,$ $\alpha=0.07,$ $\beta=0.2$ and $\gamma=1.2$ fixed in all the calculations. We consider $\sigma$ as a bifurcation parameter. The corresponding Turing bifurcation threshold and BD transition threshold are  $\sigma_T=1.861$ and  $\sigma_{BD}=2.098$ respectively. 
Both the bifurcation thresholds   are marked in  Fig. \ref{ev}(b)  along with $\vert\vert u\vert\vert_2$ of the homogeneous steady state, where   
$$\vert\vert u\vert\vert_2=\sqrt{ \int_{\Omega} (u^2)dx}.$$
Various localized structured solutions, which are different from Turing solution, emerge from the BD transition point. 

\begin{figure}[!t]
\begin{subfigure}[b]{.45\textwidth}
  \centering
\includegraphics[scale=0.5]{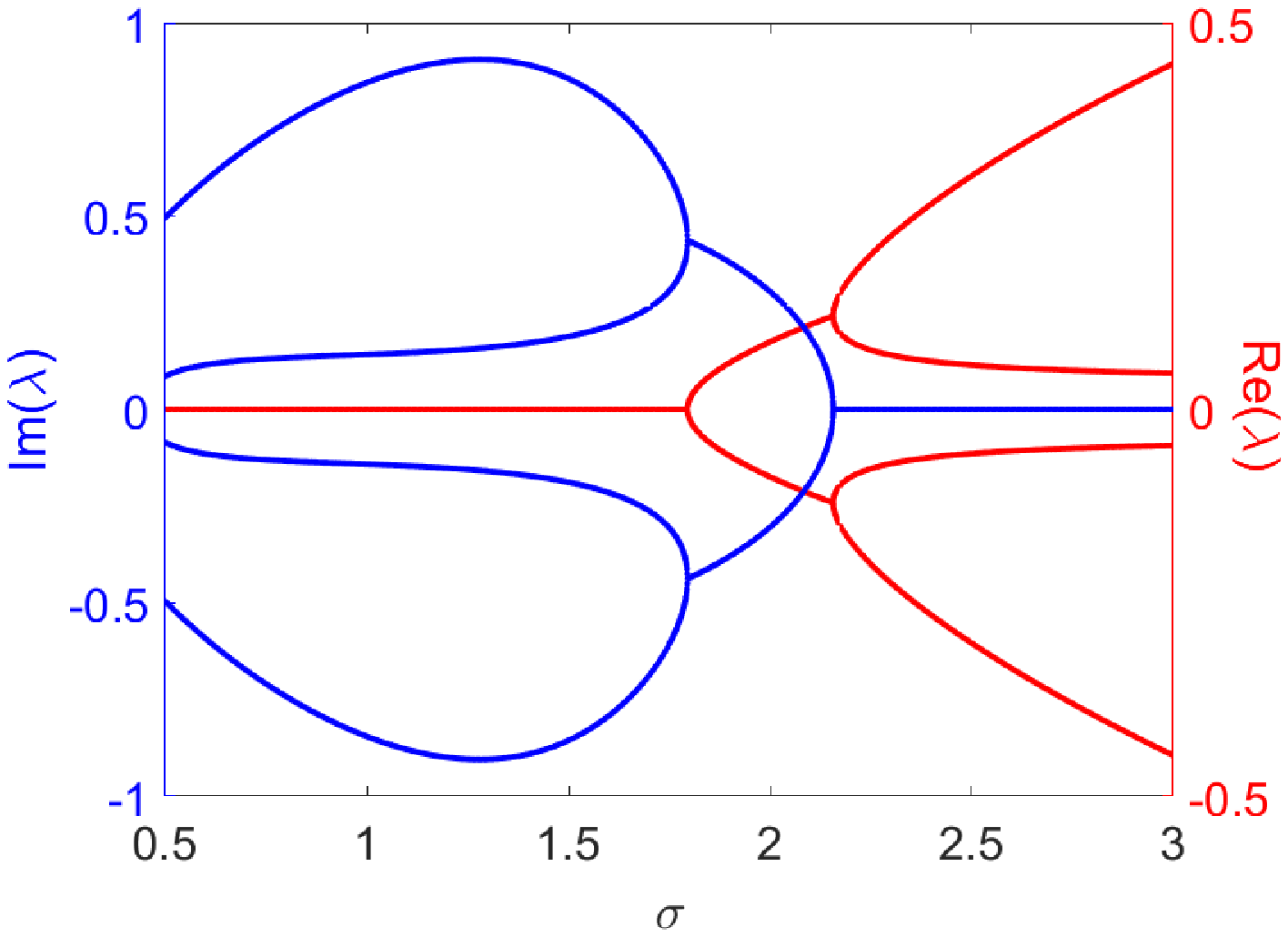}\\ 
\caption{}
\end{subfigure}
\hspace*{0.5cm}
\begin{subfigure}[b]{.45\textwidth}
  \centering
\includegraphics[scale=0.5]{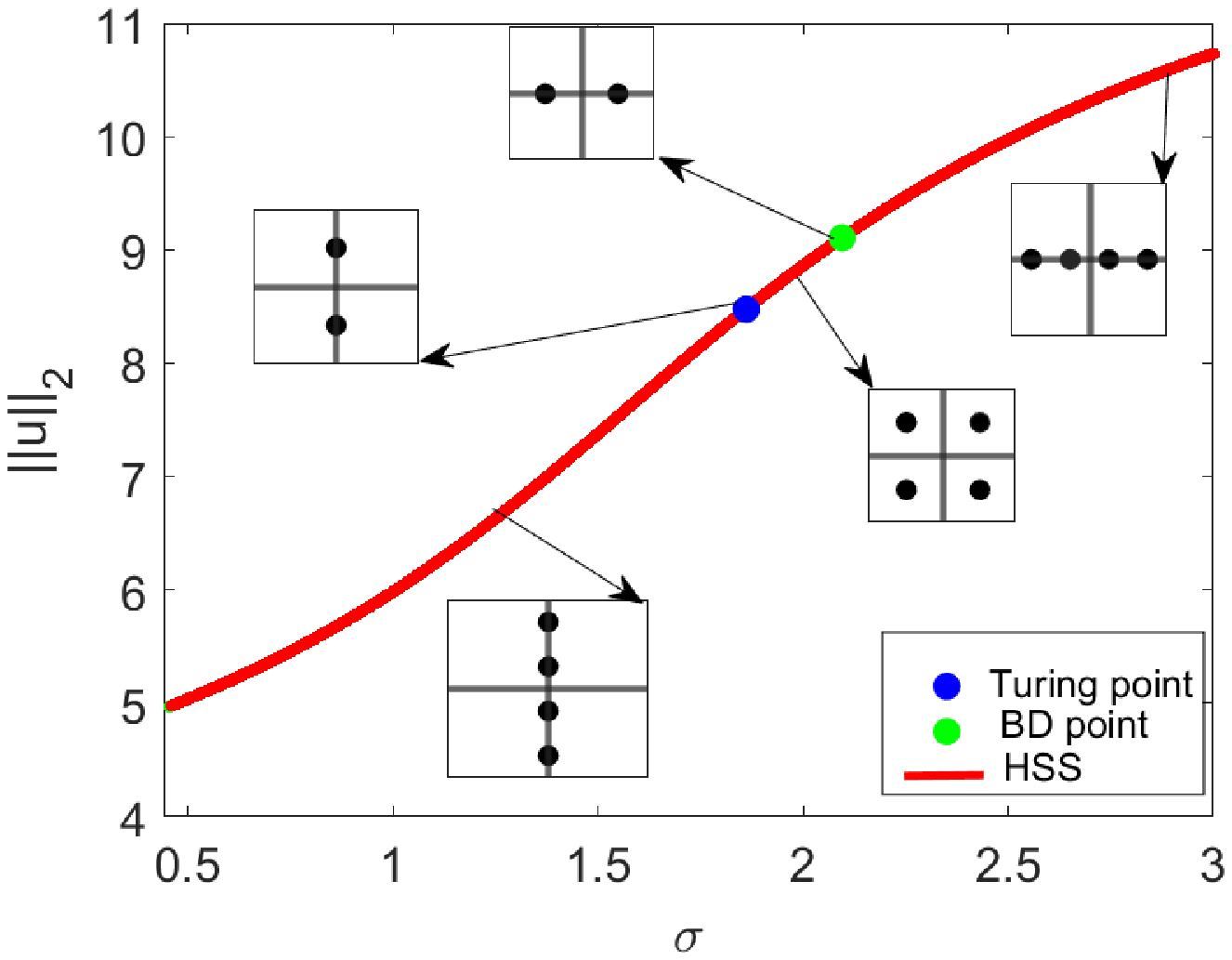}\\ 
\caption{}
\end{subfigure}
\caption{(a) Real and imaginary part of eigenvalues of the characteristic equation  (\ref{nonconsev}). (b) Locations of the eigenvalues along with coexisting homogeneous steady state (HSS) branch. Other parameter values are $\alpha=0.07,\beta=0.2, \gamma=1.2, \eta=0.1$ and $d=46$.}
\label{ev}
\end{figure}

\begin{figure}[!ht]
\begin{subfigure}[b]{.9\textwidth}
  \centering
  \centering
 \includegraphics[width=10cm,height=8cm]{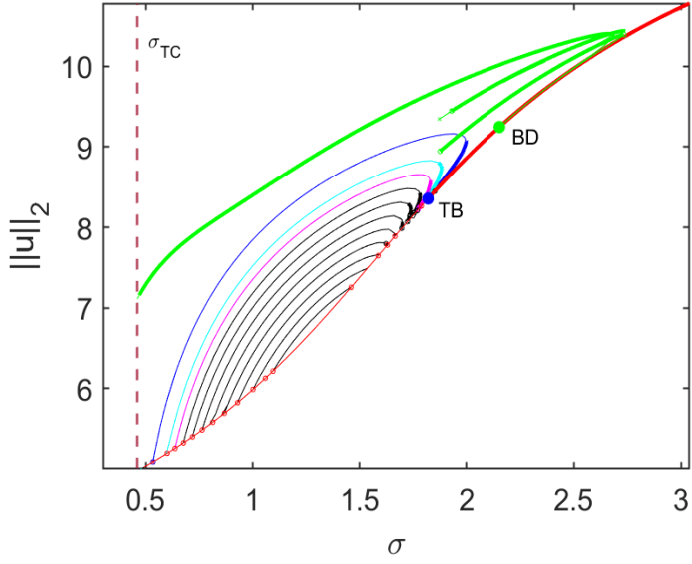}\\ 
\caption{}
\end{subfigure}\\
\begin{subfigure}[b]{.48\textwidth}
   \centering
\includegraphics[width=9cm,height=7cm]{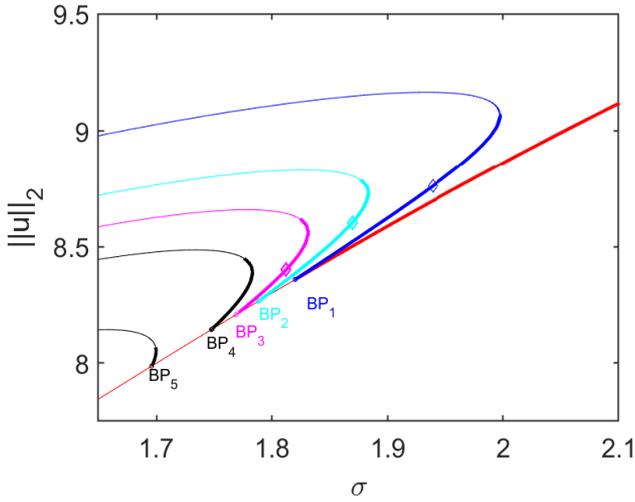}\\ 
\caption{}
\end{subfigure}
\begin{subfigure}[b]{.48\textwidth}
   \centering
\includegraphics[width=8cm,height=6cm]{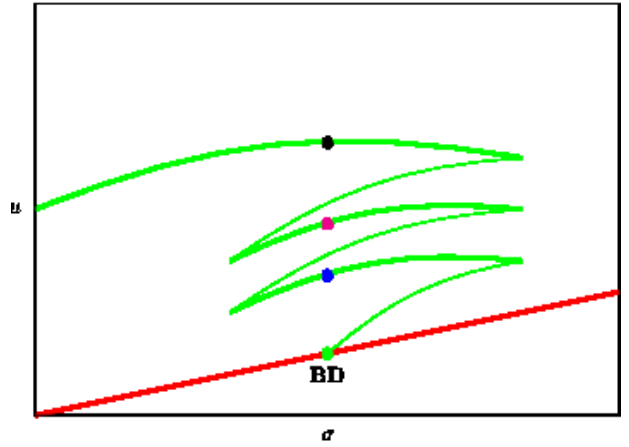}\\ 
\caption{}
\end{subfigure}

\caption{ (a) Bifurcation diagram of stationary Turing  patterns and localized patterns. Here  red color curve represents the homogeneous branch 
and  green color curves represent localized patterns. Other color curves represent Turing stationary branches of different modes. Solid color curves represent the stable branches and light color curves represent unstable branches. (b) A zoom version of the bifurcation diagram for Turing mode solutions. (c) A schematic bifurcation diagram of localized pattern solutions.  Other parameter values are $\eta=0.1,\alpha=0.07,\beta=0.2, \gamma=1.2$ and $d=46.$}  
\label{tur}
\end{figure}

To track a stationary solution against a bifurcation parameter, we use the Matlab software package pde2path \cite{uecker2014pde2path}. We consider the spatial-domain size $L=200$, and diffusion parameter $d=46$. Numerical continuation of Turing mode solutions and localized structure solutions are plotted in Fig. \ref{tur}(a). Different modes of Turing solutions  emerge from different  points (BP) [see Fig. \ref{tur}(b)]. The BP points can be found by solving for $\sigma$ (implicitly) in the following equation 
\begin{equation}
  d\left(\frac{n\pi}{L}\right)^4+(da_{10}+b_{01})\left(\frac{n\pi}{L}\right)^2+\mathrm{D}=0,
\end{equation}
where $n$ is the corresponding unstable Turing mode. For example, BP$_1$, BP$_2$, and BP$_3$ correspond to $n=19,20$ and $21$ respectively. We have shown 19-, 20- and 30- mode Turing solutions and localized solutions  in the Appendix \ref{Ap2}.

In Fig. \ref{tur}, we observe existence of multiple branches of various Turing mode solutions and localized solutions. The parameter value $\sigma=1.86$  satisfies Turing instability criteria. However, the system \eqref{pde} does not produce a stable Turing solution when a numerical simulation is performed with small amplitude spatial heterogeneity around the coexisting steady state $E_*$ as an initial condition.  The evolution of numerical solution shows a long oscillatory transition between a localized solution and Turing solution which finally settles down into a steady localized solution (see Fig. \ref{tran}). Thus we conclude that the localized solution dominates the Turing solution.
\begin{figure}[!ht]
 \begin{subfigure}[b]{.32\textwidth}
 \centering
\includegraphics[scale=0.38]{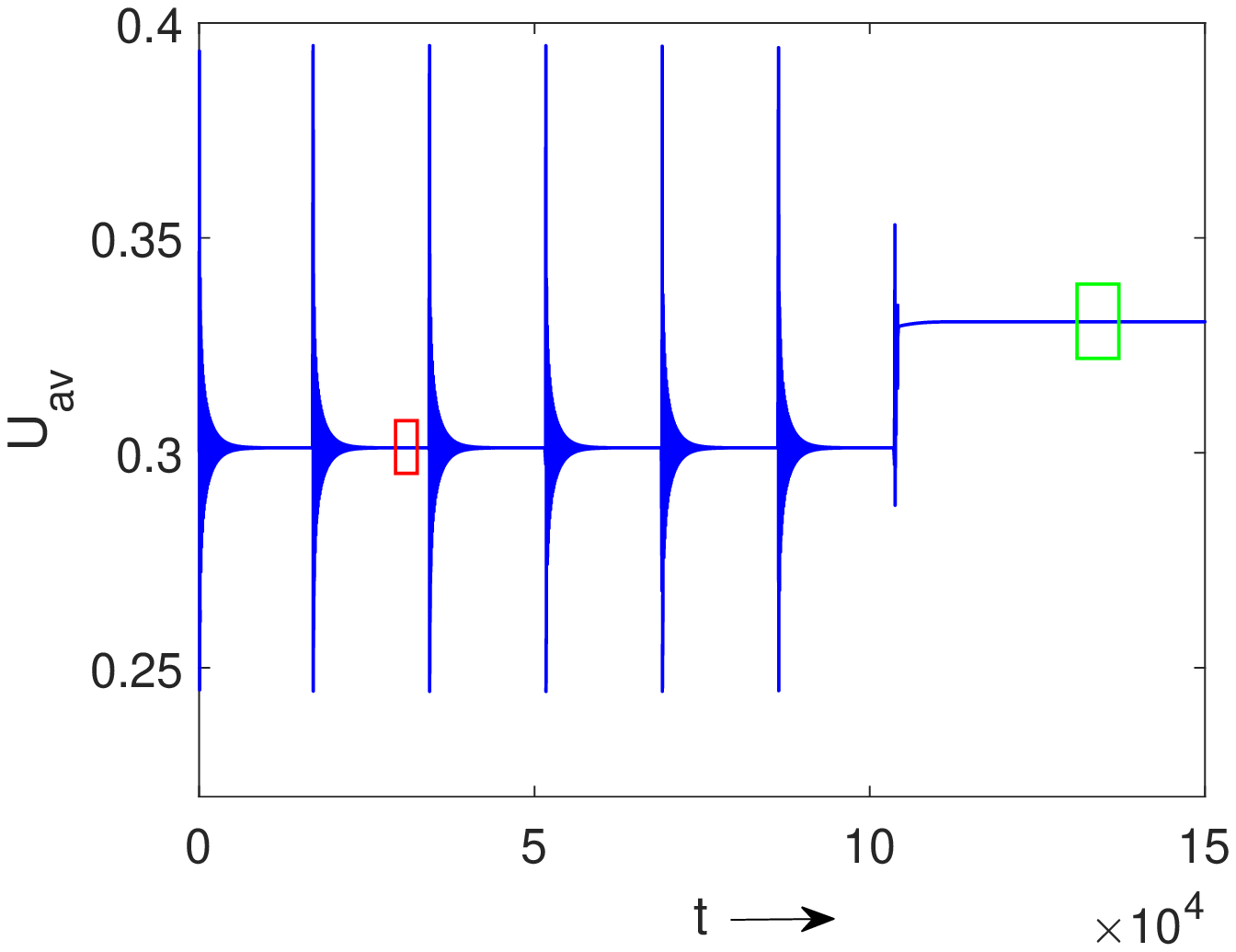}\\ 
 \caption{}
  \end{subfigure}
  \begin{subfigure}[b]{.32\textwidth}
 \centering
\includegraphics[scale=0.38]{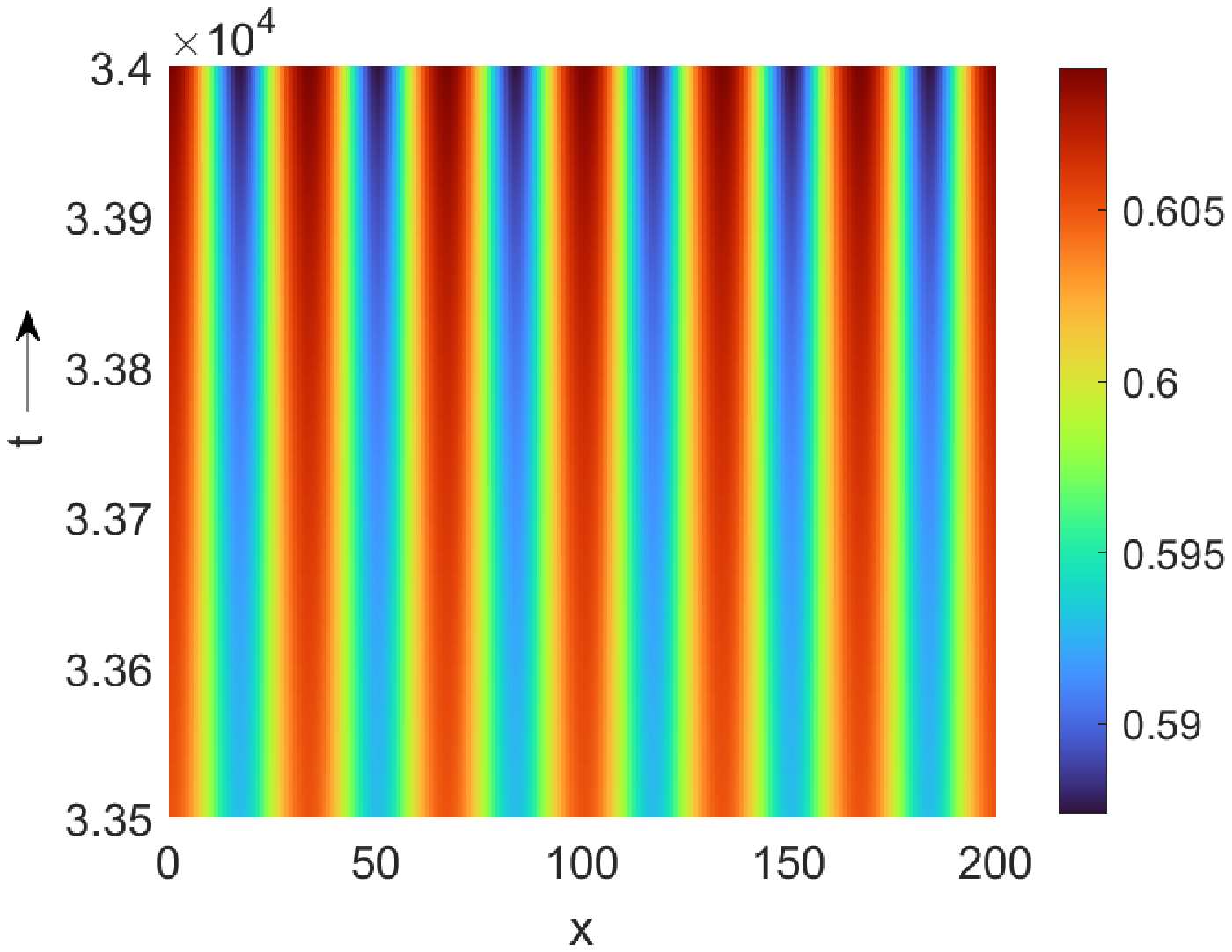}\\ 
 \caption{}
  \end{subfigure} 
 \begin{subfigure}[b]{.32\textwidth}
 \centering
\includegraphics[scale=0.38]{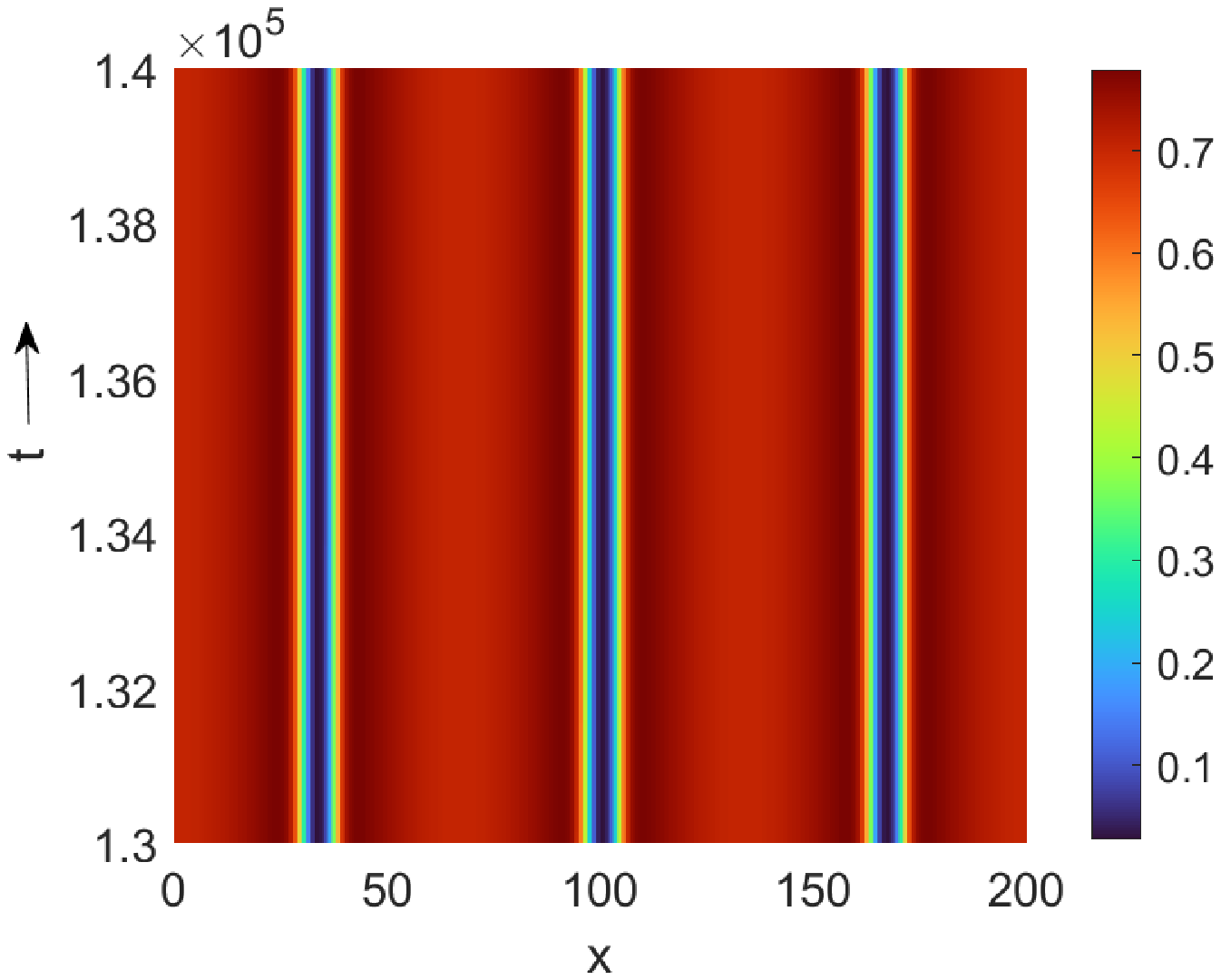}\\ 
 \caption{}
  \end{subfigure}
 \caption{Evolution of solution $u$: (a) spatial average of $u$ against time, (b) space-time plot corresponding to the red rectangle [see (a)] of the time-series plot, (c) space-time plot corresponding to the green rectangle [see (a)] of the time-series plot. Here,  parameter values are $\alpha=0.07,\beta=0.2, \gamma=1.2,\eta=0.1,\sigma=1.86$ and $d=46$.} 
\label{tran}
\end{figure}

\section{Time varying solutions}\label{dynamic}
  Here, we study some time-dependent solutions, namely, travelling wave, spatio-temporal chaos and  moving pulse solutions.
\subsection{Travelling wave solution}\label{TRWS}
Heterogeneous stationary solutions are obtained with small amplitude heterogeneous  perturbation around the unstable coexisting homogeneous steady-states. However, the introduction of predator over a small domain with the rest filled with prey population reveals invasive spread of the predators. We have observed that the coexisting homogeneous steady-state $E_*(u_*,v_*)$  may be stable or unstable and  the predator free homogeneous steady-state $E_1(u_1,0)$ is always saddle. Consider the case in which $E_*$ is a stable homogeneous steady-state.  Here, we investigate the  stable and unstable regions of the travelling wave solution  $$
u(x,t)=U(\xi)\quad\mbox{and}\quad v(x,t)=V(\xi)
$$ 
connecting the steady states $E_1$ and $E_*,$ where $\xi=x+ct$ and $c$ is the wave speed. Now, $\big(U(\xi),V(\xi)\big)$ satisfy 
\begin{subequations}
\begin{equation}
 c \frac{d U}{d \xi} =\frac{d^2 U}{d \xi^2}+F_1(U,V),
\end{equation}
\begin{equation}
c \frac{d V}{d \xi}=d\frac{d^2 V}{d \xi^2} +F_2(U,V),\end{equation}
\label{travel}
\end{subequations}
along with 
$$
\big(U(-\infty),V(-\infty)\big)=E_1 \text{ and } \big(U(\infty),V(\infty)\big)=E_*.
$$

In terms new variables defined by \cite{huang2016geometric}
\begin{alignat*}{4}
&X(t)\,=\,U(ct),\,\,Y(t)\,=\,\frac{1}{c}\left(cU(ct)-U'(ct)\right),\,\,\\
&W(t)\,=\,V(ct),\,\,
Z(t)\,=\,\frac{1}{c}\left( cV(ct)-dV'(ct)\right),
\end{alignat*}
the system (\ref{travel}) transforms into a system four first order ordinary differential equations: 
\begin{equation}
X'=c^2(X-Y),\;
Y'=F_1(X,W),\;
W'=\frac{c^2}{d}(W-Z)),\;
Z'=F_2(X,W).
\label{TRA}
\end{equation}
\color{black}
    The steady states of the system (\ref{TRA}) corresponding to $E_1$ and $E_*$ are $\mathcal{E}_1(u_1,u_1,0,0)$ and $\mathcal{E}_*(u_*,u_*,v_*,v_*).$ Now, we construct a wedge-shaped region in $\mathbb{R}^4$  such that  $\mathcal{E}_*$ lies in the interior of this region, and look for a solution connecting the steady states $\mathcal{E}_1$ and $\mathcal{E}_*$ confined within this region. If $0<X<u_1$ and $W>0,$ then
$$-W<Z'=F_2(X,W)<(\gamma-1)W.
$$
Using the above inequality, our first task is to compare the vector field of the system (\ref{TRA}) with the following planner system 
\begin{equation} \label{pl1}
    W'=\frac{c^2}{d}(W-Z)) \text{ and } Z'=-W.
\end{equation} 

\begin{prop}
The system (\ref{pl1}) has a strictly monotone decreasing solution $W_1(t),Z_1(t)$ for $t\in (-\infty,\infty)$ which satisfy  $Z_1(t)>W_1(t)>0 $ and 
$$W_1(t)\rightarrow0, Z_1(t)\rightarrow0 \text{ as }t \rightarrow \infty,$$
$$W_1(t)\rightarrow \infty, Z_1(t)\rightarrow\infty \text{ as }t \rightarrow -\infty. $$
\end{prop}
\begin{proof}
The linear system (\ref{pl1}) possesses a negative eigenvalue $$\lambda_1=\frac{c^2-\sqrt{c^4+4dc^2}}{2d}$$
with an eigenvector $$v_1=(1,m)^T,\text{ where }m=\frac{c^2+\sqrt{c^4+4dc^2}}{2c^2}>1.$$
Thus, the system (\ref{pl1}) has a local one-dimensional stable manifold in the triangular region $$S=\big\{(W,Z):Z>W\geq0\big\}\cup(0,0).$$
If a solution $(W_1(t),Z_1(t))$ starts from a point $(W_0,Z_0)\in S,$ then both 
$W_1(t), \;Z_1(t)\rightarrow0 \text{ as }t \rightarrow \infty.$ 

 Further, we can easily observe that the vector field of the system (\ref{TRA}) on the side $\{Z=W\}$ of $S$ points vertically downward and the vector field in the side  $W=0$ points to the left horizontally.  Hence, the solution $\big(W_1(t),Z_1(t)\big)$ can be extended for all $t<0$ and $\big(W_1(t),Z_1(t)\big)$ remains in the triangular region $S$ for all $t<0$. Since $W_1'(t)=\frac{c^2}{d}(W_1-Z_1)<0$ and $Z_1'(t)<0$, we conclude that both  $W_1(t), Z_1(t)\rightarrow\infty \text{ as }t \rightarrow -\infty. $
\end{proof}
Since, $W_1(t)$ is a strictly decreasing function, $W_1$ has an inverse $W_1^{-1}:(0,\infty):\rightarrow \mathbb{R}$ such that $W_1(t)=W$ if and only if $t=W_1^{-1}(W)$ for all $W>0.$ Therefore, we can express the stable manifold of the system (\ref{pl1}) inside $S$ as the graph of the function $\zeta:[0,\infty)\rightarrow [0,\infty)$  by $$\zeta(W)=Z_1\big(W_1^{-1}(W)\big) \text{ for } W>0\; \text{ with } \zeta(0)=0.$$

Now we define a wedge-shaped region $\Sigma\in \mathbb{R}^4$  \cite{huang2016geometric} as follows: 
$$\Sigma=\left\{ (X,Y,W,Z):0\leq X\leq u_1,Y\in\mathbb{R},W\geq0, \frac{1}{2}W\leq Z\leq \zeta(W)\right\}.$$ 
The boundary of $\Sigma$ consists of surfaces $B_1-B_4$ and $C_3-C_5$ defined by 
\begin{alignat*}{4}  
 B_1&=\left\{ 0<X<u_1,Y\in\mathbb{R},W>0, Z=\zeta(W) \right\}, \\
 B_2&=\left\{ 0<X<u_1,Y\in\mathbb{R},W>0, Z=\frac{1}{2}W\right\}, \\
 B_3&=\left\{ X=u_1,Y<u_1,0< \frac{1}{2}W\leq Z\leq \zeta(W) \right\}, \\
 B_4&=\left\{ X=0,0<Y,0<\frac{1}{2}W\leq Z\leq \zeta(W)\right\}, \\
 C_3&=\left\{ X=u_1,Y\geq u_1,0< \frac{1}{2}W\leq Z\leq \zeta(W) \right\}, \\
 C_4&=\left\{ X=0,0\leq Y,0<\frac{1}{2}W\leq Z\leq \zeta(W)\right\}, \\
 C_5&=\left\{ 0\leq X\leq u_1,W=Z=0\right\}. 
\end{alignat*}
Let $\Phi(t,\textbf{p})=(X(t),Y(t),W(t),Z(t))$ with $\Phi(0,\textbf{p})=\textbf{p}\in\mathbb{R}^4$ be the flow of the system (\ref{TRA}). If $M={\frac{1}{d}\left(\frac {  \gamma u_1  }{\alpha+u_1}-1\right)},$ then the flow  $\Phi(t,\textbf{p})$ satisfies the following proposition \cite{huang2016geometric}: 
\begin{prop}
\begin{itemize}
    \item[(i)] If $c\geq 2d\sqrt{M},$  then the system (\ref{TRA}) has a positive bounded travelling wave $\Phi(t,\textbf{p})=(X(t),Y(t),W(t),Z(t))\in \Sigma$ with $\Phi(t,\textbf{p})\rightarrow \mathcal{E}_1 $ as $t\rightarrow-\infty$ and $\Phi(t,\textbf{p})\rightarrow \mathcal{E}_* $ as $t\rightarrow+\infty$.
    \item[(ii)] If $0<c<2d\sqrt{M}$, then the system (\ref{TRA}) does not have a non-negative, nontrivial travelling wave solution connected to $\mathcal{E}_1$.
\end{itemize}
\end{prop}

 Next, we verify the heteroclinic connection between $\mathcal{E_*}$ and $\mathcal{E}_1$ together with stable and unstable regions of monotonic and non-monotonic travelling waves using numerical simulations.

\subsubsection{Numerical Results}
The Jacobian of the system (\ref{TRA}) evaluated at $\mathcal{E}_1$ is 
$$J(\mathcal{E}_1)=\begin{pmatrix}
c^2 & -c^2 & 0 &0 \\
j_1 & 0 & j_2 &0 \\
0 & 0& \frac{c^2}{d} &-\frac{c^2}{d}\\
0 &0& j_3 &0 
\end{pmatrix},$$
where $$j_1=\sigma u_1(1-2 {u_1})<0,\;j_2=-{\frac {u_1}{\alpha+u_1}}<0\;\;\mbox{ and }\;\; j_3={\frac{1}{d}\left(\frac {  \gamma u_1  }{\alpha+u_1}-1\right)}>0.
$$
The eigenvalues of $J(\mathcal{E}_1)$ are 
$$\lambda_{1,2}=\frac{c^2\pm c\sqrt{c^2-4j_1}}{2},\;\,  \lambda_{3,4}=\frac{c^2\pm c\sqrt{c^2-4d^2j_3}}{2d},
$$
where $\lambda_{1,2}$ are always real but  $\lambda_{3,4}$ may be real or complex conjugate. The complex conjugate eigenvalues correspond to spiral solution around $\mathcal{E}_1$ which leads to negative population density. Hence, $\lambda_{3,4}$ must be real for the existence of travelling wave, which implies $c\geq 2 d \sqrt{j_3}$.  Hence, the minimum wave speed is $c_{min}=2 d \sqrt{j_3}$. 
The Jacobian of the system (\ref{TRA}) evaluated at $\mathcal{E}_*$ is given by
$$J(\mathcal{E}_*)=\begin{pmatrix}
c^2 & -c^2 & 0 &0 \\
a_{10} & 0 &a_{01} &0 \\
0 & 0& \frac{c^2}{d} &-\frac{c^2}{d}\\
b_{10} &0& b_{01} &0 
\end{pmatrix}.$$
We find that  two of the eigenvalues of $J(\mathcal{E}_*)$ have negative real parts. There is a heteroclinic orbit $\Phi(t )$ contained in stable manifold of $\mathcal{E}_*$ and unstable manifold of $\mathcal{E}_1,$ i.e., $\Phi(t )\rightarrow \mathcal{E}_1 $ as $t\rightarrow-\infty$ and $\Phi(t )\rightarrow \mathcal{E}_* $ as $t\rightarrow+\infty$. 
We have plotted that heteroclinic connection $\Phi(t )$ in Fig. \ref{manifold}.  We also plot the boundaries $B_1$ and $B_2$ in  Fig. \ref{manifold}(b) which shows that $\Phi(t )$ is contained in the wedged-shaped region $\Sigma.$
\begin{figure}[!ht]
\begin{subfigure}[b]{.48\textwidth}
  \centering
  \centering
\includegraphics[scale=0.5]{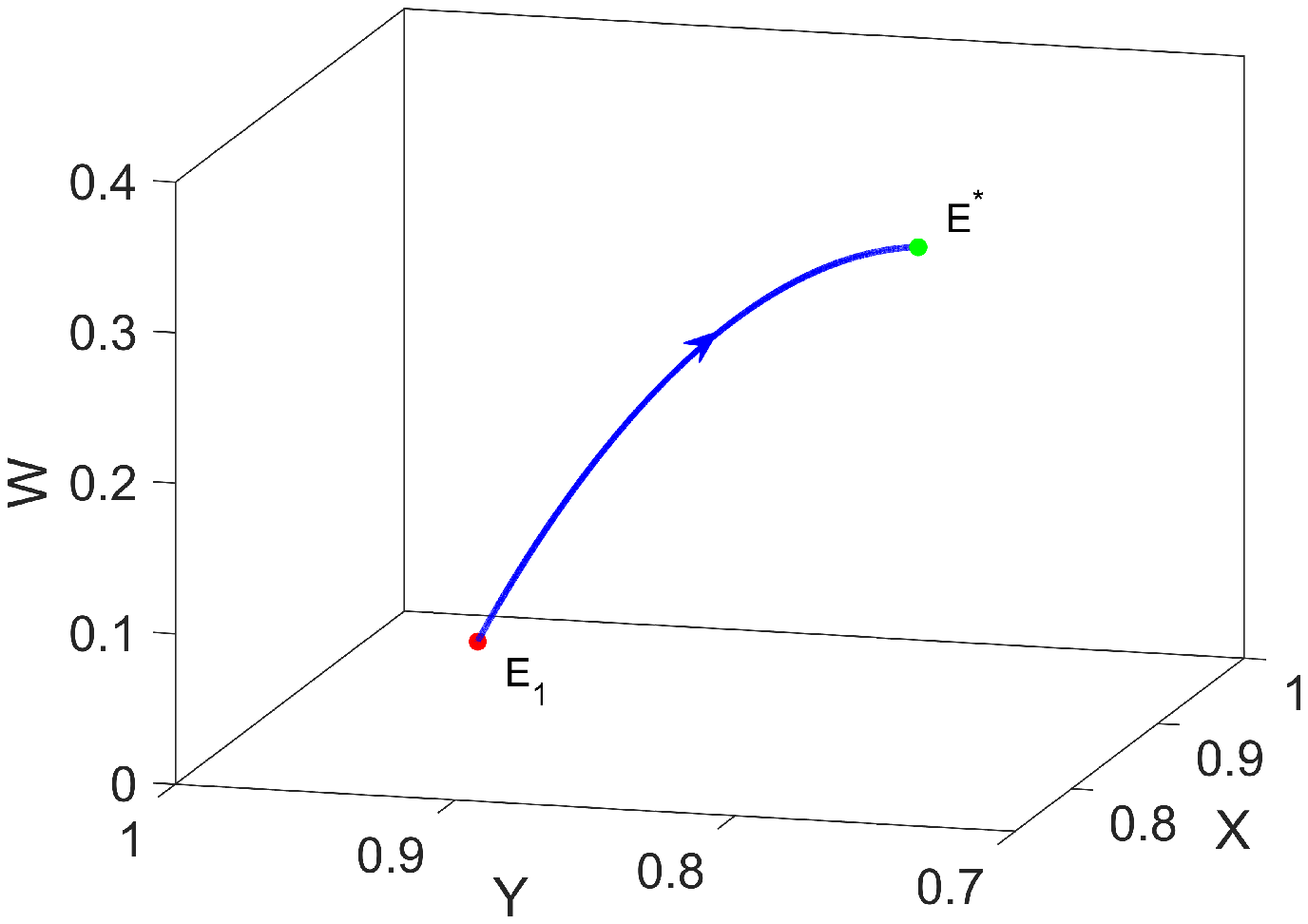}\\ 
\caption{}
\label{try}
\end{subfigure}
\begin{subfigure}[b]{.48\textwidth}
   \centering
\includegraphics[scale=0.5]{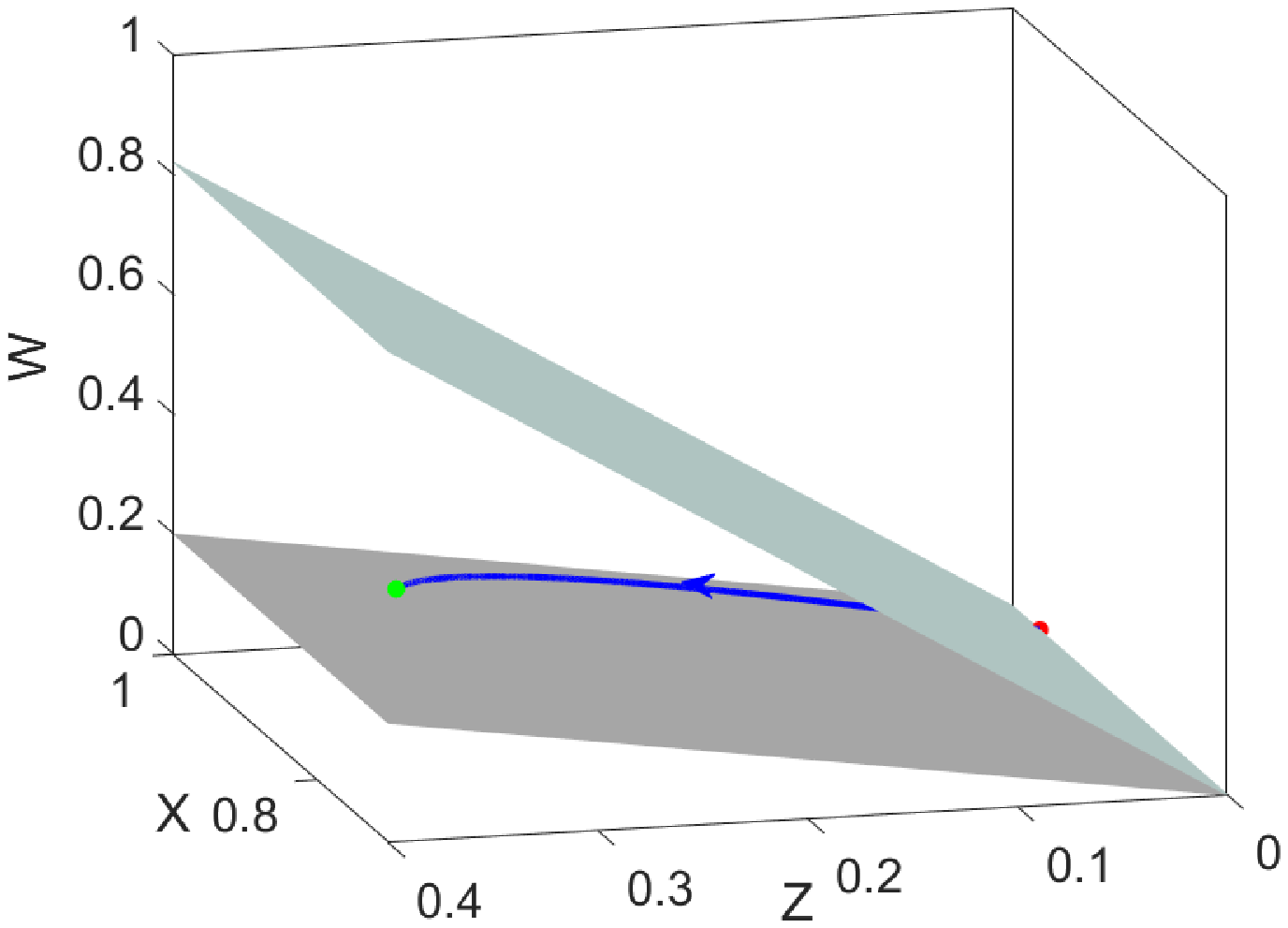}\\ 
\caption{}
\label{trz}
\end{subfigure}
\caption{ Heteroclinic connection of the system (\ref{TRA})  between $\mathcal{E}_1(0.9615,0.9615,0,0)$ and $\mathcal{E}_*( 0.7291, 0.7291,0.3791,0.3791)$ for $d=46$, $\sigma=2.7$ and $c=5.9$. (b) Same heteroclinic  connection  in between the planes $W=2Z$ and $Z=\zeta(W)$. }
\label{manifold}
\end{figure}

To capture  the travelling wave with numerical simulation, we fix the temporal parameters  $\eta=0.1,$ $ \alpha=0.07,\;\beta=0.2, \gamma=1.2,$  and consider the following initial conditions: 
\begin{equation}
    u(x,0)= \begin{cases}
  u_*  & 0\leq x<200 \\
  u_1 & 200\leq x\leq2000
\end{cases},\quad\mbox{and}\quad v(x,0)= \begin{cases}
  v_*  & 0\leq x<200 \\
  0 & 200\leq x\leq2000
\end{cases}.\;
\end{equation}
The eigenvalues of $J(\mathcal{E}_*)$ determine the nature of the travelling wave.  If all the eigenvalues of the $J(\mathcal{E}_*)$ are real, then the system (\ref{pde}) shows monotonic travelling wave. On the other hand, it exhibits  non-monotonic travelling wave if some of the eigenvalues of $J(\mathcal{E}_*)$  are complex. These complex eigenvalues are responsible for the oscillation in the non-monotonic travelling wave-front. Snapshots of monotonic and non-monotonic travelling waves at  different times $t=100,200$ and $300$  are plotted in Fig. \ref{monotone}. The shape of the solution profiles  are similar for both the monotonic and non-monotonic cases with the advancement of time. The minimum wave propagation speed $c_{min}=2d\sqrt{j_3}\approx 4.68$ 
for $\sigma=2.7$ and $d=46$.  Comparing the locations of travelling wave at two different times in Fig.  \ref{monotone}(a), we calculate the speed of the travelling wave to be $4.71$ approximately. Hence, the numerical value and analytical value of the speed of propagation are in close agreement. 

We have also plotted a diagram (see Fig. \ref{trwbif}) in the $\sigma$-$c$ parametric plane for the system \eqref{TRA}, which shows the existence and nature of the travelling wave solution of the system (\ref{pde}). The system (\ref{TRA}) does not admit any travelling wave solution  $\Phi(t,\textbf{p})$ for $c<2d\sqrt{j_3}.$ For   $c\geq 2d\sqrt{j_3}$, we obtain non-monotonic travelling wave when $\sigma>2.68$ and  monotonic travelling wave when $\sigma_H<\sigma<2.68.$

\begin{figure}[!t]
\begin{subfigure}[b]{.48\textwidth}
  \centering
  \centering
\includegraphics[scale=0.5]{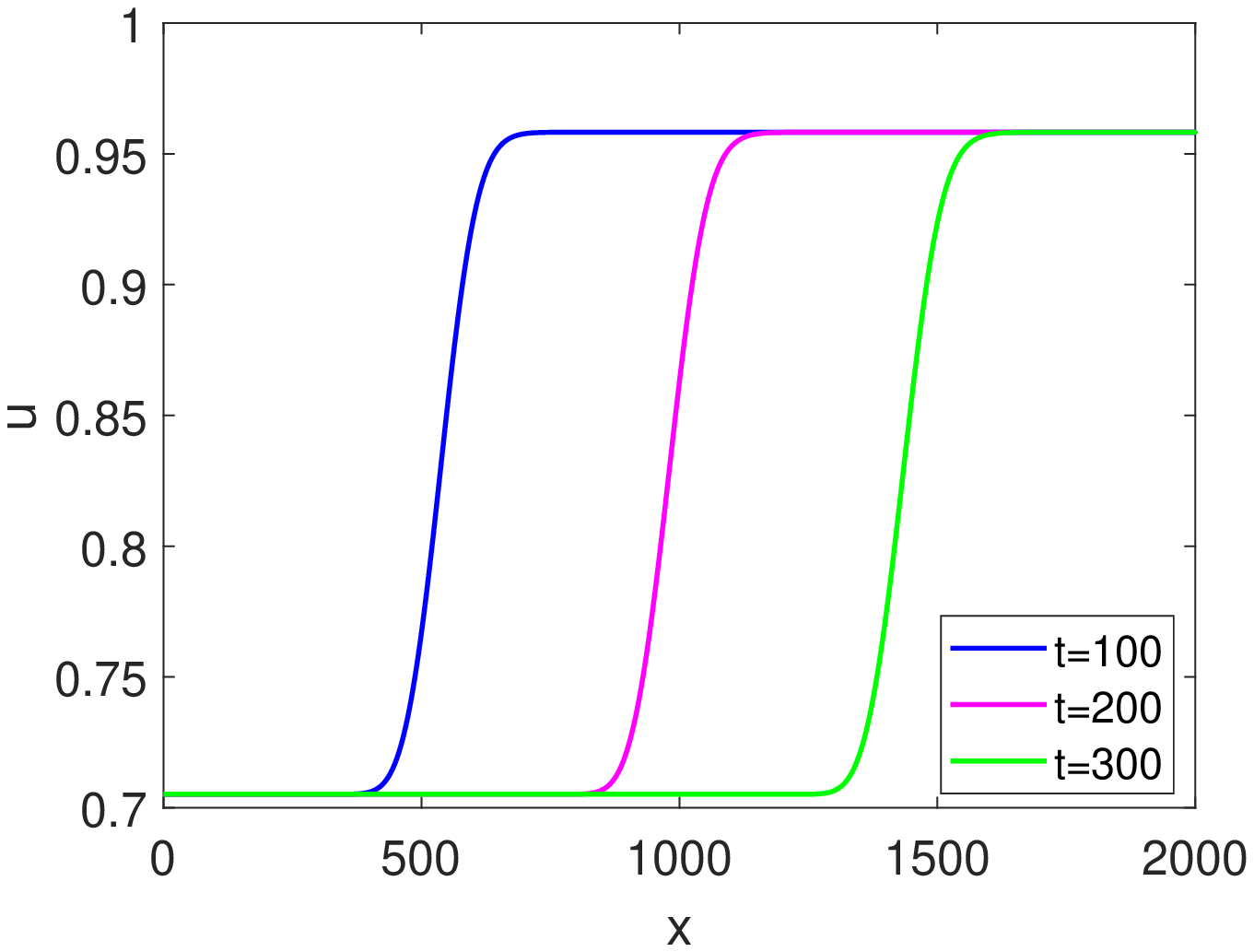}\\ 
\caption{}
\label{mu}
\end{subfigure}
\begin{subfigure}[b]{.48\textwidth}
  \centering
  \centering
\includegraphics[scale=0.5]{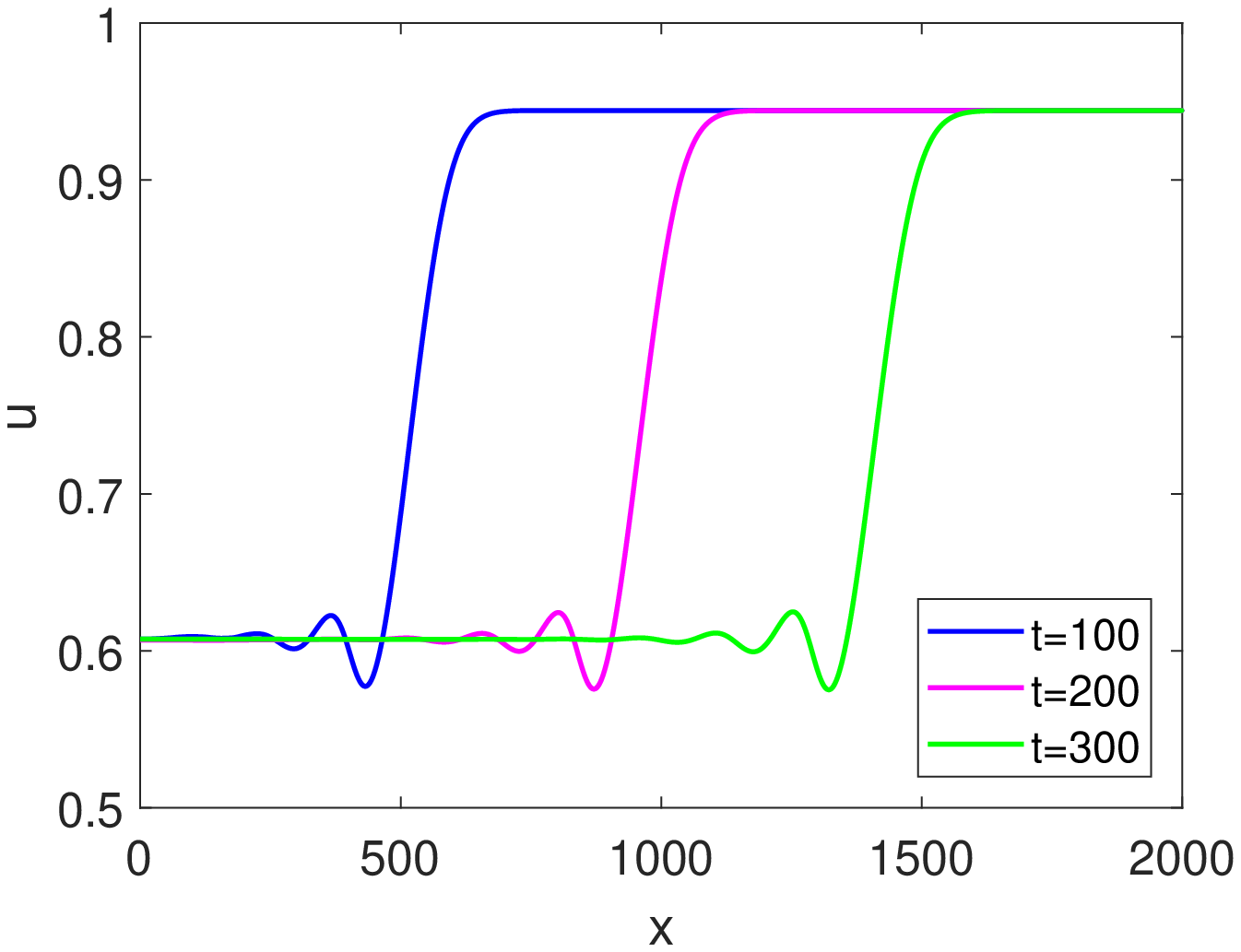}\\ 
\caption{}
\label{Nonmu}
\end{subfigure}
\caption{Travelling wave solution: (a) monotonic profile for $\sigma=2.7,$ (b) non-monotonic profile for $\sigma=1.9$. Other parameter values are $d=46, \eta=0.1, \alpha=0.07, \beta=0.2$ and $\gamma=1.2.$}
\label{monotone}
\end{figure}
\begin{figure}[!ht]
\centering
\includegraphics[scale=0.55]{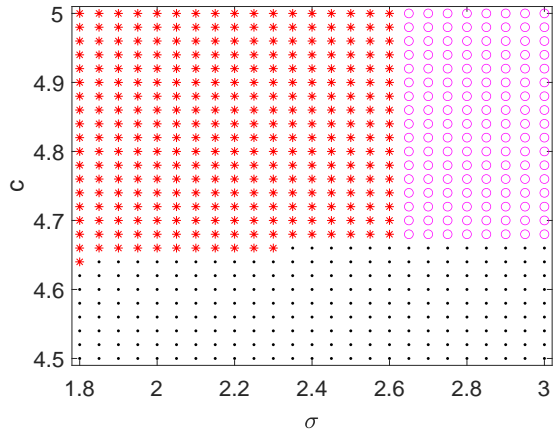}\\
\caption{Existence and nature of travelling wave solution in the $\sigma$-$c$ parametric plane. Black dots indicate region of non-existence of travelling wave solution. Red asterisks and magenta circles denote non-monotonic  and monotonic travelling waves respectively.  Other parameter values are $\eta=0.1,$ $ \alpha=0.07,\;\beta=0.2, \;\gamma=1.2$ and $d=46$.}
\label{trwbif}
\end{figure}
\subsection{Spatio-temporal Chaos}

The temporal model \eqref{ode} shows bistability between trivial equilibrium point and coexisting periodic solution for  $\sigma_{Het}<\sigma<\sigma_H$. But, introduction of a small amplitude spatial heterogeneity around the unstable coexisting steady state leads to non-homogeneous non-stationary solution for small values of $d$. We take  parameter value $d=5$,  $\sigma=1.8,$ $\vert \Omega \vert=500,$ and the corresponding results  are shown in Fig. \ref{SPC}.  The plot of the spatial average of the prey population ($U_{av}$) against time shown in Fig. \ref{SPC}(a) reveals  the chaotic nature of the solution. This is further supported by Fig. \ref{SPC}(b), which shows  that the emerging pattern does not converge to any stationary steady state. Using XPPAUT \cite{ermentrout2003simulating}, we obtain  the largest Lyapunov exponent $\lambda_{max}=0.012>0,$ which confirms the chaotic nature of the dynamics. Thus, the prey and predator populations exhibit  spatio-temporal chaos for certain  parameter value in the Hopf region. 

\begin{figure}[!ht]
\begin{subfigure}[b]{.48\textwidth}
   \centering
\includegraphics[scale=0.5]{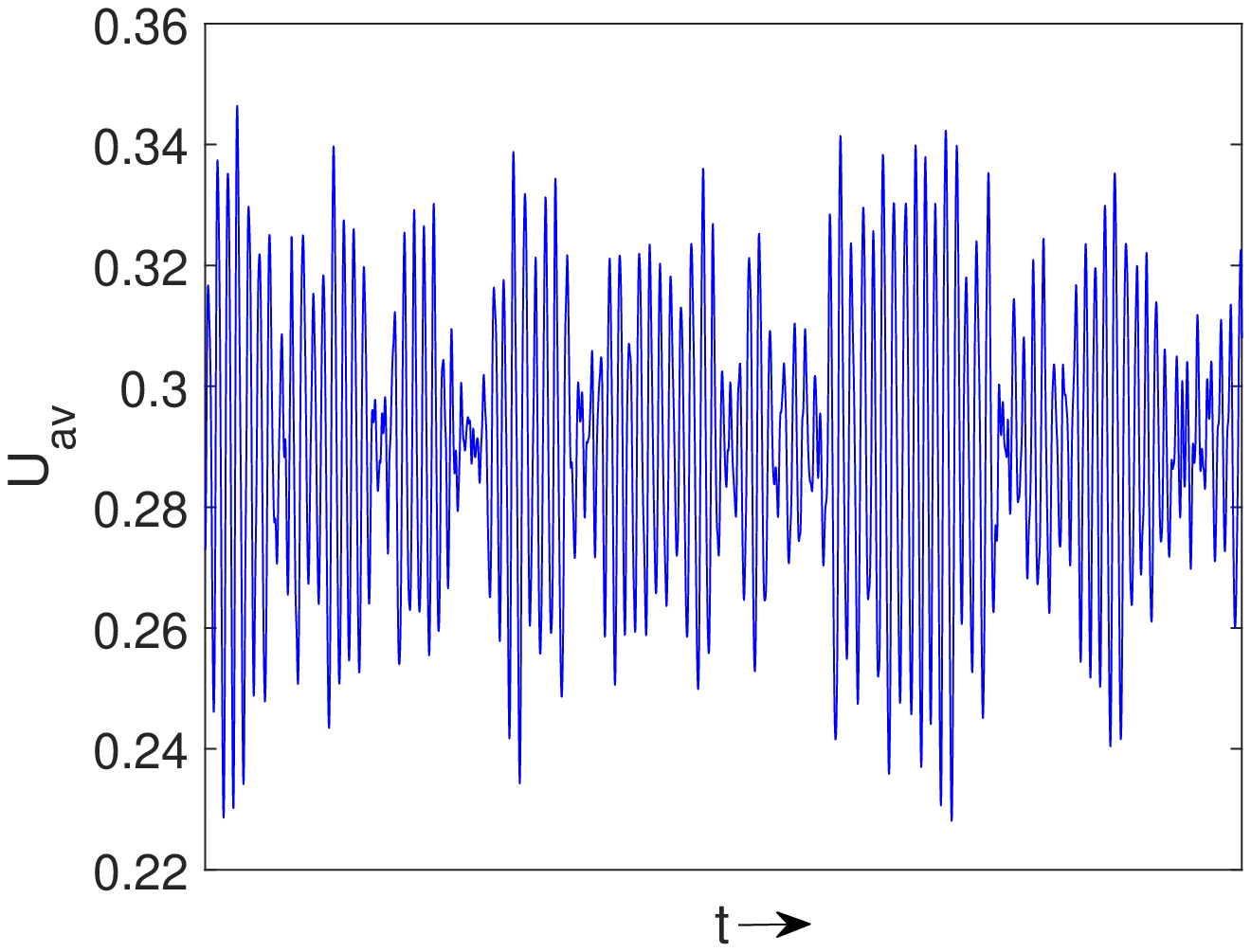}\\ 
\caption{}
\end{subfigure}
\begin{subfigure}[b]{.48\textwidth}
   \centering
\includegraphics[scale=0.5]{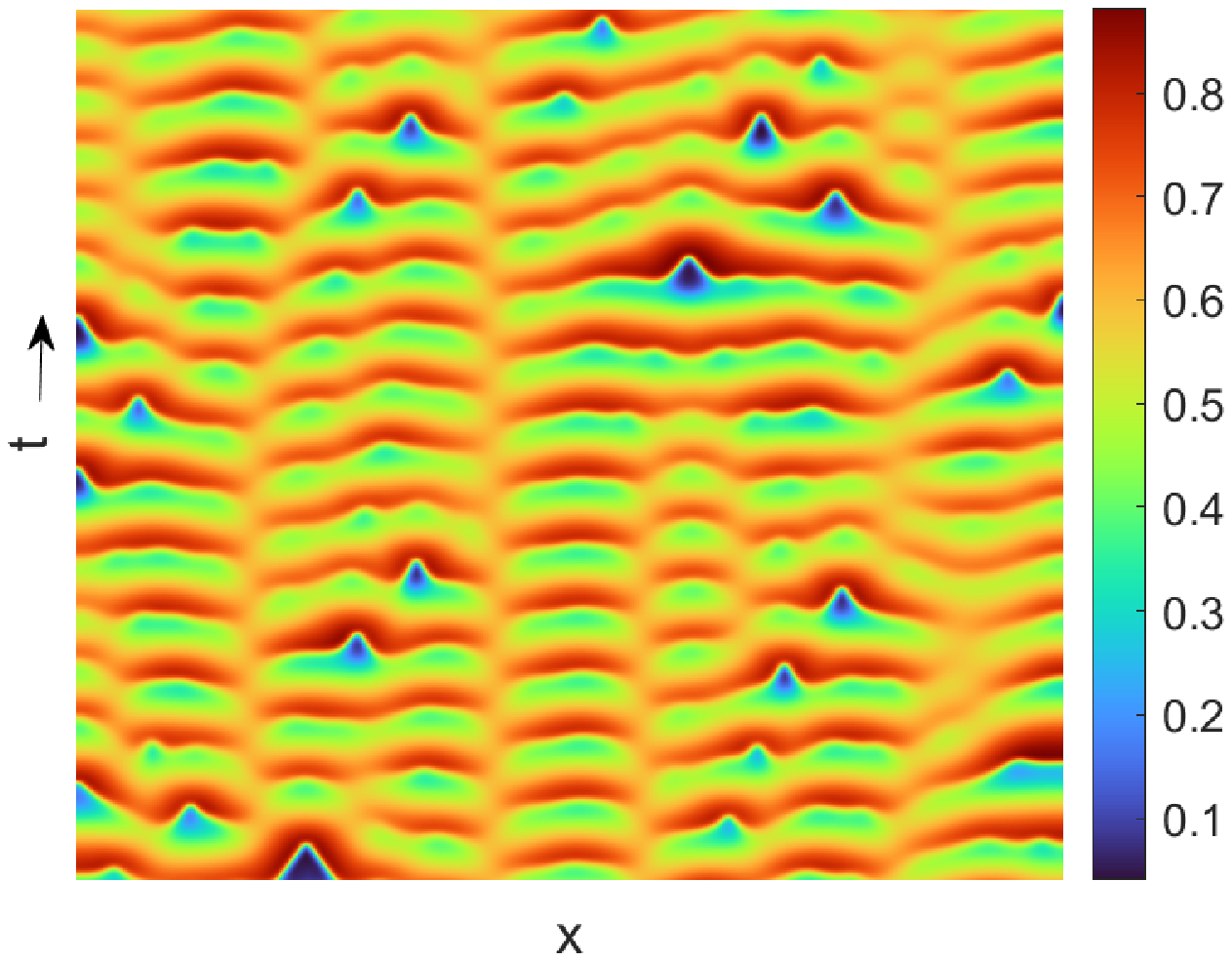}\\ 
\caption{}
\end{subfigure}

\caption{ 
Spatio-temporal chaos: (a) spatial average of the prey population against time, (b) space-time plot of the prey population distribution after the initial transient state.  Parameter values are $\alpha=0.07,\beta=0.2, \gamma=1.2,\eta=0.1$, $\sigma=1.8$ and $d=5.$}  
\label{SPC}
\end{figure}

\begin{figure}[!ht]
\begin{subfigure}[b]{.48\textwidth}
  \centering
\includegraphics[scale=0.5]{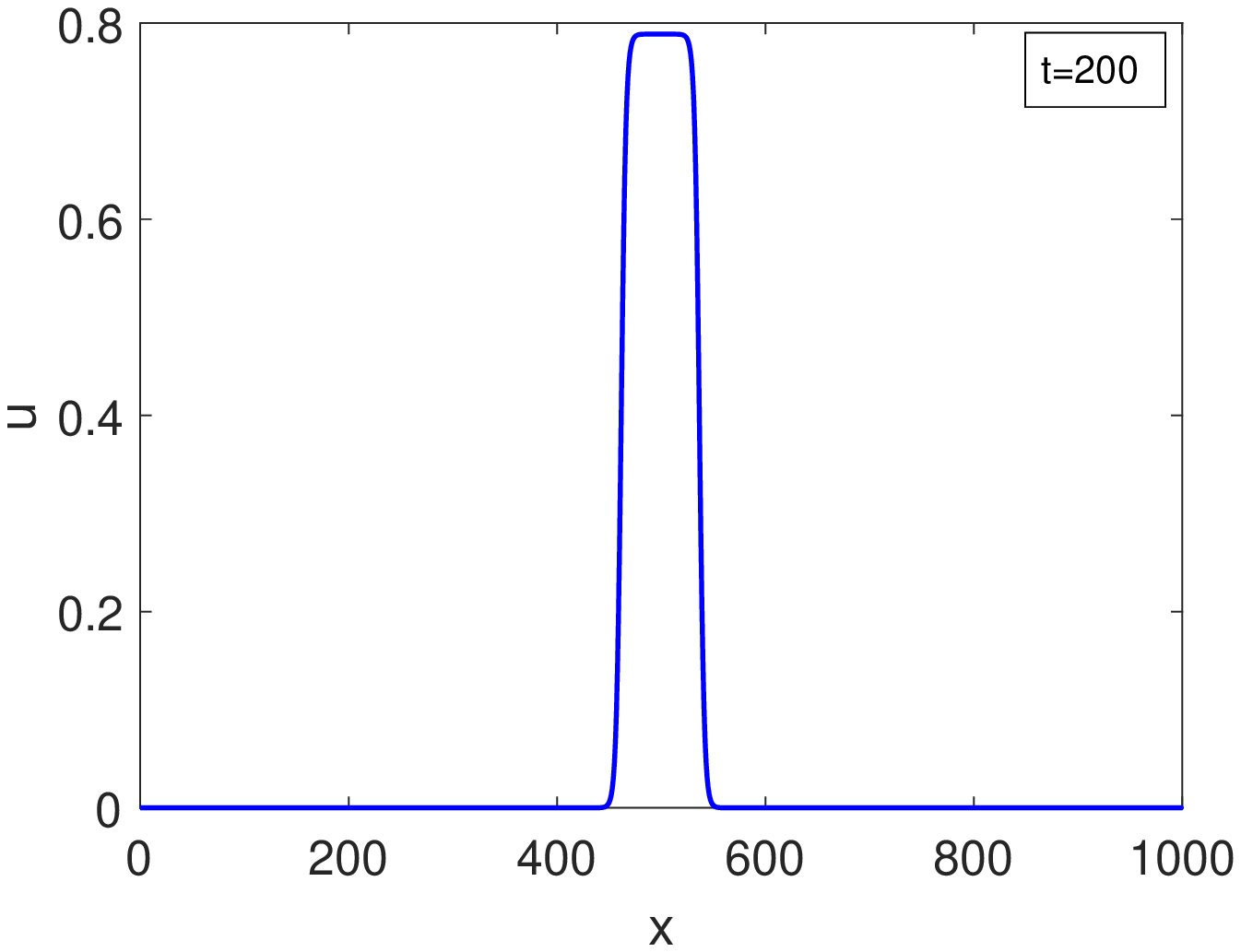}\\ 
\caption{}
\label{mp800}
\end{subfigure}
\begin{subfigure}[b]{.48\textwidth}
  \centering
\includegraphics[scale=0.5]{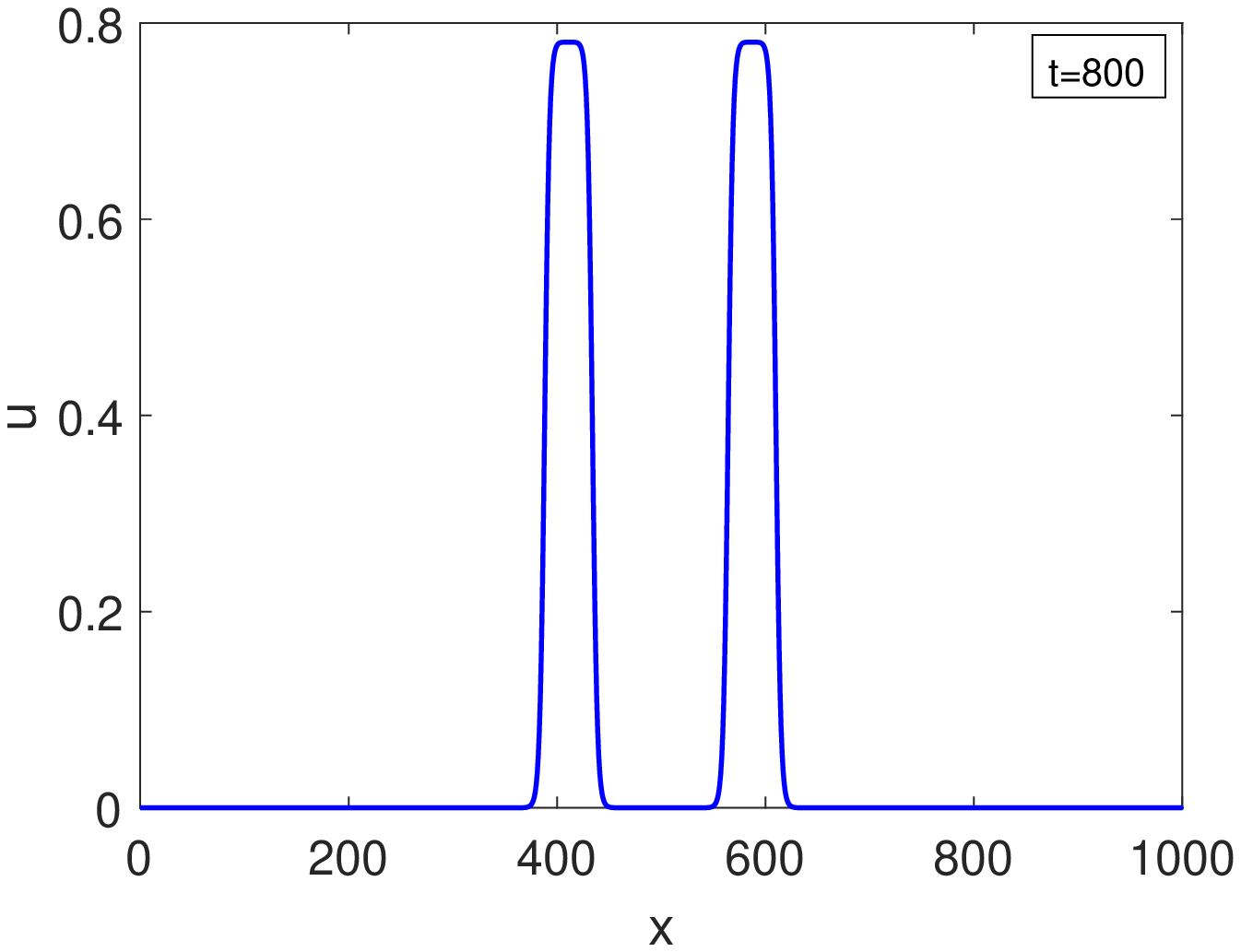}\\ 
\caption{}
\label{mp200}
\end{subfigure}
\begin{subfigure}[b]{.48\textwidth}
  \centering
\includegraphics[scale=0.5]{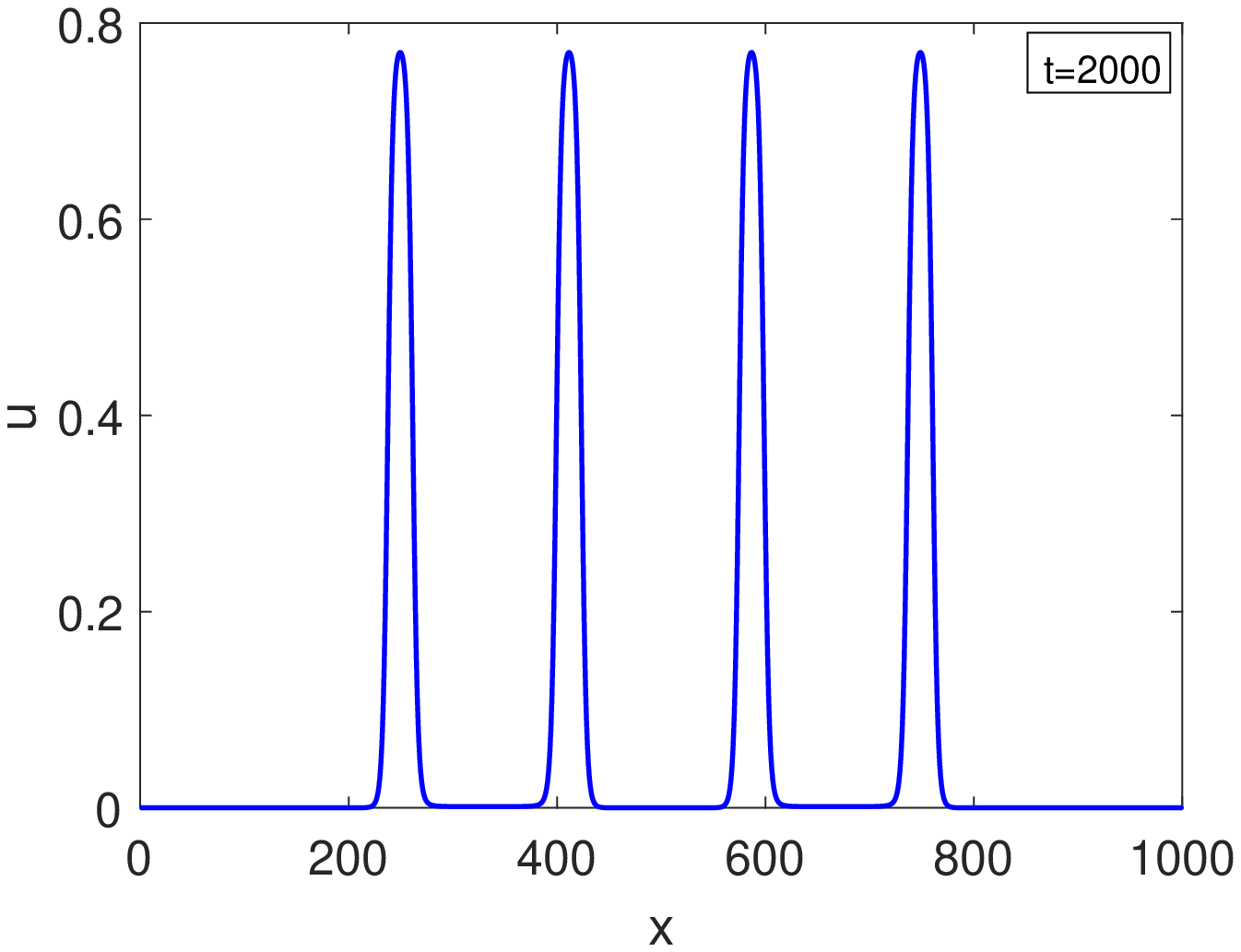}\\ 
\caption{}
\label{mp2000}
\end{subfigure}
\begin{subfigure}[b]{.48\textwidth}
  \centering
  \hspace*{0.5cm}
\includegraphics[scale=0.5]{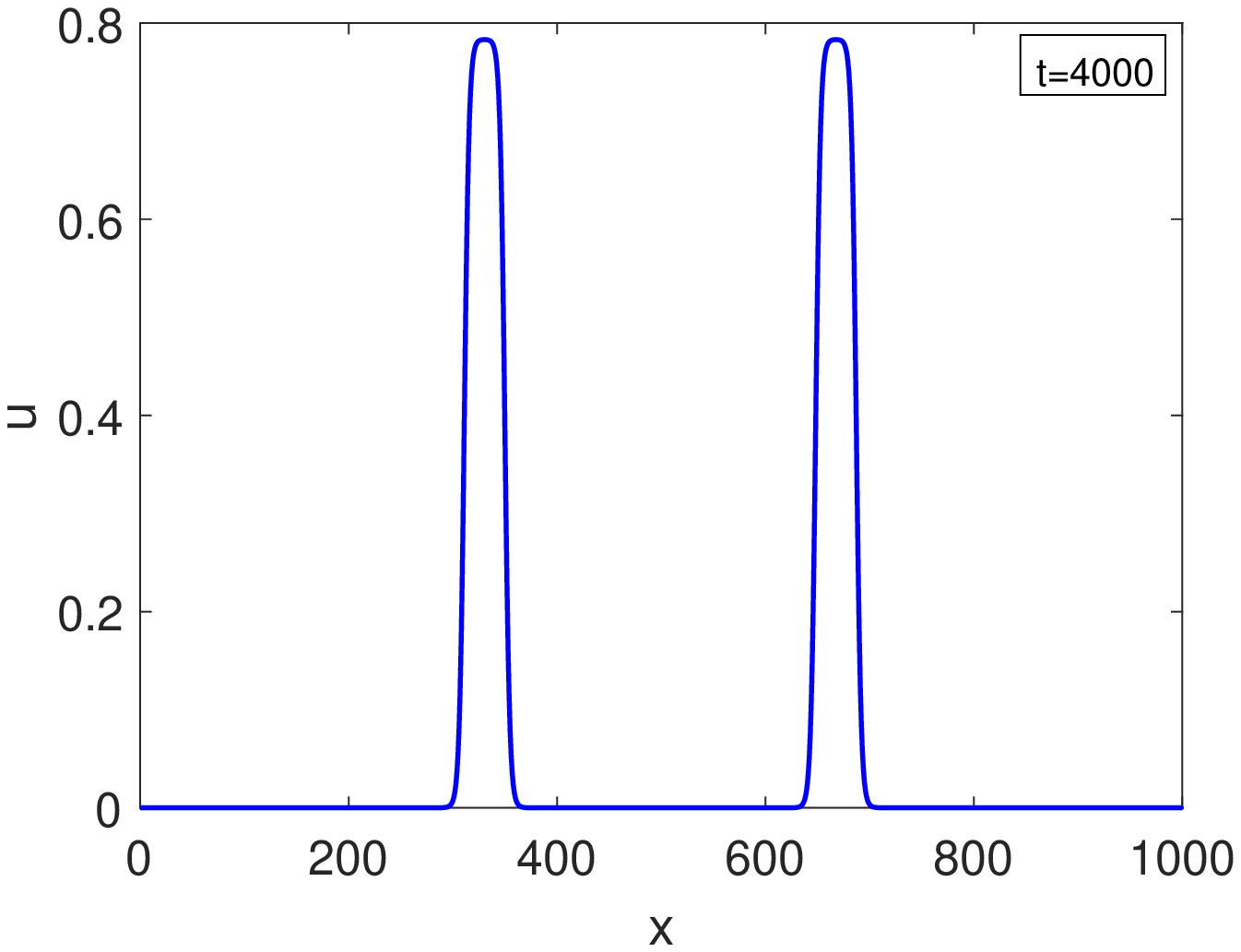}\\ 
  \caption{}
\label{mp4000}
\end{subfigure}
\caption{Multiple moving pule solution profiles at different times. Here  parameter values are $\sigma=0.6$, $d=20, \eta=0.1, \alpha=0.07, \beta=0.2$ and $\gamma=1.2.$ }
\label{mp}
\end{figure}
\subsection{Moving Pulse solution}
Now we take $\sigma<\sigma_{Het}$ for which the co-existing steady state is unstable. Suppose that the initial population distribution consists of co-existing state in a small region around the centre of the domain.  Introduction of a small amplitude spatial heterogeneity leads to moving pulse solution which shows intriguing complex dynamics.  We fix the parameter set to  $\eta=0.1,\alpha=0.07, \gamma=1.2$, $\sigma=0.6$, $L=1000$, $\beta=0.2$ and $d=20,$  and  consider the following initial distribution of the populations: 
\begin{equation}
    u(x,0)= \begin{cases}
  u_*+ \xi(x) & 495\leq x\leq 505 \\
 0 & \text{otherwise}
\end{cases}\;\mbox{and}\; v(x,0)= \begin{cases}
 v_*+ \xi(x) & 495\leq x\leq 505 \\
 0 & \text{otherwise}
\end{cases},\;
\end{equation}
where $\xi(x)$ is the Gaussian noise of amplitude $10^{-2}$.  Initially, the system \eqref{pde} produces a standing pulse solution consisting of a single island of nonzero populations surrounded by dead zones. The populations in the dead zones are almost extinct [see Fig. \ref{mp}(a)].  As time increases, the width of the island increases and splits into two islands as shown in Fig. \ref{mp}(b). These new islands are also moving towards the boundaries and they are symmetric about the centre $x=L/2.$ These two islands again split giving rise to four islands as shown in Fig. \ref{mp}(c). Through collision between  the adjacent islands, the system   either produces new islands or destroys some parent islands. The maximum possible number of islands in the considered domain is $16$. After the collision among the maximum number of islands, the number of islands reduces to two as shown in Fig. \ref{mp}(d) and the previous cycle repeats. The final spatio-temporal dynamics of the system is periodic in time with  time period $2558$ approximately.  The space-time plot of the prey population is shown in Fig. \ref{MP3}, which is a complex pattern consisting of many triangle shapes of invasion.   Such a pattern of many triangle shapes is often called the Sierpinski gasket pattern \cite{hayase2000self,kazantsev2003spiking}.

\textcolor{magenta}{}

\begin{figure}[!ht]
\centering
\includegraphics[scale=0.5]{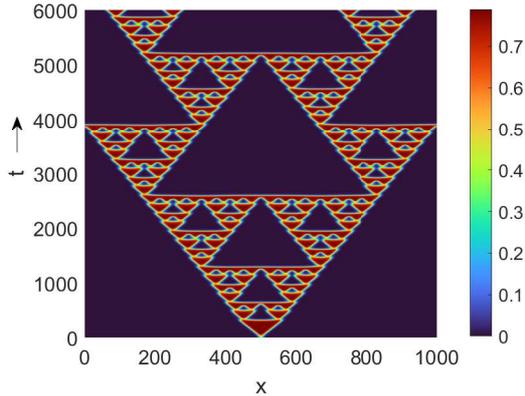}\\
\caption{Space-time plot of the prey population for moving pulse solution. Here parameter values are $\alpha=0.07,\beta=0.2, \gamma=1.2,\eta=0.1$, $\sigma=0.6$ and $d=20.$}
\label{MP3}
\end{figure}

\section{Conclusion}\label{discussion}
Our main objective of this work is to investigate the effect of the reproductive Allee effect in prey growth on population establishment for prey-predator type interaction model  with specialist predator. Parametrization of growth function and the choice of functional response are prominent in determining factors behind the dynamics produced by the interacting populations under consideration. In the absence of the Allee effect, the models with logistic prey growth, prey-dependent functional response, and linear mortality rate of the specialist predator  can not provide any indication of extinction of  both species when we consider the temporal dynamics \cite{courchamp2008allee}. The spatio-temporal extension of such models can capture the population's dynamic evolution over their habitats and predict localized extinction. These spatio-temporal models admit travelling waves, waves of invasion, and spatio-temporal chaos \cite{petrovskii1999minimal,banerjee2017spatio}. On the contrary, mutual interference among the predators can lead to localized establishment of both species through the generation of stationary spatial patterns depending upon Turing bifurcation. On the other hand, the temporal models with a strong Allee effect in prey growth indicate the possibility of extinction of  both species either due to  a global bifurcation or due to initial population density. Uncontrolled or accelerated grazing rate of a specialist predator on the only food source, which suffers from the Allee effect, drives the prey towards  localized or global extinction and hence system collapses. A spatio-temporal prey-predator model with Allee effect in prey growth, prey-dependent functional response, and linear mortality rate for specialist predator can support the establishment of both species through successful invasion \cite{morozov2006spatiotemporal}. However, a model with predator-dependent functional response and/or intra-specific competition among the predators supports stationary pattern formation as well as oscillatory or chaotic coexistence depending upon the parameter regime and the rates of diffusivity. 

Here we have presented preliminary stability and bifurcation results for the temporal model with the prey growth rate $\sigma$ as the bifurcation parameter. As per model formulation, the prey growth rate incorporates the mating success and it significantly contributes towards stable and oscillatory coexistence of both species. A decline in mating success is not only harmful to the specialist predators but also the initial abundance of the prey population can not save them from extinction. A decline in the value of $\sigma$ drives the system from stable steady-state coexistence to total extinction through an oscillatory coexistence (super-critical Hopf bifurcation) and the system collapses due to global heteroclinic bifurcation. It is interesting to note that the existence of Hopf-bifurcation within temporal setup is necessary for the Turing instability but not sufficient \cite{turing1}. RM type reaction kinetics supports supercritical Hopf bifurcation but one can't find the Turing pattern in a spatio-temporal model with RM reaction kinetics. 

The extinction scenario of populations depends not only on the Allee effect of the prey population but also on the functional response which characterizes any prey-predator type interaction. Appropriate parametrization of Beddington–DeAngelis functional response allows us to compare the  population establishment mechanism  with two different types of functional responses.  Beddington–DeAngelis functional response becomes ratio-dependent functional response for $\alpha=0,$ which shows some qualitative changes in the system dynamics for same set of other parameter values. Here, the destabilization of coexisting equilibrium point leads to the eventual extinction of both  populations through a subcritical Hopf bifurcation. Contrary to the Beddington–DeAngelis functional form, here global  bifurcation (homoclinic) doesn't indicate the extinction of both populations rather it depends upon  the initial distribution of the populations for their survival. In this case, the basin of attraction of the coexisting equilibrium point is bounded by a global bifurcation generated (unstable) limit cycle. The evolution of the prey  population with the variation of prey growth rate is discussed in Appendix \ref{Ap3}. However, neglecting  predator interference, i.e., in case of Holling type-II functional response, the mating success does not affect the  stability of  coexisting equilibria (see more details in Appendix \ref{Ap3}). 

Our systematic and through investigation of the spatio-temporal model reveals that the emergence of stationary and non-stationary patterns solely linked with the strength of species interaction, the rate of diffusivity and the species' abundance.  By constructing upper and lower solutions, the asymptotic behaviour of the solutions has been established under suitable parametric restrictions. The solution of the diffusive system is always bounded regardless of any parametric restriction, which proves the global existence of the solution.  If the maximum of the initial prey distribution is less than the reproductive Allee threshold $u_2,$ then both populations become extinct in the final state, which is analogous  to the temporal dynamics. Existence of non-constant steady state of the diffusive system has been established using Leray–Schauder degree theory. With the help of Poincare inequality, we have also shown the absence of non-constant steady state solution when the ratio of diffusivity  lies in a particular range. However, the existence and non-existence of heterogeneous solution somehow related to the domain size. Non-constant steady-state does not exist in small domain size.

There are two prominent  mechanisms behind the formation of non-constant steady state patterns: the Turing instability and the BD transition. We have found the Turing and BD transition thresholds by converting the steady-state equations into a system of four coupled ordinary differential equations. A variety of stable and unstable non-constant stationary solutions are found through numerical continuation for admissible range of parameter values. These solutions, which include various Turing mode solutions and localized solutions, emerge  through  Turing instablity \cite{dey2022analytical,hillen1996turing} and BD points \cite{champneys1998homoclinic,al2021unified}. We observe multi-stability between  various non-constant solutions along with the homogeneous steady-state solutions. The final stationary distribution of the population at large time limit depends on the initial distribution of both populations. Here we have shown that though the parameter setup satisfies the Turing instability condition but the numerical simulation reveals the settlement of population to localized pattern. In ecology, several prey population has these types of  patchy stationary distribution based on niche separation with a continuous habitat \cite{holmes1988food,fryxell2014wildlife}. Prey and predator populations exhibit identical patterns since patches with high prey density are favored by predators. There is evidence of this phenomenon in the spatio-temporal interactions of   toxic newts and arrow snakes in western North America  \cite{brodie2002evolutionary}. 


Introducing predator in a small domain leads to specific invasive pattern of predators distributed  over the entire domain which  corresponds to travelling wave solution. Generally, two types of travelling waves exist in  prey-predator models: one corresponds to  a heteroclinic trajectory connecting two homogeneous steady-states and the other corresponds to  periodic travelling wave that involve limit cycle surrounding an unstable homogeneous state. Our model admits monotone and non-monotone travelling waves which are  established by showing a heteroclinic connection between two equilibria of the corresponding system of four  first-order ordinary differential equations. The boundedness of the heteroclinic connection is established inside a wedge-shaped region \cite{huang2016geometric}. The parametric regions for the existence of monotonic  and  non-monotonic travelling waves and non-existence of travelling wave  have been obtained with the help of exhaustive numerical simulations. We have also validated the theoretical results  with numerical simulations and discussed the invasion profile over the parametric region. For a   mobility rate of prey individuals close to predator, the temporal extinction scenario alters to localized extinction and regeneration of patches at nearby locations,  resulting in spatio-temporal chaotic or moving pulse solution. For lower values of $d$ in the Hopf region and some restricted prey growth rate  $(\sigma_{Het}<\sigma<\sigma_H)$, the spatio-temporal  system exhibits time aperiodic and non-homogeneous in space spatio-temporal chaos. Further decrease in $\sigma$ gives rise to multiple moving pulse solution forming Sierpinski gasket pattern. In real world, such types of spatio-temporal regular and irregular oscillations have been reported for some prey-predator type interactions that include interaction between Daphnia and Bythotrephes \cite{lehman1993food} and interaction between  tephritid flies and thistle population \cite{jeltsch1992oscillating}.

The novelty of this work lies in the identification of Turing bifurcation along with the  BD transition which leads to the formation of local patterns through transient Turing-like patterns.  Surprisingly, the Turing patterns appeared to be transitory patterns and both the species get established at large stationary patches through BD transition mechanism. This combination of Turing instability and BD transition is an addition to the list of known mechanisms behind long transients in spatio-temporal pattern formation \cite{morozov2020long}.  The choice of parametrization for the Beddington-Deangelis functional response helps us to conclude that  this localized pattern formation scenario is influenced by the reproductive Allee effect and obtained results can be verified with other two  functional responses namely Holling type-II and ratio-dependent functional response. It is well-known that systems with a prey-dependent functional response and linear death rate,  do not produce stationary Turing patterns. But if we choose a slightly different parameter setting for which  the temporal system with Holling type-II functional response has a stable coexisting equilibrium, then the corresponding spatio-temporal model supports the formation of localized stationary patterns due to BD transition. A few localized  patterns for different values of the parameter $d$ are shown in Appendix \ref{Ap3}. Instead of the reproductive Allee effect, if we consider only the logistic growth of the prey population, then the system fails to  produce any localized pattern. Thus, we  conclude that the reproductive Allee effect plays a central role in the formation of localized patterns.


\section*{Declarations}

\textbf{Conflict of interest}: The authors declare that they have no conflict of interest.\\[1em]
\textbf{Data availability statement}: The authors declare that the manuscript has no associated data.\\
\begin{appendices}
\section{   }\label{Ap1}
Here we derive upper bounds for $I_1$, $I_2$ and $I_3$ used in Theorem \ref{prop6}. Using proposition  \ref{prop5}, we find
\begin{alignat*}{4}
  I_1=& \int _\Omega (u-\bar u)^2 \big(-(u^2+u \bar u+\bar u^2)+(u+\bar u)(u_1+u_2)-u_1u_2 \big) dx\\
  \leq&\int _\Omega (u-\bar u)^2 \left[(u+\bar u)(u_1+u_2)-u_1u_2 \right]dx\\
  \leq& u_1(2-u_2) \int _\Omega (u-\bar u)^2 dx,\end{alignat*}
\begin{alignat*}{4}
\allowdisplaybreaks
  I_2=& \frac{1}{\alpha+\bar u+\beta\bar  v}\int _\Omega (u-\bar u)\left(\frac{\alpha (\bar u \bar v-uv)+\beta v \bar v(\bar u-u)+u\bar u (\bar v-v)}{\alpha+u+\beta v}\right)dx\\=&
  \frac{1}{\alpha+\bar u+\beta\bar  v}\int _\Omega (u-\bar u)\left(\frac{ \bar v(\bar u-u)(\alpha+\beta v)+u (\bar v-v) (\alpha + \bar u)}{\alpha+u+\beta v}\right)dx  \\
  \leq&  \frac{(\alpha + \bar u)}{\alpha+\bar u+\beta\bar  v}\int _\Omega (u-\bar u)\frac{ u (\bar v-v) }{\alpha+u+\beta v}dx\\
   \leq&  \frac{u_1}{\alpha}\int _\Omega \vert (u-\bar u)\,(v-\bar v) \vert dx\\
   \leq&  \frac{u_1}{2\alpha}\int _\Omega (u-\bar u)^2dx+\frac{u_1}{2\alpha}\int _\Omega (v-\bar v)^2dx,\end{alignat*}
and \begin{alignat*}{4}
\allowdisplaybreaks
  I_3=& \int _\Omega (v-\bar v)\bar v \left(\frac{\gamma u}{\alpha+u+\beta v}-1\right)dx\\
  =&\int _\Omega (v-\bar v)\bar v \left(\frac{\gamma u}{\alpha+u+\beta v}-\frac{\gamma \bar u}{\alpha+\bar u+\beta \bar v}\right)dx\\
   =&\frac{\gamma \bar v}{\alpha+\bar u+\beta\bar  v}\int _\Omega (v-\bar v) \frac{\alpha(u-\bar u)+\beta(u\bar v-\bar u v)}{\alpha+u+\beta v} \,dx\\
   = & \frac{\gamma \bar v}{\alpha+\bar u+\beta\bar  v}\int _\Omega (v-\bar v) \frac{\beta u(\bar v-v)+(u-\bar u)(\alpha + \beta v)}{\alpha+u+\beta v}\, dx\\
   \leq & \frac{\gamma \bar v}{\alpha+\bar u+\beta\bar  v}\int _\Omega (v-\bar v) \frac{(u-\bar u)(\alpha +\beta v)}{\alpha+u+\beta v} \,dx\\
   \leq & \frac{\gamma}{\beta}\int _\Omega (v-\bar v) (u-\bar u)\, dx \\
   \leq & \frac{\gamma}{2 \beta}\int _\Omega (u-\bar u)^2dx+\frac{\gamma}{2 \beta}\int _\Omega (v-\bar v)^2dx.
\end{alignat*}
\section{ } \label{Ap2}
Here we show 19-, 20- and 30- mode Turing solutions in Fig. \ref{turingg} 
and few localized solutions in Fig. \ref{localized}.

\begin{figure}[H]
\begin{subfigure}[b]{.32\textwidth}
  \centering
  \centering
\includegraphics[scale=0.38]{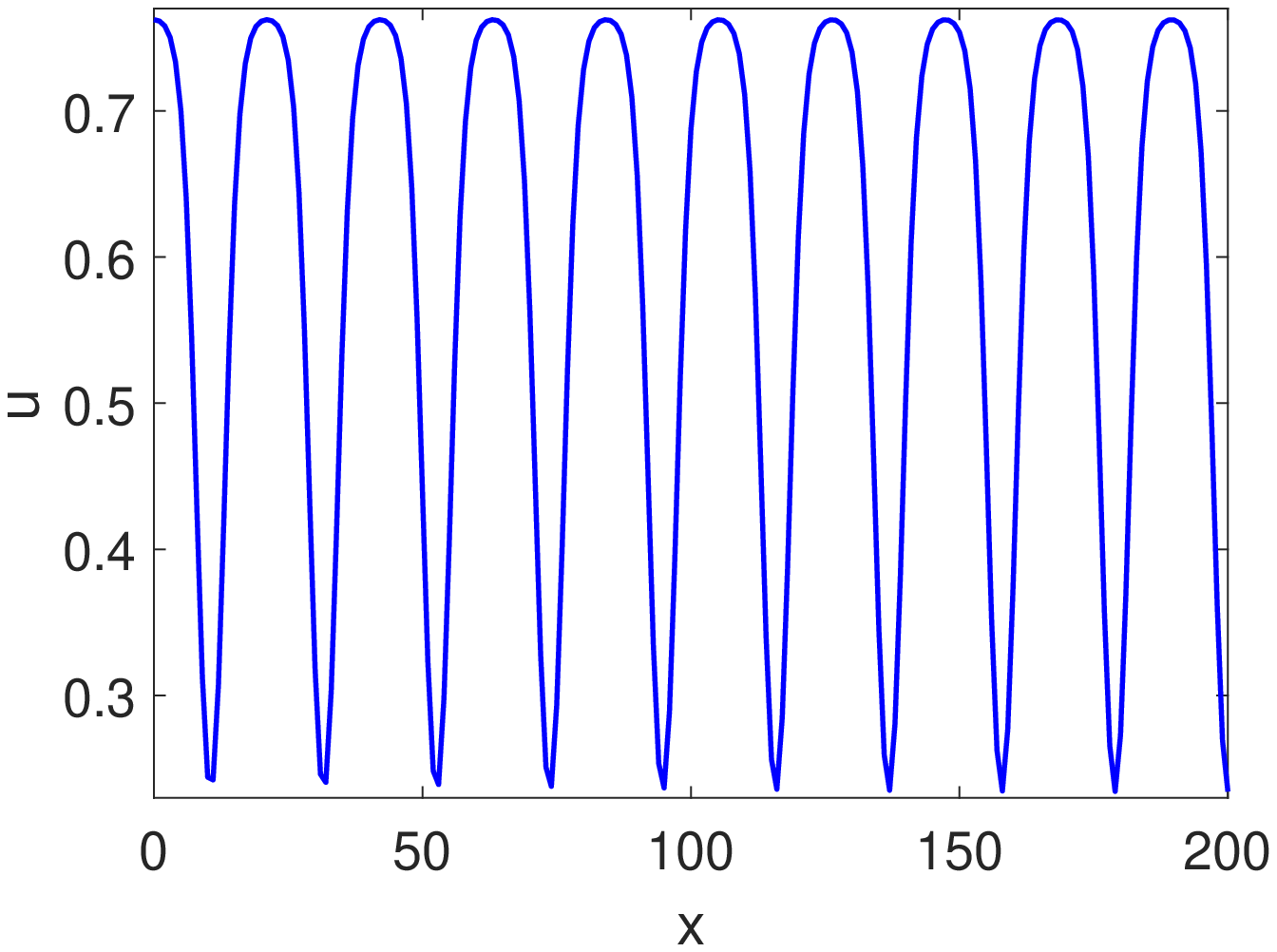}\\ 
\caption{}
\end{subfigure}
\begin{subfigure}[b]{.32\textwidth}
   \centering
\includegraphics[scale=0.38]{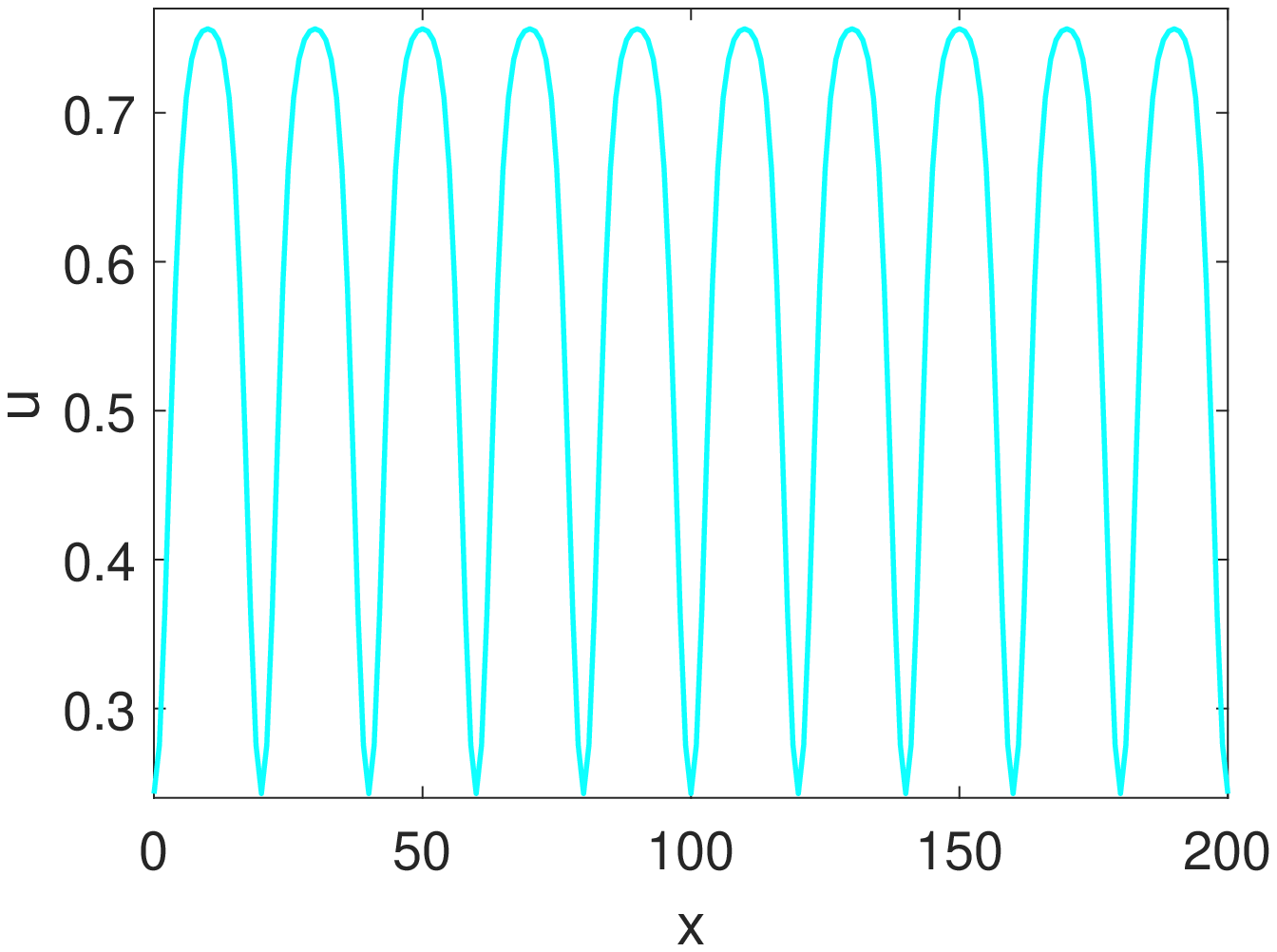}\\ 
\caption{}
\end{subfigure}
\begin{subfigure}[b]{.32\textwidth}
   \centering
\includegraphics[scale=0.38]{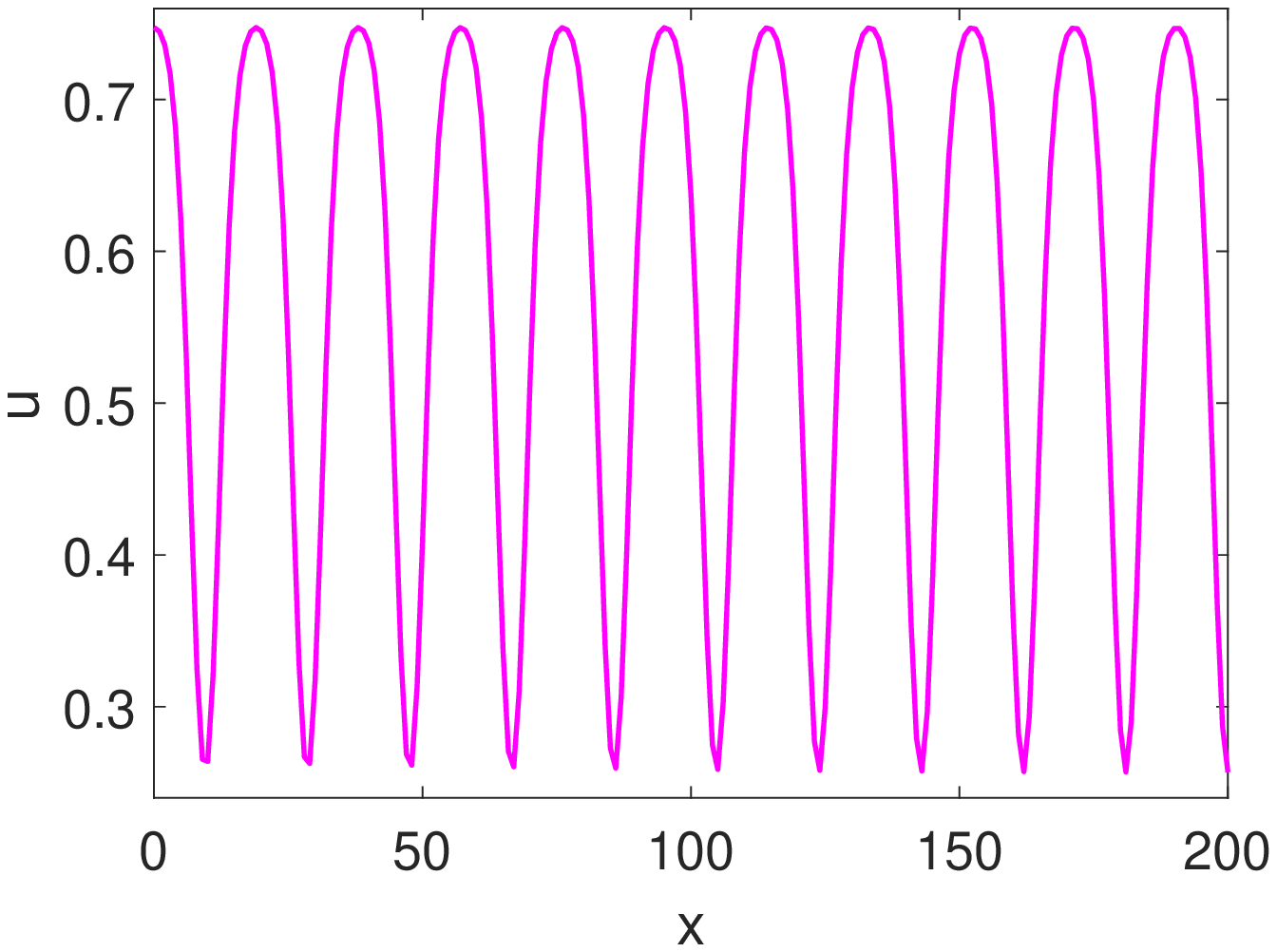}\\ 
\caption{}
\end{subfigure}
\caption{Turing pattern solution  of various modes:  (a) 19- mode solution for $\sigma=1.939$, (b) 20- mode solution  for $\sigma=1.867$, (c) 21- mode solution  for $\sigma=1.812$. The color of a solution curve corresponds to the same color diamond point shown in Fig. \ref{tur}(b).   Other parameter values are $\alpha=0.07,\beta=0.2, \gamma=1.2, \eta=0.1$ and $d=46.$}
\label{turingg}
\end{figure}

\begin{figure}[H]
\begin{subfigure}[b]{.32\textwidth}
  \centering
  \centering
\includegraphics[scale=0.38]{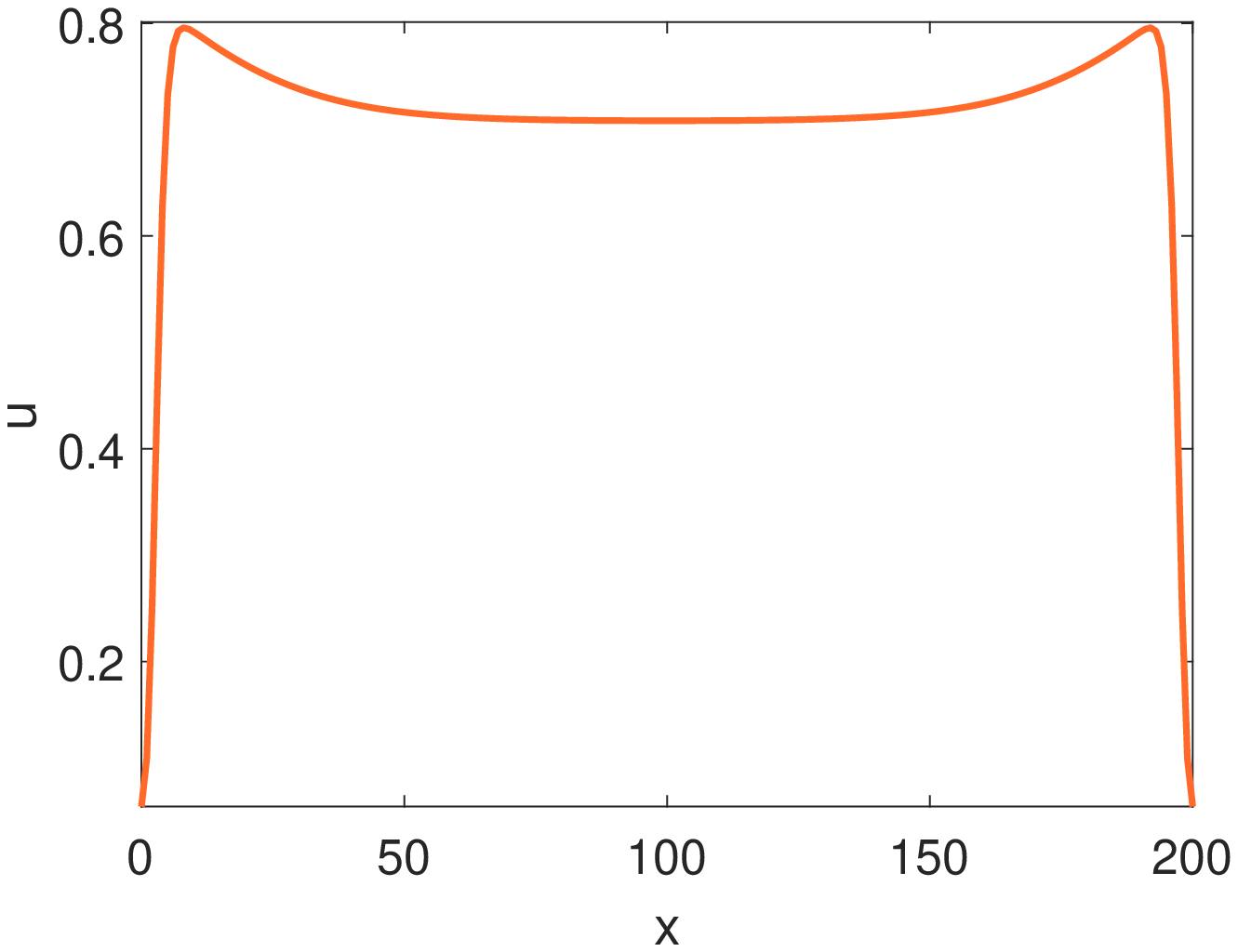}\\ 
\caption{}
\end{subfigure}
\begin{subfigure}[b]{.32\textwidth}
   \centering
\includegraphics[scale=0.38]{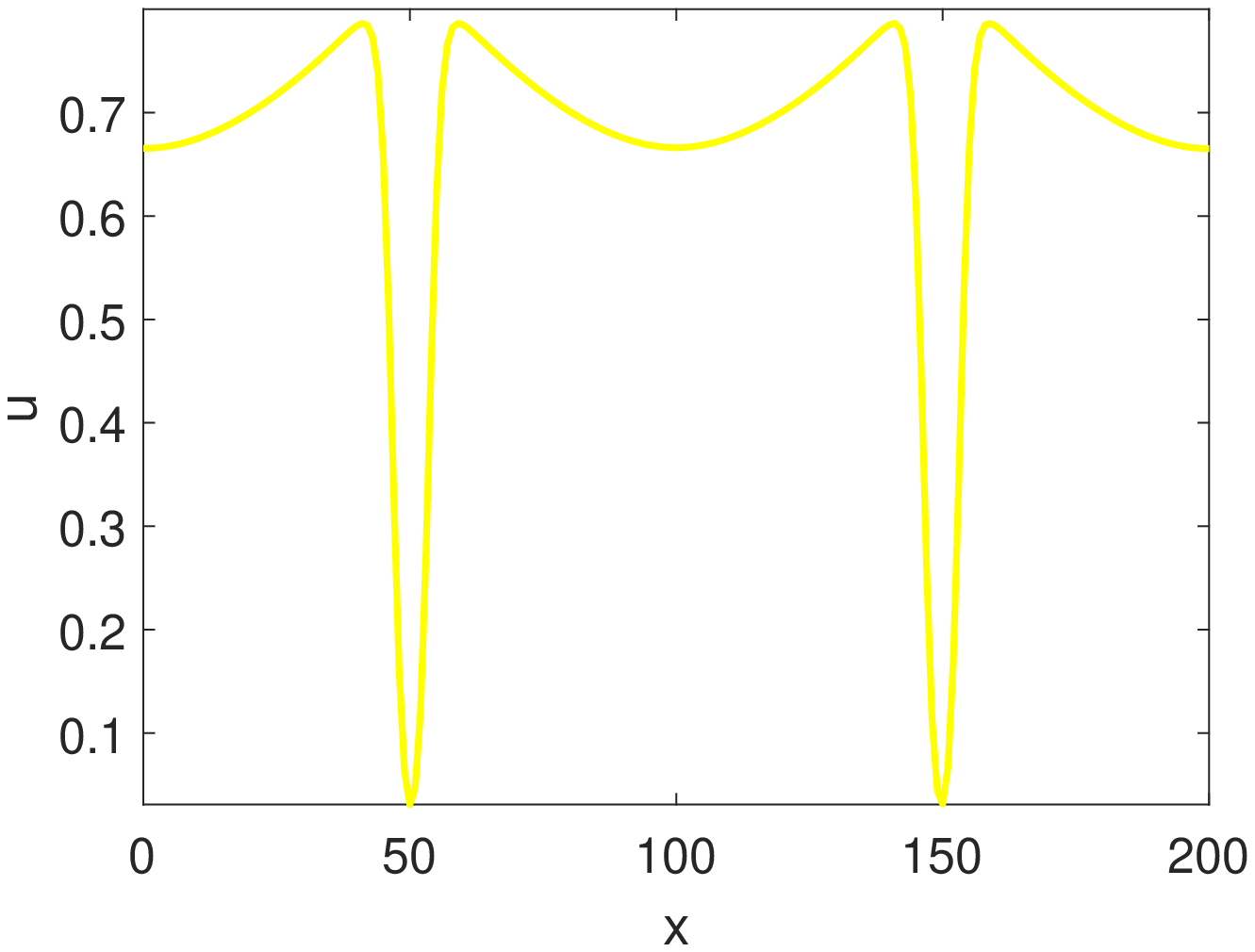}\\ 
\caption{}
\end{subfigure}
\begin{subfigure}[b]{.32\textwidth}
   \centering
\includegraphics[scale=0.38]{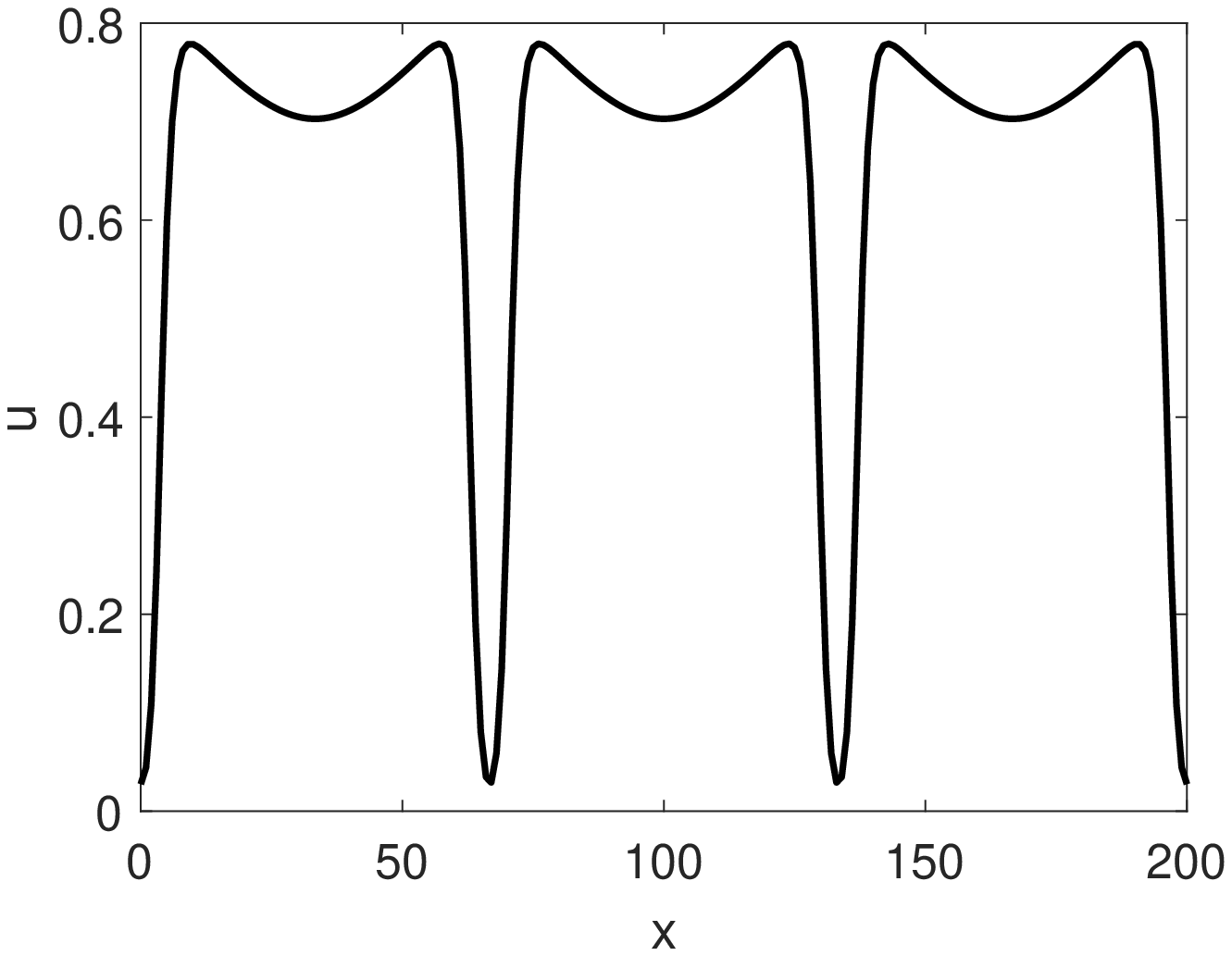}\\ 
\caption{}
\end{subfigure}
\caption{Localized pattern solutions where the colour of a solution curve corresponds to the same color point shown in Fig. \ref{tur}(c). Other parameter values are $\alpha=0.07,\beta=0.2, \gamma=1.2, \sigma=2.36,\eta=0.1$ and $d=46.$
}
\label{localized}
\end{figure}
\section{ } \label{Ap3}
\begin{figure}[!h]
 \begin{subfigure}[b]{.48\textwidth}
 \centering
\includegraphics[scale=0.5]{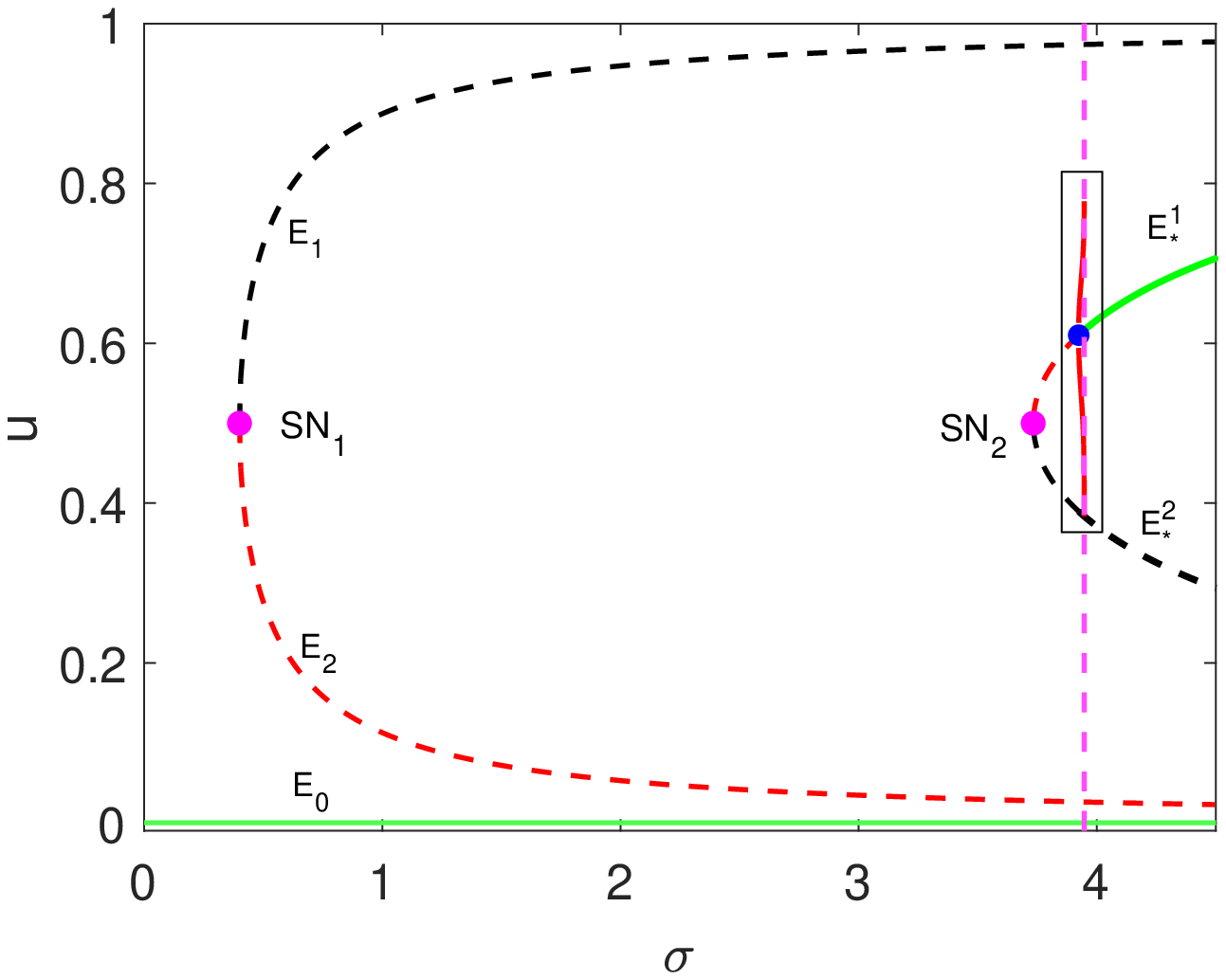}\\ 
 \caption{}
 \label{ratio1}
 \end{subfigure}
 \begin{subfigure}[b]{.48\textwidth}
   \centering
\includegraphics[scale=0.5]{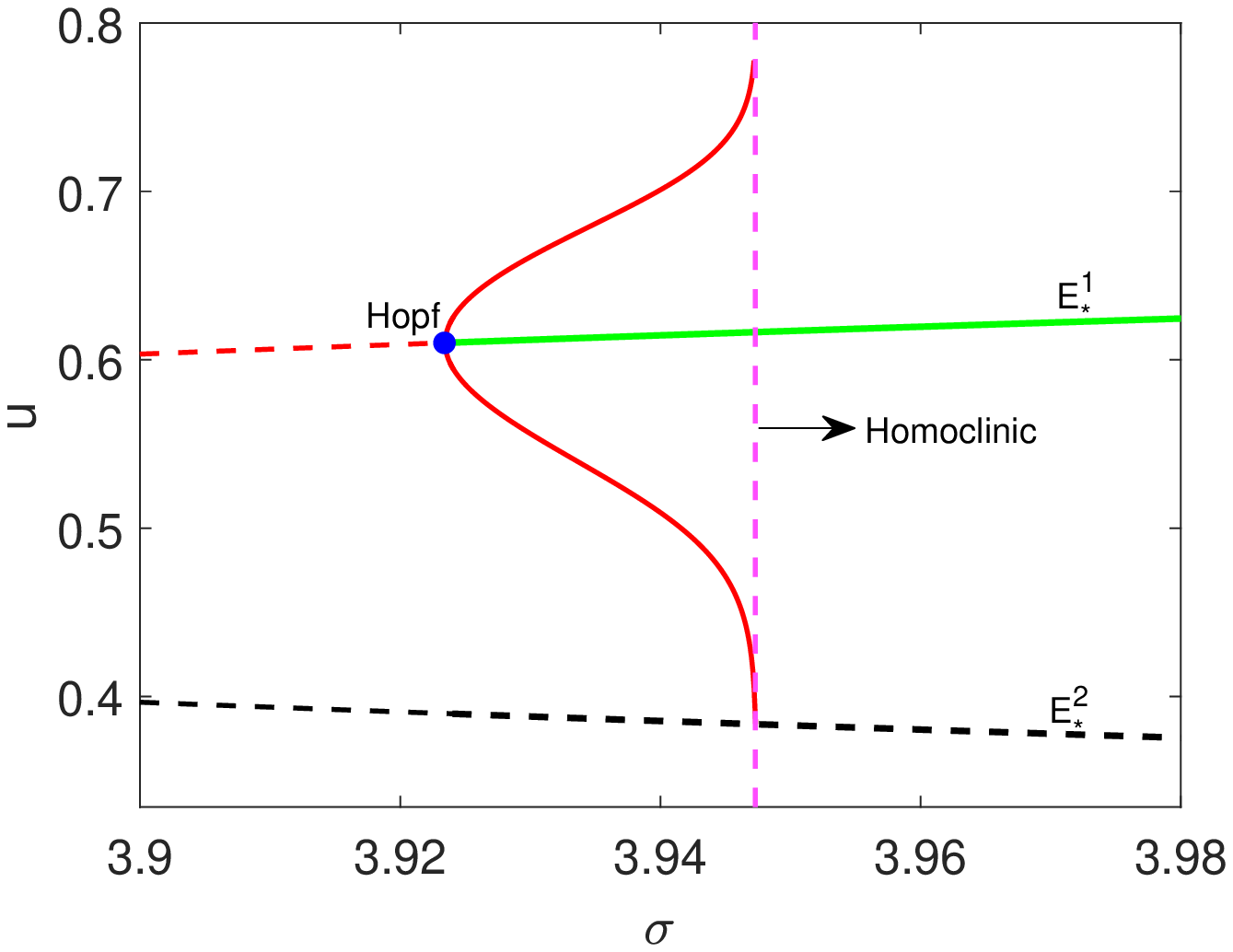}\\ 
 \caption{}
 \end{subfigure}
 \caption{(a) Bifurcation diagram in the $\sigma$-$u$ plane for the ratio-dependent functional response. (b) A zoom version of Fig. (a) for $\sigma\in [3.9,3.98]$. The point marked SN denotes the saddle-node bifurcation threshold. Here, solid green, red dashed and black dashed curves represent stable, unstable  and saddle branches of equilibria respectively. Further, solid red color curve represents the maximum and minimum of $u$ of the unstable limit cycle. Other parameter values are $\alpha=0,\beta=0.2, \gamma=1.2$ and $\eta=0.1.$} 
\label{Rdbif}
\end{figure}
Here we compare our findings with other functional responses. If we assume $\alpha=0$, then the resulting functional response is called a ratio-dependent functional response. The temporal dynamics of the system, keeping all other parameters the same with $\alpha=0$, changes  significantly. Here, two coexisting equilibria $E_*^1$ and $E_*^2$ emerge through a saddle-node bifurcation, where $E_*^2$ is always a saddle-node and the stability of $E_*^1$ depends on the Hopf bifurcation threshold $\sigma_H.$ An unstable limit cycle is generated due to subcritical Hopf bifurcation at $\sigma=\sigma_H$ which vanishes due to a global homoclinic bifurcation. The temporal dynamics is summarized in the one-parameter bifurcation diagram shown in Fig. \ref{Rdbif}. The corresponding spatio-temporal model shows stationary Turing  and localized patterns. It also exhibits dynamic patterns that include multiple  moving pulse solution, spatio-temporal chaos, and traveling wave.

However, the chosen functional response becomes  Holling type-II functional response for $\beta=0$. Then the system can have at most one coexisting equilibrium point $E_*$, whose  feasibility comes from a transcritical bifurcation. However, the stability of $E_*$ is independent of the parameter $\sigma.$ If we take the parameter $\gamma=1.11,$ then the unique coexisting equilibrium point $E_*$ is asymptotically stable (whenever it exists) for all value of $\sigma$. Although the system does not produce Turing pattern but it exhibits localized pattern. We have shown few localized patterns for different values of diffusion parameter $d$ in Fig. \ref{Holling}. Apart from the stationary pattern, the system also shows dynamic patterns that include  multiple  moving pulse solution, spatio-temporal chaos and travelling wave solution.
\begin{figure}[!h]
\begin{subfigure}[b]{.32\textwidth}
  \centering
  \centering
\includegraphics[scale=0.38]{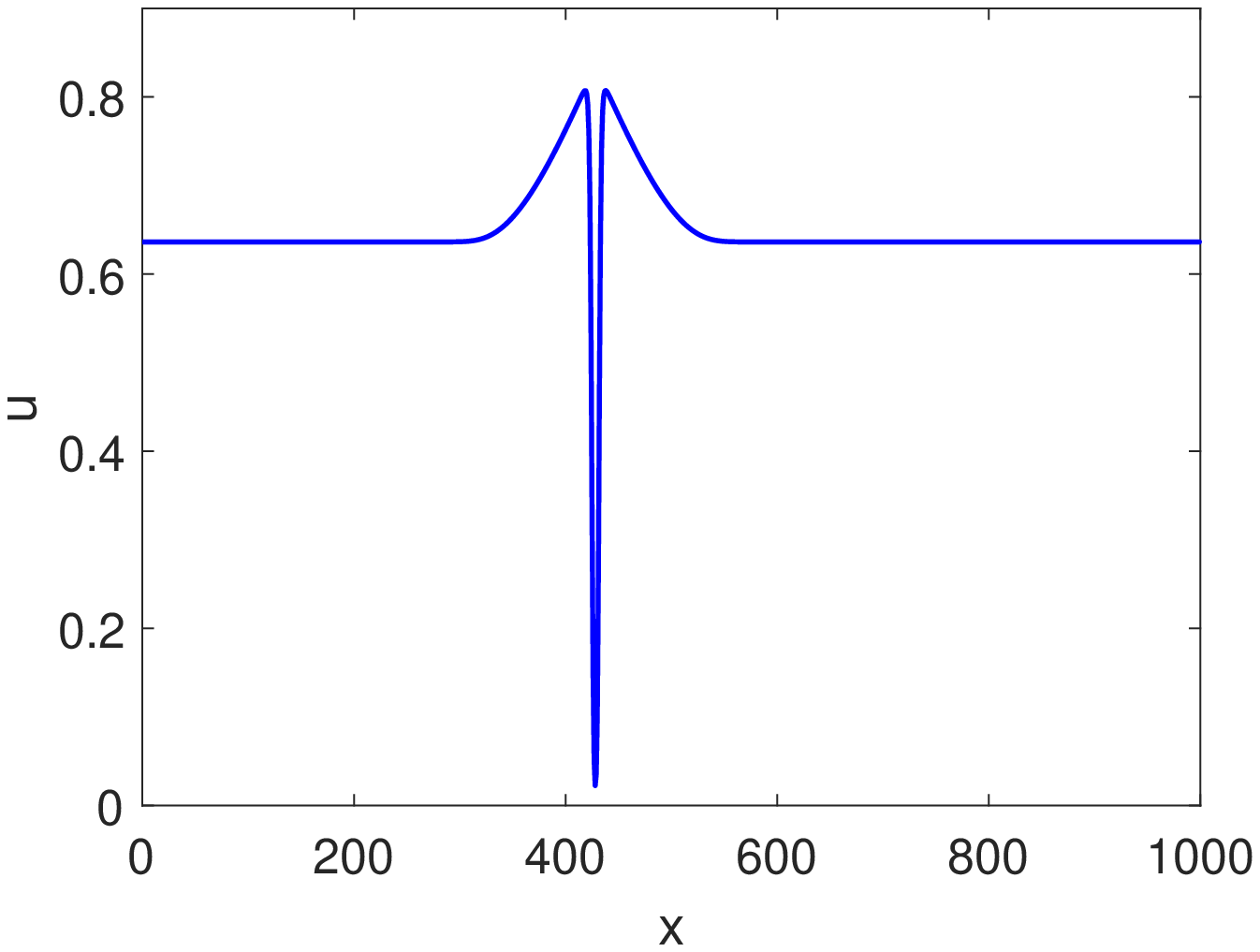}\\ 
\caption{}
\end{subfigure}
\begin{subfigure}[b]{.32\textwidth}
  \centering
  \centering
\includegraphics[scale=0.38]{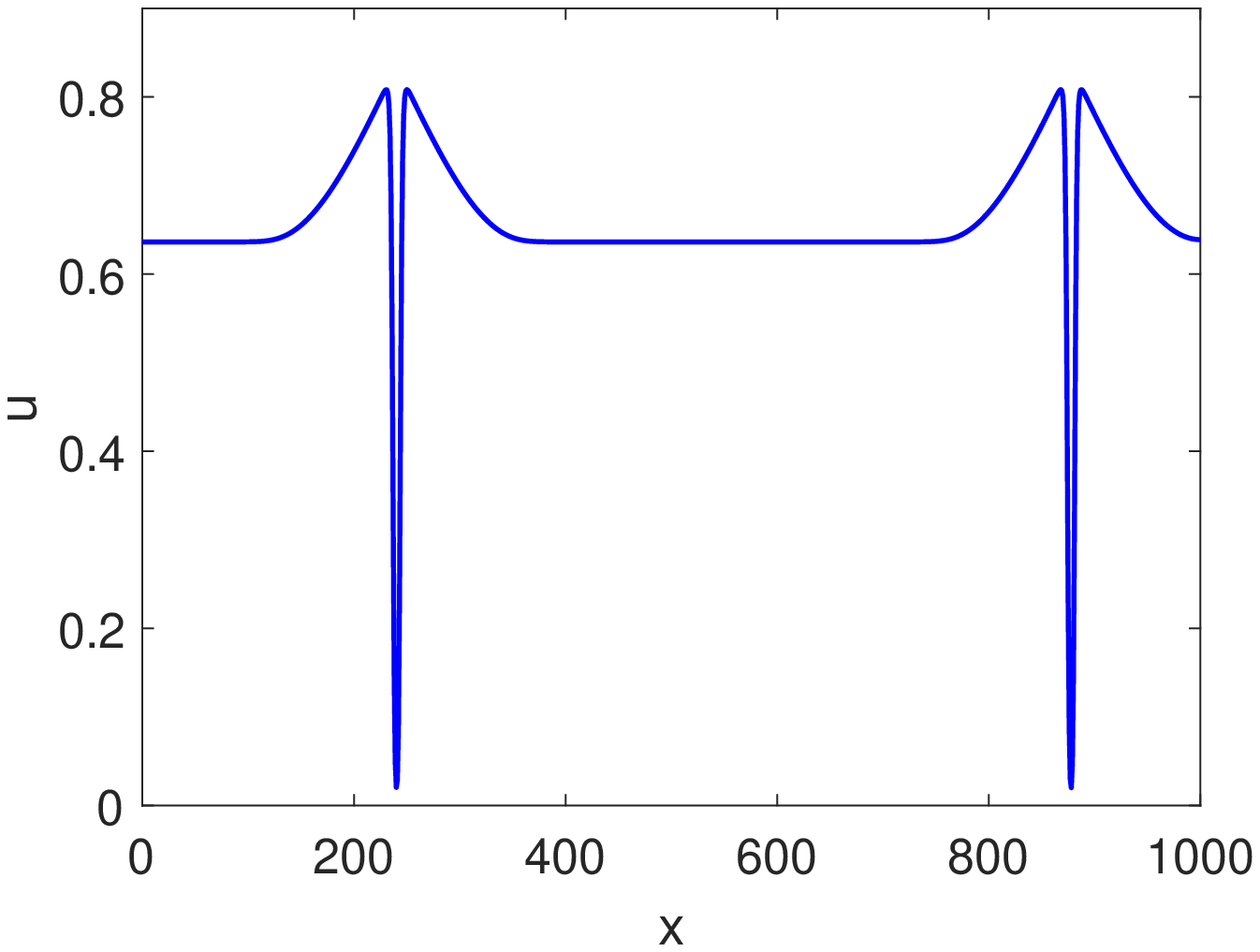}\\ 
\caption{}
\end{subfigure}
\begin{subfigure}[b]{.32\textwidth}
  \centering
  \centering
\includegraphics[scale=0.38]{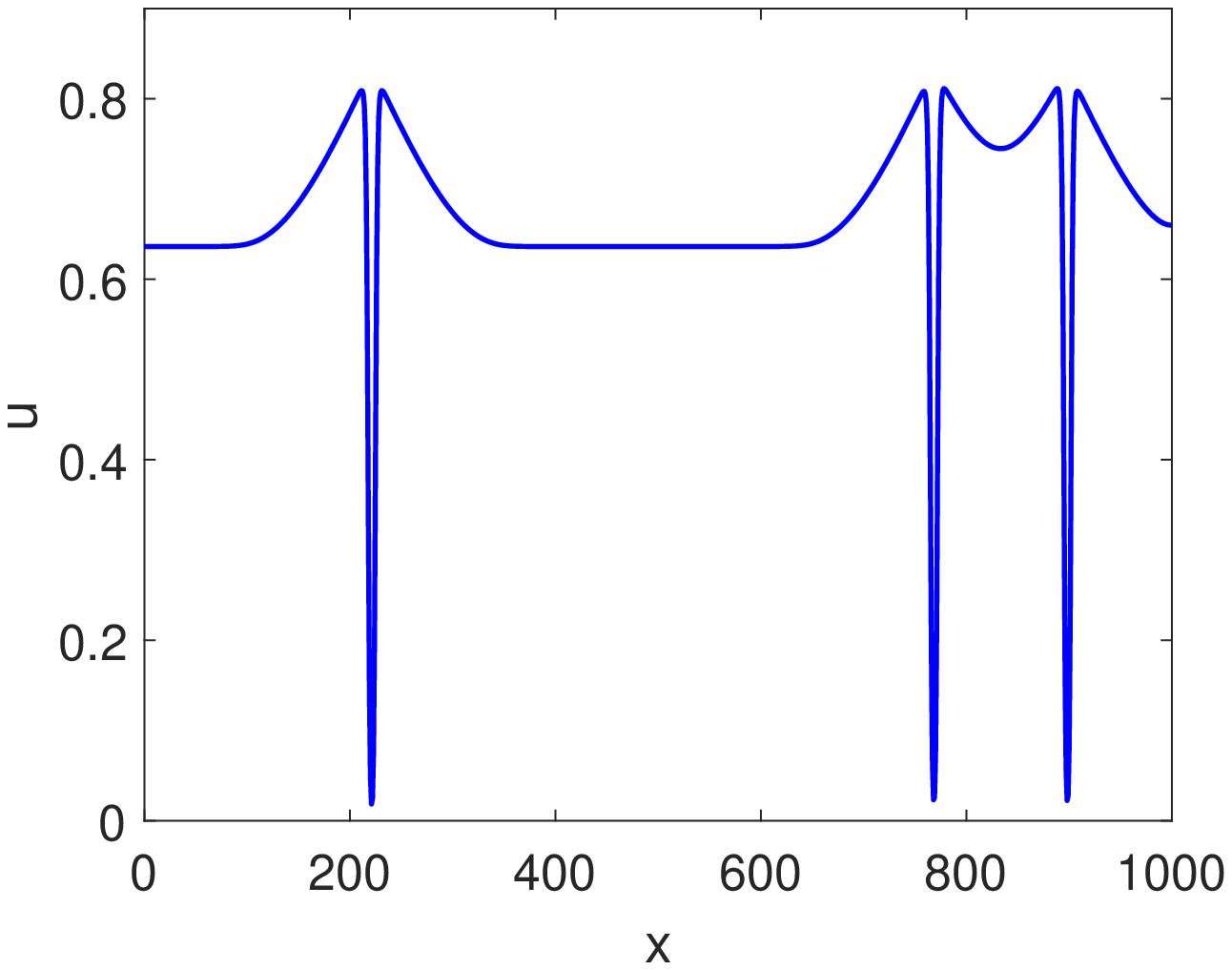}\\ 
\caption{}
\end{subfigure}
\caption{Localized patterns for the system \eqref{pde} with Holling type-II functional response: (a) $d=90$ (b)$d=100$ (c) $d=110$. Other parameter values are $\sigma=2,\gamma=1.11,\eta=0.1,\alpha=0.07$ and $\beta=0.$}
\label{Holling}
\end{figure}




\end{appendices}


\end{document}